\documentclass[10pt]{amsart}
\usepackage{amscd,amssymb,amsmath,amsxtra}
\usepackage[margin=3cm]{geometry}
\usepackage{graphicx}
\usepackage[curve]{xypic}

\makeatletter
\let\@wraptoccontribs\wraptoccontribs
\makeatother
\usepackage{amsmath}
\usepackage{epsfig}
\usepackage{graphicx}
\usepackage{eufrak}
\usepackage{dsfont}
\usepackage[english]{babel}
\usepackage{color}
\usepackage{mathabx}

\usepackage{palatino, mathpazo}
\usepackage{amscd}      
\usepackage{amssymb}
\usepackage{dsfont}
\usepackage{xypic}      
\LaTeXdiagrams          
\usepackage[all,v2]{xy}
\xyoption{2cell} \UseAllTwocells \xyoption{frame} \CompileMatrices
\usepackage{latexsym}
\usepackage{epsfig}
\usepackage{enumerate}

\usepackage{latexsym}
\usepackage{epsfig}
\usepackage{amsfonts}
\usepackage{enumerate}
\usepackage{mathrsfs}

\allowdisplaybreaks[3]

\newtheorem{prop}{Proposition}[section]
\newtheorem{lem}[prop]{Lemma}

\newtheorem{cor}[prop]{Corollary}
\newtheorem{thm}[prop]{Theorem}
\newtheorem{rmk}[prop]{Remark}
\newtheorem{example}{Example}

\newtheorem{defn}[prop]{Definition}

\newtheorem{con}[prop]{Conjecture}

            {\nolinebreak $\Box$ \end{trivlist}}

\newcommand{\noprint}[1]{}

\newcommand{\Ext}{\mbox{Ext}}

\newcommand{\Hom}{\mbox{Hom}}

\newcommand{\topo}{\mbox{\tiny top}}

\newcommand{\tw}{\mbox{\tiny tw}}

\newcommand{\E}{\mathop{\sf E}\nolimits}
\newcommand{\sfH}{\mathop{\sf H}\nolimits}

\newcommand{\N}{\mathcal{N}}

\newcommand{\XX}{{\mathfrak X}}

\renewcommand{\SS}{{\mathfrak S}}

\newcommand{\YY}{{\mathfrak Y}}

\newcommand{\ZZ}{{\mathfrak Z}}

\newcommand{\Tt}{{\mathfrak t}}

\newcommand{\zz}{{\mathbb Z}}
\newcommand{\hh}{{\mathbb H}}

\newcommand{\T}{{\mathbb T}}

\newcommand{\aaa}{{\mathbb A}}

\newcommand{\nn}{{\mathbb N}}

\renewcommand{\ll}{{\mathbb L}}
\newcommand{\qq}{{\mathbb Q}}
\newcommand{\pp}{{\mathbb P}}
\newcommand{\cc}{{\mathbb C}}

\newcommand{\Gm}{{{\mathbb G}_{\mbox{\tiny\rm m}}}}

\newcommand{\sE}{{\mathcal E}}
\newcommand{\sI}{{\mathcal I}}

\newcommand{\sL}{{\mathcal L}}

\newcommand{\sS}{{\mathcal S}}

\newcommand{\sP}{{\mathcal P}}

\newcommand{\sO}{{\mathcal O}}

\newcommand{\sX}{{\mathcal X}}

\newcommand{\sM}{{\mathcal M}}

\newcommand{\sF}{{\mathcal F}}
\newcommand{\sA}{{\mathcal A}}

\newcommand{\Coh}{\mbox{Coh}}

\newcommand{\rE}{\mathscr{E}}
\newcommand{\rM}{\mathscr{M}}
\newcommand{\sB}{\mathscr{B}}

\newcommand{\cHom}{\mathscr{H}om}

\DeclareMathOperator{\id}{id}

\DeclareMathOperator{\St}{St}
\DeclareMathOperator{\Sch}{Sch}

\DeclareMathOperator{\Hilb}{Hilb}

\DeclareMathOperator{\Aut}{Aut}

\DeclareMathOperator{\ind}{ind}
\DeclareMathOperator{\Ob}{Ob}
\DeclareMathOperator{\VW}{VW}
\DeclareMathOperator{\vw}{vw}
\DeclareMathOperator{\At}{At}
\DeclareMathOperator{\Cone}{Cone}
\DeclareMathOperator{\vir}{vir}
\DeclareMathOperator{\mov}{mov}
\DeclareMathOperator{\Pic}{Pic}
\DeclareMathOperator{\vd}{vd}
\DeclareMathOperator{\Ch}{Ch}
\DeclareMathOperator{\CR}{CR}

\DeclareMathOperator{\Higg}{Higg}

\DeclareMathOperator{\vb}{vb}
\DeclareMathOperator{\odd}{odd}
\DeclareMathOperator{\even}{even}
\DeclareMathOperator{\fppf}{fppf}
\DeclareMathOperator{\Sh}{Sh}
\DeclareMathOperator{\Tot}{Tot}
\DeclareMathOperator{\Br}{Br}
\DeclareMathOperator{\per}{per}

\DeclareMathOperator{\Td}{Td}
\DeclareMathOperator{\JS}{JS}

\DeclareMathOperator{\NS}{NS}
\DeclareMathOperator{\Num}{Num}
\DeclareMathOperator{\ess}{ess}
\DeclareMathOperator{\opt}{opt}

\DeclareMathOperator{\SU}{SU}
\DeclareMathOperator{\SL}{SL}
\DeclareMathOperator{\PGL}{PGL}

\newcommand{\rk}{\mathop{\rm rk}}

\newcommand{\ev}{\mathop{\rm ev}\nolimits}
\newcommand{\td}{\mathop{\rm td}}

\newcommand{\tr}{\mathop{\rm tr}\nolimits}

\renewcommand{\Im}{\mathop{\rm Im}}

\newcommand{\rank}{\mathop{\rm rank}\nolimits}

\newcommand{\red}{\mathop{\rm red}\nolimits}
\newcommand{\supp}{\mathop{\rm supp}}

\newcommand{\Jac}{\mathop{\rm Jac}\nolimits}

\newcommand{\ob}{\mathop{\rm ob}}

\newcommand{\spec}{\mathop{\rm Spec}\nolimits}

\newcommand{\Sym}{\mathop{\rm Sym}\nolimits}
\newcommand{\proj}{\mathop{\rm Proj}\nolimits}
\newcommand{\tor}{\mathop{\rm tor}\nolimits}

\numberwithin{equation}{subsection}

\newcommand {\mat}      [1] {\left(\begin{array}{#1}}
\newcommand {\rix}          {\end{array}\right)}

\title[Counting twisted sheaves and  S-duality]{Counting twisted sheaves and S-duality}
\author{Yunfeng Jiang}
\address{Department of Mathematics\\ University of Kansas\\ 405 Snow Hall 1460 Jayhawk Blvd\\Lawrence KS 66045 USA} 
\email{y.jiang@ku.edu}

\begin{document}
\sloppy \maketitle
\begin{abstract}
We provide a definition of Tanaka-Thomas's Vafa-Witten invariants for \'etale gerbes over smooth projective surfaces using the moduli spaces  of $\mu_r$-gerbe twisted  sheaves and Higgs sheaves.  Twisted sheaves and their moduli are naturally used to study the period-index theorem for the corresponding  $\mu_r$-gerbe in the Brauer group of the surfaces. 
Deformation and obstruction theory of the twisted  sheaves and Higgs sheaves behave like general sheaves and Higgs sheaves. We define virtual fundamental  classes on the moduli spaces and define the twisted  Vafa-Witten invariants using virtual localization and the Behrend function on the moduli spaces. 
As applications for the  Langlands dual group $\SU(r)/\zz_r$ of $\SU(r)$, we define the $\SU(r)/\zz_r$-Vafa-Witten  invariants using the twisted  invariants for \'etale gerbes, and prove the S-duality conjecture of Vafa-Witten for the projective plane in rank two and for K3 surfaces in prime ranks.  We also conjecture for other surfaces.
\end{abstract}

\maketitle


\section{Introduction}

The Langlands dual gauge group of  the gauge group $\SU(r)$ is $\SU(r)/\zz_r$. The main goal of  this paper is to provide a method to define the  $\SU(r)/\zz_r$-Vafa-Witten invariants and check the S-duality conjecture of Vafa-Witten in \cite{VW}.  We study and define the twisted Vafa-Witten invariants for $\mu_r$-gerbes over smooth projective surfaces by using  the moduli stack of $\mu_r$-gerbe twisted stable sheaves in \cite{Lieblich_Duke} and Higgs sheaves.   We mainly follow the idea of  Tanaka-Thomas in \cite{TT1}, \cite{TT2} to define the twisted Vafa-Witten invariants.   We conjecture that the twisted Vafa-Witten invariants for  all $\mu_r$-gerbes $\SS\to S$ on a surface $S$ give rise to the $^{L}\SU(r)=\SU(r)/\zz_{r}$-Vafa-Witten invariants.  We prove the conjecture for $\pp^2$ in rank two; and K3 surfaces in all the prime ranks. 

In the case of K3 surfaces, one of  the novel discoveries  is that the essentially trivial $\mu_r$-gerbes and optimal $\mu_r$-gerbes on a K3 surface $S$
give different twisted Vafa-Witten invariants. 
These different type of $\mu_r$-gerbes depend on the Picard number of the surface $S$. 
In the rank two case,  for a K3 surface, this new observation gives the prediction  formula of Vafa-Witten for the Langlands dual gauge group $\SU(2)/\zz_2=SO(3)$ in \cite[Formula (4.18)]{VW}. We  prove the S-duality conjecture of Vafa-Witten for K3 surfaces  comparing with the result of Tanaka-Thomas in \cite[\S 5]{TT2} for the $\SU(2)$-Vafa-Witten invariants.
The techniques used in this paper should work for any smooth surface. 

\subsection{S-duality conjecture of Vafa-Witten}\label{subsec_VW_history_intro}

We briefly review the S-duality conjecture of $N =4$ supersymmetric Yang-Mills theory on a real 4-manifold $M$ \cite{VW}.  More details can be found in \cite{VW}, and a review is given in \cite{Jiang_ICCM}. 
This theory involves coupling constants $\theta, g$ combined as follows
$$\tau:= \frac{\theta}{2\pi} + \frac{4\pi i}{g^2}.$$
The S-duality predicts that the transformation $\tau\to -\frac{1}{\tau}$ maps the partition function for gauge group $G$ to the partition function with Langlands dual gauge group $^{L}G$.  Vafa-Witten consider a $4$-manifold $M$ underlying a smooth projective surface $S$ over $\cc$ and $G =\SU(r)$. The Langlands dual group 
$^{L}\SU(r)=\SU(r)/\zz_r$. We make these transformations more precise following \cite[\S 3]{VW}. 
Think $\tau$ as the parameter of the upper half plane $\hh$.  Let $\Gamma_0(4)\subset \SL(2,\zz)$ be the subgroup
$$\Gamma_0(4)=\left\{
\mat{cc} a&b\\
c&d\rix\in \SL(2,\zz):  4| c\right\}.$$

The group $\Gamma_0(4)$ acts on $\hh$ by
$$\tau\mapsto \frac{a\tau+b}{c\tau+d}.$$ The group $SL(2,\zz)$ is generated by transformations
$$S=\mat{cc} 0&-1\\
1&0\rix;  \quad T=\mat{cc} 1&1\\
0&1\rix.$$
From \cite{VW}, invariance under $T$ is the assertion that physics is periodic in $\theta$ with period $2\pi$, and 
$S$
is equivalent at $\theta=0$ to the transformation $\frac{g^2}{4\pi}\mapsto (\frac{g^2}{4\pi})^{-1}$ originally proposed by
Montonen and Olive \cite{MO}. One can check  that 
$T(\tau)=\tau+1$, and $S(\tau)=-\frac{1}{\tau}$. 

For a smooth projective surface $S$,  let $Z(\SU(r);\tau)=Z(\SU(r); q)$ be the partition function which counts the invariants of instanton moduli spaces, where we let $q=e^{2\pi i \tau}$. Similarly let  
$Z(\SU(r)/\zz_r; \tau)$ be the  partition function which counts the invariants of $\SU(r)/\zz_r$-instanton moduli spaces.  As pointed out in \cite[\S 3]{VW}, when some vanishing theorem holds, the invariants count the Euler characteristic of  instanton moduli spaces. We will see a mathematical meaning of this vanishing. 
Now the S-duality predicts the following:  
\begin{con}
The transformation $T$ acts on $Z(\SU(r); q)$, and the $S$-transformation sends
\begin{equation}\label{eqn_S_transformation}
Z\left(\SU(r); -\frac{1}{\tau}\right)=\pm r^{-\frac{\chi}{2}}\left(\frac{\tau}{i}\right)^{\frac{\omega}{2}}Z(\SU(r)/\zz_r; \tau).
\end{equation}
for some $\omega$, where  $\chi:=\chi(S)$ is the topological Euler number of $S$.  
\end{con}
Usually $\omega=\chi$.
This is Formula (3.18) in \cite{VW}. 
In mathematics we think $Z(\SU(r);\tau)=Z(\SU(r);q)$ as the partition function which counts the invariants of  moduli space of vector bundles or Higgs bundles on $S$.   Let 
$$\eta(q)=q^{\frac{1}{24}}\prod_{k\geq 1}(1-q^k)$$
be the Dedekind eta function.  Let 
$$\widehat{Z}(\SU(r);q)=\eta(q)^{-w}\cdot Z(SU(r); q),$$
and we will see that $\widehat{Z}(\SU(r);\tau)$ is the partition function of the moduli space of Gieseker stable Higgs sheaves.
Then S-duality predicts:
\begin{con}\label{con_S_transformation_2}
\begin{equation}\label{eqn_S_transformation_2}
\widehat{Z}\left(SU(r); -\frac{1}{\tau}\right)=\pm r^{-\frac{\chi}{2}}\widehat{Z}(SU(r)/\zz_r; \tau)
\end{equation}
\end{con}
This is  Formula (3.15) in \cite{VW}. 
Then $T^4$ acts on the $\SU(r)/\zz_r$-theory to itself; and 
$ST^4 S=\mat{cc} 1&0\\
4&1\rix$ will map the $SU(r)$-theory to itself.   Note that 
$\Gamma_0(4)=\langle T, ST^4 S\rangle$ is generated by $T, ST^4 S$. 
 In the case of a spin manifold, we get the subgroup of $\SL(2,\zz)$ generated by $S$ and $ST^2S$, which is the group 
$$\Gamma_0(2)=\left\{
\mat{cc} a&b\\
c&d\rix\in \SL(2,\zz):  2| c\right\}.$$  
We will see this for K3 surfaces. 
Therefore if the S-duality conjecture holds, the partition function $Z(\SU(r);\tau)=Z(\SU(r);q)$ is a modular form with modular group 
$\Gamma_0(4)$ or $\Gamma_0(2)$ if $M$ is a spin manifold.

In \cite[\S 4]{VW}, Vafa-Witten checked the S-duality for the cases $K3$ surface and $\pp^2$, and gave a prediction on  a formula  (5.38) of \cite[\S 5]{VW}) for general type surfaces.  For $\pp^2$, Vafa-Witten used the mathematical results of Klyachko and Yoshioka, and for $K3$ surfaces, they predicted the formula from physics.  

In algebraic geometry the instanton  invariants are the Euler characteristic of the moduli space of Gieseker or slope stable coherent sheaves on $S$. This corresponds to the case in \cite{VW} such that the obstruction sheaves  vanish.     It is worth mentioning that the blow-up formula of the Vafa-Witten invariants in this case was proved by Li-Qin in \cite{LQ}. 
But to the author's knowledge there are few theories or defining invariants in algebraic geometry for the Langlands dual group $\SU(r)/\zz_r$.  
In the rank $2$ case, the Langlands dual group $\SU(2)/\zz_2=SO(3)$. 
There exist some theories for the $SO(3)$-Donaldson invariants for the surface $S$, see \cite{KM}, \cite{Gottsche}, \cite{MW}.

\subsection{Twisted Vafa-Witten invariants}\label{subsec_tw-sted_VW_intro}

In differential geometry solutions of the Vafa-Witten equation on a projective surface $S$ are given by 
polystable Higgs bundles on the surface $S$, see \cite{TT1}. 
The moduli space of Higgs bundles 
has a partial compactification by Gieseker semistable Higgs pairs $(E,\phi)$ on $S$, where $E$ is a torsion free  coherent sheaf with rank $\rk >0$, and $\phi\in \Hom_{S}(E, E\otimes K_S)$ is a section called a Higgs field.  
Tanaka and Thomas \cite{TT1}, \cite{TT2} have developed a theory of Vafa-Witten invariants using the moduli space $\N$ of Gieseker semi-stable Higgs pairs $(E,\phi)$ on $S$ with topological data $(\rk=\rank, c_1(E), c_2(E))$.  They actually defined the Vafa-Witten invariants   using the moduli space $\N_L^{\perp}$ of Higgs pairs with fixed determinant $L\in\Pic(S)$ and trace-free  $\phi$.  

For the gauge group $\SU(r)$, as in the case of \cite{TT1}, the structure group for the moduli space $\N_L^{\perp}$ is 
$\SL_r(\kappa)$.   The Langlands dual group of $\SL_r(\kappa)$ is 
$^{L}\SL_r(\kappa)=\PGL_r(\kappa)$.
Thus the corresponding moduli space, for the gauge group $\SU(r)/\zz_r$,  should be the moduli space of $\PGL_r(\kappa)$-Higgs bundles or sheaves. 
We consider   $\mu_r$-gerbes on $S$, where $\mu_r$ is the cyclic group of order $r$.  From \cite{Lieblich_ANT},  for a $\mu_r$-gerbe  $\SS\to S$ over a surface $S$ such that the corresponding Brauer class is nontrivial, the moduli stack of semistable $\SS$-twisted sheaves, in some sense, can be taken as a cover over the moduli space of semistable generalized Azumaya algebras on $S$, which in turn, is isomorphic to the the moduli space of $PGL_r(\kappa)$-bundles or sheaves on $S$. This statement is generalized to the twisted Higgs sheaves in \cite{Jiang_2019-2}. Therefore it is reasonable to propose that 
 the Vafa-Witten invariants for the $\mu_r$-gerbes on $S$ will give the mathematical invariants for the Langlands dual group $\SU(r)/\zz_r$.  

Let $p: \SS\to S$ be a $\mu_r$-gerbe, and  fix a polarization $\sO_S(1)$. 
The gerbe $\SS$ is a surface DM stack, and all the equivalence classes of $\mu_r$-gerbes over $S$ are classified by the second \'etale cohomology group
$H^2(S,\mu_r)$.
We have the Vafa-Witten invariants studied in \cite{JP} by the moduli space of stable Higgs sheaves on $\SS$.
But this is too broad for the S-duality conjecture, and we then restrict to a subcategory of sheaves on $\SS$, called the category of $\SS$-twisted sheaves 
in \cite{Lieblich_Duke}. We believe that this is the right category to define twisted Vafa-Witten invariants and check the S-duality.

For a  $\mu_r$-gerbe  $\SS\to S$ over a surface $S$ and let $I\SS$ be the inertia stack,  the $\SS$-twisted sheaves were defined in \cite{Lieblich_Duke}, and we will review it in \S \ref{subsec_gerbe_twisted_sheaf}. Roughly speaking the twisted sheaf is a sheaf with the transition function on an open cover of $\SS$ is gerbe cocycle $[\SS]\in H^2(S,\mu_r)$.
A $\mu_r$-gerbe twisted sheaf is always given by a character morphism $\chi: \mu_r\to \Gm$, and it can be understood as a $\Gm$-gerbe twisted sheaf. 
From the exact sequence
$$1\to \mu_r\rightarrow \sO_S^*\stackrel{(\cdot)^r}{\longrightarrow}\sO_S^*\to 1$$
and the long exact sequence for cohomology
\begin{equation}\label{eqn_long_exact_sequence_intro}
\cdots \to H^1(S,\sO_S^*)\rightarrow H^2(S,\mu_r)\stackrel{\varphi}{\longrightarrow} H^2(S,\sO_S^*)\rightarrow \cdots
\end{equation}
A  $\mu_r$-gerbe $\SS\to S$ corresponds to a $\Gm$-gerbe $[\varphi([\SS])]$.   A quasi-coherent sheaf $E$ on the $\Gm$-gerbe $[\varphi([\SS])]$ has a canonical decomposition 
$$E=\bigoplus_{i}E_i$$
where $E_i$ is the eigensheaf on which the structure group $\Gm$ acts by $\lambda\cdot f=\lambda^i f$.  A  $\Gm$-gerbe twisted sheaf is a quasi-coherent sheaf $E$ such that the structure group (stabilizer group) action 
$$\Gm\times E\to E$$
is given by scalar multiplication, i.e., $E=E_1$.

For a $\SS$-twisted sheaf $E$,  the geometric Hilbert polynomial is defined as
$P^g_E(m) = \chi^g(E(m))$, 
where 
$$\chi^g(E)= [I\SS: \SS]\deg(\Ch(E)\cdot \Td_{\SS})$$
and $[I\SS:\SS]$ is the degree of the inertia stack over the stack $\SS$.
Then we can write down 
$$P^g_E(m) =\sum_{i=0}^d\alpha_{i}(E)\frac{m^i}{i!}.$$
The reduced geometric Hilbert polynomial for pure sheaves, which is denoted  by $p^g_{E}$;  is the monic polynomial with rational coefficients 
$\frac{P_E^g}{\alpha_{d}(E)}$. 
Let $E$ be a pure $\SS$-twisted coherent sheaf, it is semistable if for every proper twisted  subsheaf
$F \subset E$ we have  $p^g_{F} \leq p^g_E$ and it is stable if the same is true with a strict inequality.
Then fixing a geometric  Hilbert polynomial $P$, the moduli stack of $\SS$-semistable coherent sheaves 
$\sM:=\sM^{ss,\tw}_{\SS/\kappa}(P)$ on $\SS$ is constructed in \cite{Lieblich_Duke}.  If the stability and semistability coincide, the coarse moduli space 
$\sM$ is a projective scheme. 

The $\SS$-twisted Higgs sheaves can be similarly defined. The twisted Higgs pair $(E,\phi)$  is semistable if for every proper $\phi$-invariant subsheaf
$F \subset E$ we have  $p^g_{F} \leq p^g_E$.  Let $\N^{\tw}:=\N^{s,\tw}_{\SS/\kappa}(P)$ be the moduli stack of stable $\SS$-twisted Higgs pairs on $\SS$ with geometric Hilbert polynomial $P$.  

Let $\XX:=\mbox{Tot}(K_{\SS})$ be the canonical line bundle of $\SS$, then $\XX$ is a smooth Calabi-Yau threefold DM stack.  
The DM stack $\XX\to X:=\Tot(K_S)$ is also a $\mu_r$-gerbe, and has the same class in $H^2(S,\mu_r)=H^2(X,\mu_r)$.
By spectral theory again, the category of $\SS$-twisted Higgs pairs on $\SS$ is equivalent to the category of $\XX$-twisted  torsion sheaves $\sE_\phi$  on $\XX$ supporting on $\SS\subset \XX$.
Let $\pi: \XX\to \SS$ be the projection.  One can take a projectivization $\overline{\XX}=\proj (K_{\sS}\oplus\sO_{\SS})$, and consider the moduli space of $\XX$-twisted stable torsion sheaves 
on $\overline{\XX}$ with geometric Hilbert polynomial $P$. The open part that is supported on the zero section  $\SS$ is isomorphic to the moduli stack of stable $\SS$-twisted Higgs pairs $\N^{\tw}$ on $\SS$ with geometric Hilbert polynomial $P$. 

We still restrict to the moduli stack  $\N^{\perp,\tw}_{L}$ of stable $\XX$-twisted Higgs pairs 
$(E,\phi)$ with fixed determinant $L$ and trace-free on $\phi$.
There is also a symmetric perfect obstruction theory on  the moduli stack $\N^{\perp,\tw}_{L}$.    
Hence a virtual cycle $[\N^{\perp,\tw}_{L}]^{\vir}\in H_0(\N^{\perp,\tw}_{L})$.  The moduli stack $\N^{\perp,\tw}_{L}$ is non-compact, but admits a 
$\Gm$-action scaling the Higgs field $\phi$.  The invariants are defined by using virtual localization in  \cite{GP}.
Let $[(\N_L^{\perp,\tw})^{\Gm}]^{\vir}$ be the induced virtual cycle on the fixed loci, and $N^{\vir}$ be the virtual normal bundle. 
For the DM stack $\SS$, orbifold Grothendieck-Riemann-Roch theorem implies that fixing a geometric Hilbert polynomial $P$ is the same as fixing 
data $\alpha=(r=\rk, L,c_2)\in H^*(\SS,\qq)$.
\begin{defn}\label{defn_twisted_VW_intro}(Definition \ref{defn_SU_twisted_VW_invariants})
We define 
\begin{equation}\label{eqn_localized_SU_DM_invariants_intr}
\VW^{\tw}_{\alpha}(\SS):=\int_{[(\N_L^{\perp,\tw})^{\Gm}]^{\vir}}\frac{1}{e(N^{\vir})}.
\end{equation}
We call it the big twisted Vafa-Witten invariant for the gauge group $\SU(r)/\zz_r$. 
\end{defn}

On the moduli stack $\N^{\perp,\tw}_{L}$,  the Behrend function 
$$\nu_{\N}: \N^{\perp,\tw}_{L}\to \zz$$
is defined in \cite{Behrend}. 
\begin{defn}\label{defn_twisted_vw_intro}(Definition \ref{defn_SU_twisted_vw_invariants})
We define 
$$
\vw^{\tw}_{\alpha}(\SS)=\chi(\N^{\perp,\tw}_{L}, \nu_{\N})
$$
as the weighted Euler characteristic. We call it the small twisted Vafa-Witten invariant for the gauge group $\SU(r)/\zz_r$. 
\end{defn}
We will see later that the small Vafa-Witten invariants $\vw^{\tw}_{\alpha}(\SS)$ take important role in the calculations. 
If $r=1$, then any $\mu_r$-gerbe $\SS=S$, and we cover the Vafa-Witten invariants $\VW_{\alpha}(S)$ and $\vw_{\alpha}(S)$ defined in \cite{TT1}.

The moduli space $\N_L^{\perp,\tw}$ admits a $\Gm$-action induced by the $\Gm$-action on the total space $\XX$ of the canonical line bundle $K_{\SS}$.  There are two type of $\Gm$-fixed loci on 
$\N_L^{\perp,\tw}$ .  The first one corresponds to the $\Gm$-fixed $\SS$-twisted Higgs pairs $(E,\phi)$ such that the Higgs fields $\phi=0$.  This fixed locus is just the moduli stack $\sM^{s,\tw}_{\SS/\kappa}(P)$ of stable $\SS$-twisted torsion free sheaves $E$ on $\SS$. This is called the {\em Instanton Branch} as in \cite{TT1}.  The second type corresponds to $\Gm$-fixed $\SS$-twisted  Higgs pairs $(E,\phi)$ such that the Higgs fields $\phi\neq 0$.  This case mostly happens when the surfaces $S$ are general type, and this component is called the {\em Monopole} branch.  See \S \ref{subsec_CStar_fixed_locus} for more details. 
It is interesting to do the calculations for $\mu_r$-gerbes on surfaces for this branch. 

\subsection{S-duality conjecture-Global view}\label{subsec_S-duality_global_view}

Our goal is to use the twisted Vafa-Witten invariants $\VW^{\tw}, \vw^{\tw}$ as in Definitions \ref{defn_twisted_VW_intro} and \ref{defn_twisted_vw_intro} to study the $\SU(r)/\zz_r$-Vafa-Witten invariants. 
These twisted invariants, up to now, are defined using the moduli stack of $\SS$-twisted stable sheaves or Higgs sheaves. 
We first need to generalize them to count strictly semistable $\SS$-twisted  sheaves or Higgs sheaves. 

Let $\SS\to S$ be a $\mu_r$-gerbe over a surface $S$.  The $\SS$-twisted sheaves on $\SS$, or the $\SS$-twisted Higgs sheaves forms a category.   There are stability conditions on it.  Then we apply the technique of \cite{JS}, \cite{Joyce07} to count the semistable objects in this category.  When applying to the category of coherent sheaves on a Calabi-Yau threefold, one gets the generalized Donaldson-Thomas invariants.   We apply it to the category of $\SS$-twisted Higgs sheaves to get the generalized twisted Vafa-Witten invariants.

We briefly review the construction and see \S \ref{sec_Joyce-Song} for more details.  On the category $\Coh_c^{\tw}(\XX)$ of $\SS$-twisted Higgs sheaves, the Hall algebra $H(\sA^{\tw})$ is an algebra over 
$K(\St/\kappa)$, the relative Grothendieck ring of stacks.  The moduli stack $\N^{ss,\tw}_{\alpha}(\XX)$ of Gieseker semistable $\XX$-twisted sheaves is an element in the Hall algebra  $H(\sA^{\tw})$.  Then Joyce \cite{Joyce07} defined an element $\epsilon(\alpha)$, called the virtually indecomposable object in  $H(\sA^{\tw})$ for a class $\alpha\in K_0(\XX)$.   Joyce proves that this element has the special form in the Hall algebra and apply the integration map (basically applying the Behrend function to get weighted Euler characteristic), and get the generalized Vafa-Witten invariants 
$\JS^{\tw}_{\alpha}(\XX)$. See \S \ref{subsec_Hall_algebra} for more details. 
For $\alpha=(\rk,L, c_2)\in H^*(\SS,\qq)$, the generalized twisted Vafa-Witten invariant $\vw^{\tw}$ is defined by 
\begin{equation}\label{eqn_vw_generalized_intro}
\vw^{\tw}_{\alpha}(\SS):=(-1)^{h^0(K_{\SS})}\JS^{\tw}_{\alpha}(\XX)\in \qq.
\end{equation}

Similar to \cite{TT2}, the big generalized twisted Vafa-Witten invariants $\VW^{\tw}_{\alpha}(\SS)$ are defined by Conjecture \ref{con_JS_wall_crossing_VW}, motivated by the wall crossing formula of Joyce-Song \cite{JS} for $\vw^{\tw}_{\alpha}(\SS)$.  If semistablity coincides with stability, then the twisted Vafa-Witten invariants $\VW^{\tw}_{\alpha}(\SS)$ is the one in Definition \ref{defn_twisted_VW_intro}.

If the  $\mu_r$-gerbe $\SS\to S$ is trivial, i.e., $\SS=[S/\mu_r]$, then the twisted  Vafa-Witten invariants $\VW^{\tw}, \vw^{\tw}$ (stable ones or generalized ones) are the same as the Vafa-Witten invariants  $\VW, \vw$ defined  in \cite{TT1}, \cite{TT2}.  Then Conjecture \ref{con_JS_wall_crossing_VW} is true for $K_{S}<0$ and $K3$ surfaces.  
For other $\mu_r$-gerbe $\SS\to S$, if the moduli stack $\N^{s,\tw}_{\SS/\kappa}(\alpha)$ is smooth, then $\VW^{\tw}=\vw^{\tw}$.

Recall that a $\mu_r$-gerbe $\SS\to S$ is called  essentially trivial if its class is in the image of the morphism 
$$H^1(S,\sO_S^*)\to H^{2}(S,\mu_r).$$ Thus an essentially trivial $\mu_r$-gerbe $\SS\to S$ is given by a line bundle $\sL\in \Pic(S)$.   A
$\mu_r$-gerbe $\SS\to S$ is called {\em optimal} if the order $|[\SS]|\in H^2(S,\sO^*)_{\tor}$ in the cohomological Brauer group $\Br^\prime(S)=H^2(S,\sO_S^*)_{\tor}$ is of order $r$.
From the long exact sequence (\ref{eqn_long_exact_sequence_intro}), 
an essentially trivial $\mu_r$-gerbe $\SS\to S$ has image zero in $H^2(S,\sO_S^*)$.

It should be interesting to study the Conjecture  \ref{con_JS_wall_crossing_VW} for essentially trivial $\mu_r$-gerbes or optimal $\mu_r$-gerbes  on $S$, and prove that $\VW^{\tw}=\vw^{\tw}$ for essentially trivial $\mu_r$-gerbes or optimal $\mu_r$-gerbes over a K3 surface, and find their difference.
In this paper we only use its first term and 
 we use the difference to calculate the twisted Vafa-Witten invariants for K3 surfaces and prove the S-duality for prime ranks.  It should be possible to generalize the work in \cite{MT} to 
 $\mu_r$-gerbes $\SS$ on a K3 surface and prove $\VW^{\tw}(\SS)=\vw^{\tw}(\SS)$; and a multiple cover formula for the generalized Vafa-Witten invariants of Toda \cite{Toda_JDG}.

\begin{rmk}
Before we define the $\SU(r)/\zz_r$-Vafa-Witten invariants, we make a remark that if in  the $\mu_r$-gerbe $\SS\to S$, $r=1$, then $\SS$ is just the surface $S$. Then the invariants $\VW^{\tw}(\SS), \vw^{\tw}(\SS)$ we defined are just the Vafa-Witten invariants $\VW(S), \vw(S)$ defined in \cite{TT1}, \cite{TT2}.
\end{rmk}

Now we define the $\SU(r)/\zz_r$-Vafa-Witten invariants.  Let $S$ be a smooth projective surface. 

\begin{defn}\label{defn_SU/Zr_VW_intro}
Fix an $r\in\zz_{>0}$, for any $\mu_r$-gerbe $p: \SS_g\to S$
corresponding to $g\in H^2(S,\mu_r)$, let $\sL\in\Pic(\SS_g)$ and 
let 
$$Z_{r,\sL}(\SS_g, q):=\sum_{c_2}\VW^{\tw}_{(r,\sL, c_2)}(\SS_g)q^{c_2}$$
be the generating function of the twisted Vafa-Witten invariants. 

Let us fix a line bundle 
$L\in\Pic(S)$, and define for any essentially trivial $\mu_r$-gerbe $\SS_g\to S$ corresponding to the line bundle $\sL_g\in\Pic(S)$, 
$L_g:=p^*L\otimes \sL_g$; 
for all the other $\mu_r$-gerbe $\SS_g\to S$, keep the same $L_g=p^*L$. 
Also for $L\in\Pic(S)$, let $\overline{L}\in H^2(S,\mu_r)$ be the image under the morphism 
$H^1(S,\sO_S^*)\to H^2(S,\mu_r)$. 

We define
$$Z_{r,L}(S, SU(r)/\zz_r;q):=\sum_{g\in H^2(S,\mu_r)}e^{\frac{2\pi i g\cdot \overline{L}}{r}}Z_{r,L_g}(\SS_g, q).$$
We call it the partition function of $\SU(r)/\zz_r$-Vafa-Witten invariants. This is parallel to the physics conjecture in \cite[(5.22)]{LL}. 
\end{defn}

\begin{con}\label{con_S_duality_intro}(Conjecture \ref{con_S_duality})
For a smooth projective surface $S$, the partition function 
of $\SU(r)$-Vafa-Witten invariants 
$$Z_{r,L}(S, \SU(r);q)=\sum_{c_2}\VW_{(r,L, c_2)}(S)q^{c_2}$$ 
 and 
the partition function of $\SU(r)/\zz_r$-Vafa-Witten invariants 
$Z_{r,L}(S, \SU(r)/\zz_r;q)$ satisfy  the S-duality conjecture in Conjecture (\ref{eqn_S_transformation_2}). 
\end{con}

Our first result is for the projective plane $\pp^2$.   Since $K_{\pp^2}<0$, there are no second component (Monopole Branch) for the moduli space of semistable Higgs sheaves as in \cite{TT1}.   In this case the Vafa-Witten invariants 
$\VW_{(2,\sO,c_2)}(\pp^2)=\vw_{(2,0,c_2)}(\pp^2)$ is just (up to a sign) the Euler characteristic of the moduli space $M^{ss}_{2,0,c_2}(\pp^2)$ of semistable torsion free sheaves.  Let $N_{\pp^2}(2, c_1, c_2)$ be the moduli space of stable vector bundles 
of rank $2$, first Chern class $c_1$ and second Chern class $c_2$.  Let 
$$Z_{c_1}^{\vb, \pp^2}(q)=\sum_{c_2}\chi(N_{\pp^2}(2, c_1, c_2))q^{c_2}$$
be the partition function. 

The $\mu_2$-gerbes on $\pp^2$ are classified by $H^2(\pp^2, \mu_2)=\mu_2$.  Thus one is the trivial $\mu_2$-gerbe $[\pp^2/\mu_2]$, and the other is the nontrivial $\mu_2$-gerbe corresponding to the nontrivial line bundle $\sO(-1)$ on $\pp^2$, it is $\pp(2,2,2)$, the weighted projective stack.  
The $\mu_2$-gerbes on $\pp^2$ are all essentially trivial, and the $\pp(2,2,2)$-twisted sheaves behave like the sheaves on $\pp(2,2,2)$.  
Then in this case the $\VW_{(2,\sO,c_2)}(\pp^2)=\vw_{(2,0,c_2)}(\pp^2)$ is just (up to a sign) the Euler characteristic of the moduli space $M^{ss}_{2,0,c_2}(\pp^2)$ of semistable torsion free sheaves. From a result as in \cite{GJK}, we have that 
in this case the first Chern class $c_1$ is always even. 
Let $\lambda\in\{0,1\}$ index the component in the inertia stack $I\pp(2,2,2)=\pp(2,2,2)\cup\pp(2,2,2)$.
Let  $N_{\pp(2,2,2)}(2, c_1, c_2)$ be the moduli space of stable vector bundles
of rank $2$, first Chern class $c_1$ and second Chern class $c_2$.  Let 
$$Z_{c_1,\lambda}^{\vb, \pp(2,2,2)}(q)=\sum_{c_2}\chi(N_{\pp(2,2,2)}(2, c_1, c_2))q^{c_2}$$
be the partition function. Then  we have
\begin{thm}\label{thm_S-duality_P2_intro}(Theorem \ref{thm_S-duality_P2})
We define 
$$Z^{\pp^2}_0(\SU(2)/\zz_2; \tau):=\frac{1}{2}\cdot\left(q^{-2}\cdot Z_{0,0}^{\vb, \pp(2,2,2)}(q)+q^{-\frac{15}{4}}\cdot Z_{0,1}^{\vb, \pp(2,2,2)}(q)\right)$$
and 
$$Z^{\pp^2}_1(\SU(2)/\zz_2; \tau):=\frac{1}{2}\cdot\left(q^{-6}Z_{2,1}^{\vb, \pp(2,2,2)}(q)- q^{-\frac{15}{4}}\cdot Z_{2,0}^{\vb, \pp(2,2,2)}(q)\right).$$
Write 
$$Z_0^{\pp^2}\left(\SU(2);\tau\right)=q^{-2}\cdot Z_{0}^{\vb,\pp^2}(q);\quad 
Z_1^{\pp^2}\left(\SU(2); \tau\right)=q^{-\frac{15}{4}}\cdot Z_{1}^{\vb,\pp^2}(q).$$  

Then 
under the $S$-transformation  $\tau\mapsto -\frac{1}{\tau}$, we have:
$$
Z_0^{\pp^2}\left(\SU(2); -\frac{1}{\tau}\right)=\pm 2^{-\frac{3}{2}}\left(\frac{\tau}{i}\right)^{\frac{3}{2}}Z_0^{\pp^2}(\SU(2)/\zz_2; \tau)
$$ 
and 
$$
Z_1^{\pp^2}\left(\SU(2); -\frac{1}{\tau}\right)=\pm 2^{-\frac{3}{2}}\left(\frac{\tau}{i}\right)^{\frac{3}{2}}Z_1^{\pp^2}(\SU(2)/\zz_2; \tau)
$$ 
\end{thm}
The S-duality conjecture is proved based on the observation that $Z_{0,0}^{\vb, \pp(2,2,2)}(q)$ is the same as the partition of the trivial $\mu_2$-gerbe $[\pp^2/\mu_2]$, since 
the component of $I\pp(2,2,2)$ corresponding to $\lambda=0$ means that $\mu_2$-action is trivial; and $Z_{0,1}^{\vb, \pp(2,2,2)}(q)$ is the same as the partition function 
$Z_{2,0}^{\vb, \pp(2,2,2)}(q)$ from Theorem \ref{thm_partition_P222}. The proof for the second transformation formula is similar.

\subsection{S-duality conjecture-K3 surfaces}\label{subsec_S-duality_K3_intro}

Let $S$ be a K3 surface. 
The following result for K3 surfaces was proved in \cite[Theorem 1.7]{TT2} by calculating the invariants $\vw_{\alpha}(S)$ using Toda's multiple cover formula \cite{Toda_JDG}:
 \begin{equation}\label{eqn_partition_function_K3}
Z_{r,0}(S,\SU(r);q):= \sum_{c_2}\VW_{r,0,c_2}(S)q^{c_2}=\sum_{d|r}\frac{d}{r^2}q^r\sum_{j=0}^{d-1}\eta\left(e^{\frac{2\pi i j}{d}}q^{\frac{r}{d^2}}\right)^{-24}.
 \end{equation}
 where 
 $$\eta(q)=q^{\frac{1}{24}}\prod_{k>0}(1-q^k)$$
 is the Dedekind eta function;  and $Z_{r,0}(S,\SU(r);q)$ is   the generating series of rank $r$ trivial determinant Vafa-Witten invariants. 

 If $r$ is a prime number, then 
 \begin{equation}\label{rem_K3_partition_prime}
\sum_{c_2}\VW_{r,0,c_2}(S)q^{c_2}=\frac{1}{r^2}q^r \eta(q^r)^{-24}+\frac{1}{r}q^r\sum_{j=0}^{r-1}\eta\left(e^{\frac{2\pi i j}{r}}q^{\frac{1}{r}}\right)^{-24},
\end{equation}
which  is the prediction formula in \cite[\S 4.1]{VW}.

 \subsubsection{Rank $2$ $S$-transformation}
 
 Let us study in detail in the rank two case.  
 Let $$Z(\SU(2); q):=Z_{2,0}(S,\SU(2);q); \quad Z(\SU(2)/\zz_2; q):=Z_{2,0}(S,\SU(2)/\zz_2;q)$$
 after fixing the rank $2=r$, trivial determinant $\det(E)=\sO$. 
First we write down Tanaka-Thomas's partition function as 
\begin{align}\label{eqn_SU2_partition}
Z(\SU(2); q)&:=\sum_{c_2}\VW_{2,0,c_2}(S)q^{c_2}\\
&=\frac{1}{4}q^2 \eta(q^2)^{-24}+\frac{1}{2}q^2\left(\eta\left(q^{\frac{1}{2}}\right)^{-24}+\eta\left(-q^{\frac{1}{2}}\right)^{-24}\right) \nonumber \\
&=\frac{1}{4}q^2 G(q^2)+\frac{1}{2}q^2\left(G(q^{\frac{1}{2}})+G(-q^{\frac{1}{2}})\right) \nonumber
\end{align} 
Here $q=e^{2\pi i\tau}$ and $\tau$ is the parameter fo the upper half plane.
We use Vafa-Witten's notation in \cite{VW} and denote by
$$G(q):=\eta(q)^{-24}.$$

Under the $S$-transformation $\tau\mapsto -\frac{1}{\tau}$, from modular transformation properties \cite[\S 4.1]{VW}, we have 
$$
\begin{cases}
G(-q^{\frac{1}{2}})\mapsto \tau^{-12}G(-q^{\frac{1}{2}}),\\
G(q^{\frac{1}{2}})\mapsto 2^{-12}\tau^{-12}G(q^2),\\
G(q^2)\mapsto 2^{12}\tau^{-12}G(q^{\frac{1}{2}}).
\end{cases}
$$
From \cite[\S 4.1]{VW}, first we shift back $q^2$, and under $S$ transformation 
$$Z(\SU(2); q)\mapsto \frac{1}{4}2^{12}\tau^{-12}G(q^{\frac{1}{2}})+\frac{1}{2}2^{-12}\tau^{-12}G(q^2) +\frac{1}{2}\tau^{-12}G(-q^{\frac{1}{2}}).$$
Thus from (\ref{eqn_S_transformation}), after modifying $2^{11}$, and shift back $q^2$ we get 
\begin{equation}\label{eqn_SU2_Z2_K3}
Z(\SU(2)/\zz_2; q)=\frac{1}{4}q^2 G(q^2)+q^2 \left(2^{21}\cdot G(q^{\frac{1}{2}})+2^{10}\cdot G(-q^{\frac{1}{2}})\right), 
\end{equation}
which  should be the partition function $Z(\SU(2)/\zz_2, q)$ for the K3 surface $S$. 
This is Formula (4.18) in \cite{VW},  and has modular properties for $\Gamma_0(2)$.

\subsubsection{Proof of rank two formula}
 
We prove that the twisted Vafa-Witten invariants $\VW^{\tw}$ (also $\vw^{\tw}$) in Conjecture \ref{con_JS_wall_crossing_VW} for the $\mu_2$-gerbes on $S$ give the formula 
(\ref{eqn_SU2_Z2_K3}). Thus we prove Conjecture \ref{con_S_duality_intro} for K3 surfaces $S$ in rank two. 
We have the following result: 

\begin{thm}\label{thm_SU2Z2_partition_function_intro}(Theorem \ref{thm_SU2Z2_partition_function})
Let $S$ be a smooth projective K3 surface with Picard number $\rho(S)$. Then 
$$Z(S, \SU(2)/\zz_2; q)=\frac{1}{4}q^2 G(q^2)+q^2 \left(2^{21}\cdot G(q^{\frac{1}{2}})+2^{\rho(S)-1}\cdot G(-q^{\frac{1}{2}})\right). $$
\end{thm}

\begin{cor}\label{cor_S-duality_K3_intro}(Corollary \ref{cor_S-duality_K3})
Let $S$ be a smooth projective K3 surface with Picard number $\rho(S)=11$.  Then 
$$Z(S, \SU(2)/\zz_2; q)=\frac{1}{4}q^2 G(q^2)+q^2 \left(2^{21}\cdot G(q^{\frac{1}{2}})+2^{10}\cdot G(-q^{\frac{1}{2}})\right).$$
\end{cor}

It is fun to include the calculation by Maple of the series $Z(S, SU(2); q)$ and $Z(S, \SU(2)/\zz_2; q)$ for  a smooth projective K3 surface $S$ with Picard number $\rho(S)=11$.
We have 
\begin{align}\label{eqn_partition_function_SU2_maple}
Z(S, \SU(2); q)&=\frac{1}{4}q^2 G(q^2)+\frac{1}{2}q^2 \left(G(q^{\frac{1}{2}})+ G(-q^{\frac{1}{2}})\right) \\ \nonumber
&=\frac{1}{4}+30 q^2+ 3200 q^3+176337 q^4+ 5930496 q^{5} \\ \nonumber
&+ 143184800 q^{6}+ 2705114280 q^{7} +O(q^8)  \nonumber 
\end{align}
which is the formula of Tanaka-Thomas in \cite[\S 5]{TT2}.

\begin{align}\label{eqn_partition_function_SU2/Z2_maple}
Z(S, \SU(2)/\zz_2; q)&=\frac{1}{4}q^2 G(q^2)+q^2 \left(2^{21}\cdot G(q^{\frac{1}{2}})+2^{10}\cdot G(-q^{\frac{1}{2}})\right)\\ \nonumber
&=\frac{1}{4}+2096128 q^{\frac{3}{2}}+50356230 q^{2}+679145472 q^{\frac{5}{2}}+6714163200 q^3 \\ \nonumber  
&+53765683200 q^{\frac{7}{2}} 
+ 369816109137 q^{4}+2250654556160q^{\frac{9}{2}} + \\ \nonumber
&+12443224375296 q^{5}+ 63258156057600 q^{\frac{11}{2}}+O(q^6) \nonumber 
\end{align}
From this calculation,  if we define $\vw^{\SU(2)/\zz_2}_{\alpha}(S)$ as the Vafa-Witten invariants for the $\SU(2)/\zz_2$ theory, we get that 
$$\vw^{\SU(2)/\zz_2}_{2,0}(S)=\frac{1}{4}$$
and 
$$\vw^{\SU(2)/\zz_2}_{2,\frac{3}{2}}(S)=\frac{1}{2}(2^{22}-2^{11})=2096128.$$
Note that $\frac{1}{2}(2^{22}-2^{11})=2^{21}-2^{10}=2096128$,  and this invariant is purely given by optimal $\mu_2$-gerbes on $S$.

The method to prove Theorem \ref{thm_SU2Z2_partition_function_intro} is the following.  Let $S$ be a smooth K3 surface 
with Picard number $\rho(S)$, then $0\leq \rho(S)\leq 20$ over a field $\kappa$ of character zero.   Over positive characteristic $p>0$ field $\kappa$, the Picard number can be $\rho(S)=22$, which in this case $S$ is called a ``supersingular K3".   The cohomology 
$H^2(S,\mu_2)=\zz_2^{22}$.  Since the Picard number is $\rho(S)$, there are totally $2^{\rho(S)}$ number  of equivalent essentially trivial $\mu_2$-gerbes on $S$;  and all the other 
$\mu_2$-gerbes on $S$ are optimal, meaning that the order $|[\SS]|=2$ in $H^2(S,\sO_S^*)_{\tor}$. There are  totally $2^{22}-2^{\rho(S)}$ number  of equivalent optimal $\mu_2$-gerbes on $S$.

For the trivial $\mu_2$-gerbe $\SS\to S$, the twisted Vafa-Witten invariants $\VW^{\tw}(\SS)=\vw^{\tw}(\SS)$ are the same as usual Vafa-Witten invariants for $S$ in \cite{TT2}, see Proposition \ref{prop_trivial_mur_gerbe_K3}.  For the non-trivial essentially trivial  $\mu_2$-gerbe $\SS\to S$, the twisted Vafa-Witten invariants $\VW^{\tw}(\SS)=\vw^{\tw}(\SS)$ are calculated in Proposition \ref{prop_twisted_vw_essential_trivial}. They behave a bit different from the trivial gerbe case,  based on one fact that a rank two vector bundle on $\SS$ with non-trivial second Stiefel-Whitney class (first Chern class (modulo $2$) ) can not split.  

Let $\SS\to S$ be an optimal $\mu_2$-gerbe over $S$.  In this case we use the theory of Huybrechts and Stellari \cite{HS} for the Hodge structure for twisted K3, and  of Yoshioka \cite{Yoshioka2} for the moduli stack of  gerbe twisted sheaves.   An important observation is that the index $\ind(\SS)$ of the gerbe $\SS$, which by definition the minimal rank $r$ such that there exists a $\SS$-twisted locally free of rank $r$ sheaf on generic scheme $S$, is actually the period $\per(\SS)=|[\SS]|=2$, i.e., the order of $[\SS]$ in $H^2(S,\sO_S^*)_{\tor}$. 
In  the moduli stack of rank $2$ $\SS$-twisted sheaves or Higgs sheaves $(E,\phi)$, $E$ must be stable.  Thus studying the moduli stack of $\SS$-twisted sheaves in rank $2$ is a non-commutative analogue of the Picard scheme in the sense that the $\SS$-twisted sheaves are essentially rank one right modules over an Azumaya algebra on $S$.  Using this we show that the moduli stack of stable $\SS$-twisted rank $2$ Higgs sheaves is a $\mu_2$-gerbe over the Yoshioka moduli space of $[\SS]$ ( as a $\Gm$-gerbe)-twisted stable Higgs sheaves on $S$. 
Then we use the theory of Yoshioka to calculate the  twisted Vafa-Witten invariants $\VW^{\tw}(\SS)=\vw^{\tw}(\SS)$, see Corollary \ref{cor_vw_twisted_optimal}. 
All the calculations imply the result in Theorem \ref{thm_SU2Z2_partition_function_intro}.

\subsubsection{Vafa-Witten's S-duality conjecture}

If $S$ is a complex K3 surface, the result in Corollary \ref{cor_S-duality_K3_intro} depends on the complex structure on the K3 surface $S$. Vafa-Witten's prediction in \cite{VW} does not depend on the complex structure.  We modify Conjecture \ref{con_S_duality_intro} a bit for K3 surfaces in rank two and prove that the modified formula $Z^\prime(S, \SU(2)/\zz_2; q)$ does not depend on the complex structure. 
In the paper \cite[Page 55]{VW},  Vafa-Witten actually sum over topological data which means that  they sum all the first Chern classes 
$g\in H^2(S,\zz)/2\cdot H^2(S,\zz)$ such that $g^2\equiv 0,  2 \mod 4$.  There are totally three types of these data: 
$g= 0$, $g^2\equiv 0\mod 4$ but nonzero, and $g^2\equiv 2\mod 4$ which are called the zero, even non-zero and odd classes respectively.  Let $n_0=1, n_{\even}, n_{\odd}$ be the number of zero classes $g$, even non-zero  classes $g$ and odd classes $g$  respectively.  Then 
$$n_{\even}=\frac{2^{22}+2^{11}}{2}-1; \quad  n_{\odd}=\frac{2^{22}-2^{11}}{2}.$$ 
There are two types of non-zero even classes $g\in H^2(S,\zz)/2\cdot H^2(S,\zz)$. The first one consists of algebraic classes (i.e. $(1,1)$-classes), and the number is 
$n_{\even}^1=(2^{\rho(S)}-1)$.  The second one consists of non-algebraic classes, and the number is  $n_{\even}^2=n_{\even}-(2^{\rho(S)}-1)$.   Since the Picard number $\rho(S)\leq 20$, these numbers make sense. 

Let 
$\SS_0\to S$ be the trivial $\mu_2$-gerbe; and $\SS_{\ess}\to S$ a nontrivial essentially trivial $\mu_2$-gerbe on $S$ corresponding to a line bundle $\sL_g\in \Pic(S)$.  
We set 
$Z_0(q):=Z_{2,\sO}(\SS_0, q)$ and calculate the partition functions:
$$Z_{2,\sL_g}(\SS_{\ess}, q)=\frac{1}{2}q^2 \left(G(q^{\frac{1}{2}})+ G(-q^{\frac{1}{2}})\right)$$ 
and 
$$Z_0(q)=\frac{1}{4}q^2 G(q^2)+\frac{1}{2}q^2 \left(G(q^{\frac{1}{2}})+ G(-q^{\frac{1}{2}})\right).$$

Let  $\SS_{\opt}\to S$ be an optimal $\mu_2$-gerbe  on $S$, the twisted Vafa-Witten invariants are summing over second Chern classes $c_2$ taking values in half integers. We set 
$$Z_{2,\sO}(\SS_{\opt}, (-1)^2\cdot q):=\sum_{c_2}\vw^{\tw}_{v}(\SS_{\opt})(-1)^{2c_2}q^{c_2}$$ 
and calculate 
$$Z_{2, \sO}(\SS_{\opt}, q)=\frac{1}{2}q^2 G(q^{\frac{1}{2}}); \quad Z_{2, \sO}(\SS_{\opt}, (-1)^2\cdot q)=\frac{1}{2}q^2 G(-q^{\frac{1}{2}}).$$

For any line bundle $L\in\Pic(S)$, the moduli stack $\N^{\tw}_{2,L,c_2}(\SS_{\opt})$ of $\SS_{\opt}$-twisted stable Higgs sheaves with fixed determinant $L$ is  isomorphic to the  moduli stack  $\N^{\tw}_{2,\sO,c^\prime_2}(\SS_{\opt})$ with  trivial determinant.  Therefore the partition function $Z_{2, L}(\SS_{\opt}, q)$ or $Z_{2,L}(\SS_{\opt}, (-1)^2\cdot q)$ is the same as 
$Z_{2, \sO}(\SS_{\opt}, q)$ or $Z_{2, \sO}(\SS_{\opt}, (-1)^2\cdot q)$. 

For the even classes $g\in H^2(S,\zz)/2\cdot H^2(S,\zz)$ such that $g\neq 0$, 
we define 
$$ 
Z_{\even}(q):=
Z_{2, \sO}(\SS_{\opt}, q)+ Z_{2, \sO}(\SS_{\opt}, (-1)^2\cdot q).
$$

For odd classes $g\in H^2(S,\zz)/2\cdot H^2(S,\zz)$, we define 
$$Z_{\odd}(q):=Z_{2, \sO}(\SS_{\opt}, q)+(-1)^{2\pi i\frac{1}{2}}\cdot Z_{2, \sO}(\SS_{\opt}, (-1)^2\cdot q).$$

Then we have:
\begin{thm}\label{thm_SU2/Z2_Vafa-Witten_K3_intro}(Theorem \ref{thm_SU2/Z2_Vafa-Witten_K3})
Let $S$ be a smooth complex K3 surface. 
Define
\begin{align*}
Z^{\prime}_{2,\sO}(S, \SU(2)/\zz_2;q):=Z_0(q)+ n_{\even}\cdot Z_{\even}(q)
+n_{\odd}\cdot Z_{\odd}(q).
\end{align*}
We call it the partition function of $\SU(2)/\zz_2$-Vafa-Witten invariants, and we have:
$$Z^{\prime}(S, \SU(2)/\zz_2; q)=\frac{1}{4}q^2 G(q^2)+q^2 \left(2^{21}\cdot G(q^{\frac{1}{2}})+2^{10}\cdot G(-q^{\frac{1}{2}})\right).$$
This proves the prediction of Vafa-Witten in \cite[Formula (4.18)]{VW} for the gauge group $\SU(2)/\zz_2$ and the S-duality conjecture (\ref{eqn_S_transformation_2}),  and   (\ref{eqn_SU2_Z2_K3}).
\end{thm}

The formula is a direct calculation from the proof of  Theorem \ref{thm_SU2Z2_partition_function_intro}, 
see \S \ref{sub_VW_con_K3}.
We also talk about the higher rank case for the K3 surface $S$.  For the prime rank $r$, the result for the partition function of 
$\SU(r)/\zz_r$-Vafa-Witten invariants  and Vafa-Witten's conjecture can be similarly obtained as in the rank $2$ case, see Theorem \ref{thm_SU2Z2_partition_function_higher_prime_rank}, and Theorem \ref{thm_SUr/Zr_Vafa-Witten_K3}. 
The nonprime rank $r$ case is generalized in \cite{Jiang_Tseng}, where a twisted version of Toda's multiple cover formula is essential. 
In \cite{JK}, the authors prove the refined version of the  S-duality conjecture of K3 surfaces using  the K-theoretical refined Vafa-Witten invariants \cite{Thomas2}.

\subsection{Outline} 
This  paper is outlined as follows.  We review the basic materials  of algebraic stacks, Deligne-Mumford stacks and gerbes in \S \ref{sec_gerbes_stacks}.  We define the moduli of gerbe twisted semistable sheaves and semistable Higgs sheaves in \S \ref{sec_moduli_twisted_Higgs_pairs}.  We study the perfect obstruction theory of the moduli stack of twisted Higgs sheaves in \S \ref{sec_POT_virtual_Class_VW_Invariants} and define the  twisted Vafa-Witten invariants. 
We study the Joyce-Song twisted stable pairs in \S \ref{sec_Joyce-Song} and define generalized twisted Vafa-Witten invariants. We also define the $\SU(r)/\zz_r$-Vafa-Witten invariants and conjecture that its partition function satisfies the S-duality conjecture. The case of   $\pp^2$ is proved  in rank two.  Finally in \S \ref{sec_S-duality_K3} we prove the S-duality conjecture for K3 surfaces for prime ranks.

\subsection{Convention}
We work over an algebraically closed field $\kappa$ of character zero throughout of the paper. We denote by $\Gm$ the multiplicative group over $\kappa$.   We use Roman letter $E$ to represent a coherent sheaf on a projective DM stack or an \'etale gerbe  $\SS$, and use curl latter $\sE$ to represent the sheaves on the total space Tot$(\sL)$ of a line bundle $\sL$ over $\SS$. 
We reserve {\em $\rk$} for the rank of the torsion free coherent sheaves $E$, and when checking the S-duality for $\SU(r)/\zz_r$, $r=\rk$.
We keep the convention in \cite{VW} to use $\SU(r)/\zz_r$ as the Langlands dual group of $\SU(r)$.  When discussing the S-duality the rank $\rk=r$.

\subsection*{Acknowledgments}

Y. J. would like to thank  Kai Behrend,  Huai-Liang Chang,  Amin Gholampour, Shui Guo, Martijn Kool,  Yang Li,  Song Sun, and Hsian-Hua Tseng for valuable discussions.  Many thanks to Martijn Kool for his interest in the paper and valuable comments, and Richard Thomas for his nice comments of the paper and communication of Brauer classes, especially explaining the idea of Vafa-Witten of summing over Chern classes for $\PGL_2$-bundles on K3 surfaces.  This work is partially supported by  NSF DMS-1600997.


\section{Preliminaries on \'etale gerbes and stacks}\label{sec_gerbes_stacks}

In this section we review the basic materials of stacks.  Our main reference is \cite{LMB} and \cite{Stack_Project}.  Since we mainly work on \'etale $\mu_r$-gerbes over schemes and $\mu_r$-gerbes are closely related to $\Gm$-gerbes on a scheme, we recall the basic notion of algebraic stacks (Artin stacks) and also Deligne-Mumford (DM) stacks.   Several interesting examples and basic knowledge were reviewed in \cite{JP}.

\subsection{Algebraic stacks and Deligne-Mumford stacks}\label{subsec_Artin_DM_Stack}
Let us always fix a Noetherian scheme $R$, and let $(\Sch/R)$ be the Grothendieck topology, either \'etale or $\fppf$ topology.  The \'etale topology will give us DM stacks and $\fppf$ topology will give us algebraic stacks (Artin stacks). 
In later sections we take $R$ to be the base field $\kappa$.

Let us first fix notations. Let $\XX$ be a functor and 
$f_i: U_i\to U$ be a morphism of schemes and $X\in \XX(U)$ an object.  Denote by $X|_i:=f_i^*X$ and 
$X_i|_{ij}=f_{ij, i}^*X_i$ where 
$f_{ij, i}: U_i\times_{U}U_j\to U_i$ is the morphism and $X_i\in \XX(U_i)$. 

\begin{defn}\label{defn_stacks}
A stack  is a sheaf $\XX$ of groupoid, i.e., a $2$-functor (presheaf) that satisfies the following sheaf axioms.  Let 
$\{U_i\to U\}_{i\in I}$ be a covering of $U$ in the site $(\Sch/R)$. Then 
\begin{enumerate}
\item (Gluing) If $X$ and $Y$ are two objects of $\XX(U)$, and $\varphi_i: X|_{i}\to Y|_i$ are morphisms such that $\varphi_i|_{ij}=\varphi_j|_{ij}$, then there exists a morphism 
$\eta: X\to Y$ such that $\eta|_{i}=\varphi_i$.
\item  (Mono presheaf) If $X$ and $Y$ are two objects of $\XX(U)$, and $\varphi: X\to Y$, $\psi: X\to Y$ are morphisms such that $\varphi|_{i}=\psi|_{i}$, then $\varphi=\psi$.
\item (Gluing of objects) If $X_i$ are objects of $\XX(U_i)$ and $\varphi_{ij}: X_j|_{ij}\to X_i|_{ij}$ are morphisms satisfying the cocycle condition 
$\varphi_{ij}|_{ijk}\circ \varphi_{jk}|_{ijk}=\varphi_{ik}|_{ijk}$, then there exists an object $X$ of $\XX(U)$ ans $\varphi_i: X|_i\stackrel{\cong}{\longrightarrow}X_i$ such that 
$\varphi_{ji}\circ \varphi_i|_{ij}=\varphi_{j}|_{ij}$. 
\end{enumerate}
\end{defn}

\begin{defn}\label{defn_DM_stack}
We take $\Sch/R$ to be the category of $R$-schemes with the \'etale topology.  Let $\XX$ be a stack.  $\XX$ is a DM stack if 
\begin{enumerate}
\item The diagonal $\Delta_{\XX}$ is representable, quasi-compact and separated. 
\item There exists a scheme $U$ (called atlas) and an \'etale (hence locally of finite type) and surjective morphism $u: U\to \XX$. 
\end{enumerate}
\end{defn}

A morphism $f: \XX\to \YY$ of DM stacks is called ``representable" if every morphism $g: S\to \YY$ from a scheme $S$, the fibre product $S\times_{g, \YY, f}\YY$ is a scheme. In particular, any morphism from a scheme to $\YY$ is representable. 

\begin{defn}\label{defn_inertia_stack}
Let $\XX$ be a smooth DM stack. The inertia stack $I\XX$ associated to $\XX$ is defined to be the fibre product:
$$I\XX:=\XX\times_{\Delta, \XX\times \XX}\XX$$
where $\Delta: \XX\to \XX\times\XX$ is the diagonal morphism.
\end{defn}
The objects in the category underlying $I\XX$ is: 
$$\Ob(I\XX)=\{(x,g)| x\in \Ob(\XX), g\in \Aut_{\XX}(x)\}.$$
\begin{rmk}
\begin{enumerate}
\item There exists a morphism $q: I\XX\to \XX$ given by $(x,g)\mapsto x$;
\item There exists a decomposition 
$$I\XX=\bigsqcup_{r\in\nn}\Hom(\mbox{Rep}(B\mu_r, \XX)),$$
hence $I\XX$ can be decomposed into connected components, and 
$$I\XX=\bigsqcup_{i\in\mathcal{I}}\XX_i$$
where $\mathcal{I}$ is the index set.  We have $\XX_0=\{(x, \id)|x\in \Ob(\XX), \id\in \Aut_{\XX}(\XX)\}=\XX$. 
\end{enumerate}
\end{rmk}

\begin{defn}\label{defn_Artin_stack}
Let $\XX$ be a stack over the site $(\Sch/R)$ with the $\fppf$ topology.   Assume that 
\begin{enumerate}
\item The diagonal $\Delta_{\XX}$ is representable, quasi-compact and separated. 
\item There exists a scheme $U$ (called atlas) and a smooth and surjective morphism $u: U\to \XX$ in $\fppf$ topology. 
\end{enumerate}
Then we call $\XX$ an Artin stack.
\end{defn}

\begin{example}
A very interesting stack is given by the quotient stack $\XX=[X/G]$, where $X$ is a smooth scheme and $G$ an algebraic group acting on $X$.  
The stack $[X/G]$ classifies isomorphic classes of principle $G$-bundles $E$ over a scheme $B\to [X/G]$ together with a $G$-equivariant morphism:
\[
\xymatrix{
E\ar[r]^{f}\ar[d]& X\ar[d]^{u}\\
B\ar[r]& [X/G]
}
\]
The smooth atlas of $\XX$ is $X\to [X/G]$.  If $G$ is a finite group, then $\XX$ is a DM stack.  All stacks in this paper locally is a quotient stack. 

Let  $\mathcal{I}=\{(g)| g\in G\}$, where $(g)$ is the conjugacy class.  The inertia stack $I\sX=\bigsqcup_{(g)}\sX_{(g)}$, and $\sX_{(g)}=[M^g/C(g)]$, where $M^g$ is the $g$ fixed locus of $M$ and $C(g)$ is the centralizer of $g$. 
\end{example}

\subsection{Gerbes and essentially trivial gerbes}\label{subsec_gerbe_essential}

Gerbes are interesting algebraic stacks.  Let us fix a base algebraic stack $X$.  Later on we always consider $X$ to be a smooth scheme. 

\begin{defn}\label{defn_gerbes}
A gerbe $\XX$ over $X$ is a stack $\XX$ such that the following is satisfied:
\begin{enumerate}
\item  For any $U\in X$ there exists a covering $U^\prime\to U$ such that $\XX(U^\prime)\neq \emptyset$.
\item  For any $U\in X$ and any $s, s^\prime\in \XX(U)$, there exists a covering $U^\prime\to U$ such that 
$s|_{U}$ is isomorphic to $s^\prime|_{U^\prime}$.
\end{enumerate}
\end{defn}

\begin{defn}\label{defn_gerbe_sheaf}
Let 
$\XX$  be a gerbe over $X$. 
The sheaf associated to $\XX$, denoted by $\Sh(\XX)$ is the sheafification  of the presheaf whose sections over $U\subset X$ are isomorphic classes of objects in the fiber category $\XX(U)$. 
\end{defn}

In this paper we mainly work on $A$-gerbes for an abelian group scheme $A$ over $R$.  First let $\widetilde{\XX}$ be the classifying topos of $\XX$, i.e., the topos of sheaves on the site of $\XX$. Then we have:

$\bullet$ If $\XX\to X$ is a gerbe and the inertia stack  $I\XX$ is an abelian group sheaf on $\XX$, then there exists an abelian sheaf $A$ on $X$ such that 
$$\widetilde{\pi}^* A\cong I\XX$$
on  $\widetilde{\XX}$, where $\widetilde{\pi}:  \widetilde{\XX}\to X$ is the morphism of topoi. 

\begin{defn}\label{defn_A_gerbe}
An $A$-gerbe $\XX$ over $X$ is a gerbe $\XX\to X$ such that 
$A_{\XX}\cong I\XX$ in $\widetilde{\XX}$, where $A_{\XX}$ is a sheaf of abelian group $A$ on $\XX$. 
\end{defn}

Let $X$ be a scheme, the $A$-gerbes $\XX$ over $X$ are classified by the second \'etale cohomology group 
$H^2(X, A)$.  In this paper we are interested in the group $A=\mu_r$ or $A=\Gm$, i.e., the $\mu_r$-gerbes on a scheme 
$X$; or the $\Gm$-gerbes on $X$.  It is not hard to see that a $\mu_r$-gerbe $\XX$ on a scheme 
$X$ is a DM stack.  We always let $\chi: A\to \Gm$ be a character. 

\begin{defn}\label{defn_essential_trivial_gerbe}
A  $\mu_r$-gerbe $\XX\to X$ on a scheme 
$X$ is (geometrically) {\em essentially trivial} if the class $[\XX]\in H^2(X, \mu_r)$ has trivial image in $H^2(X, \Gm)$ under the natural map
$H^2(X, \mu_r)\to H^2(X, \Gm)$. 
\end{defn}

Consider the short exact sequence:
$$1\to \mu_r\longrightarrow \Gm\stackrel{(\cdot)^r}{\longrightarrow}\Gm\to 1$$
and taking cohomology we have the exact sequence
\begin{equation}\label{eqn_cohomology_mur_Gm}
\cdots\to H^1(X, \Gm)\longrightarrow H^2(X, \mu_r)\longrightarrow H^2(X, \Gm)\to \cdots
\end{equation}
Thus a $\mu_r$-gerbe $\XX\to X$ is essentially trivial if and only if it is given by the image class of a line bundle 
$\sL\in \Pic(X)$ under the first map in (\ref{eqn_cohomology_mur_Gm}). Also in this case one can understand the $\mu_r$-gerbe $\XX\to X$
as a quotient 
$[\Tot(\sL^{\times})/\Gm]$, where $\Tot(\sL^{\times})$ denote the total space of the line bundle $\sL$ minus the zero section $X$, and $\Gm$ acts on the fiber by the 
$r$-th power. 

\subsection{Brauer group and optimal gerbes}\label{subsec_Brauer_optimal_gerbe}

In this section we talk about the Brauer group and optimal gerbes.  For the basic knowledge of Brauer group, see \cite{Milne}. 
Let $X$ be a scheme, the Brauer group $\Br(X)$ parametrizes equivalence classes of sheaves of Azumaya algebras on $X$. 

\begin{defn}\label{defn_Azumaya_algebra}
An Azumaya algebra on $X$ is an associative (non-commutative) $\sO_X$-algebra $\sA$ which is locally isomorphic to a matrix algebra $M_r(\sO_X)$ for some $r>0$.
\end{defn}

Two Azumaya algebras $\sA$ and $\sA^\prime$ are called equivalent if there exist non-zero vector bundles $E$ and $E^\prime$ such that $\sA\otimes \sE nd(E)$ and $\sA^\prime\otimes \sE nd(E^\prime)$ are isomorphic Azumaya algebras.
The Brauer group $\Br(X)$ is the set of isomorphism classes of Azumaya algebras modulo the above equivalence relation.

The cohomological Brauer group $\Br^\prime(X)$ is the torsion part $H^2(X, \sO_X^*)_{\tor}$ of the \'etale cohomology $H^2(X, \sO_X^*)$, which classifies the equivalence classes of $\Gm$-gerbes on $X$.  The following result is due to de Jong.

\begin{thm}\label{thm_de_Jong}(de Jong \cite{de_Jong})
If $X$ is a quasi-compact separate scheme admitting an ample invertible sheaf, then $\Br(X)=\Br^\prime(X)$. 
\end{thm}

\begin{defn}\label{defn_optimal_gerbe}
Let $\XX\to X$ be a $\mu_r$-gerbe. The period $\per(\XX)$ is the order of $[\XX]$ in $\Br(X)$. 
A $\mu_r$-gerbe $\XX\to X$ is called {\em optimal} if the period $\per(\XX)=r$. 
\end{defn}

The Brauer group also classifies Brauer-Severi varieties, and these are projective bundles $P$ over $X$, see \cite{Artin}.  
For us, if we have an optimal $\mu_r$-gerbe $\XX\to X$ over a scheme $X$,  then from Theorem \ref{thm_de_Jong}, $\Br(X)=\Br^\prime(X)$. We let 
$o([\XX])$ be the image of $[\XX]$ under the map
$$o: H^2(X, \mu_r)\to H^2(X, \sO_X^*)$$
which must lie in the torsion part of $H^2(X, \sO_X^*)$.   Any Azumaya algebra $\sA$ is locally free of constant rank $r^2$. Two Azumaya algebras are isomorphic if they are isomorphic as $\sO_X$-algebras. By Skolem-Noether theorem,  $\Aut(M_r(\kappa)) = \PGL_r(\kappa)$ (acting by conjugation). Therefore, the set of isomorphism classes of Azumaya algebras A of rank $r^2$ is in bijection with the set $H^1(X,\PGL_r)$. From the exact sequence
$$1\to \mu_r\longrightarrow SL_r\longrightarrow PGL_r\to 1$$
we have a morphism by taking cohomology $\delta^\prime: H^1(X,PGL_r(\sO_X))\to H^2(X, \mu_r)$. We set $\delta:=o\circ\delta^\prime$.
Then the optimal $\mu_r$-gerbe $[\XX]$ also defines a class $[P]\in H^1(X,\PGL_r)$ such that $\delta^\prime[P]=[\XX]$.  This projective bundle 
$P\to X$ is the {\em Brauer-Severi} variety. 

\begin{rmk}
 If a $\mu_r$-gerbe $\XX\to X$ is essentially trivial, then its class $[\XX]$ is zero in $H^2(X, \Gm)$, therefore also defines a zero class $[P]\in H^1(X,\PGL_r)$. In this case the Brauer-Severi variety $P$ is just $X$.
\end{rmk}

\section{Moduli space of twisted sheaves and twisted  Higgs  sheaves on surfaces}\label{sec_moduli_twisted_Higgs_pairs}

In this section we recall the notion of $\mu_r$-gerbe $\XX$-twisted sheaves on $\XX\to X$, and define its stability conditions and moduli stacks.  We mainly follow the arguments in \cite{Lieblich_Duke}.  We also generalized the $\mu_r$-gerbe $\XX$-twisted sheaves on $\XX\to X$ to twisted Higgs sheaves and define their moduli stack. 

For sheaves on algebraic stacks,  see \cite[06TF]{Stack_Project}.  The category of gerbe $\XX$-twisted sheaves on the  stack $\XX$ is naturally a subcategory of the category of sheaves on $\XX$.

\subsection{\'Etale gerbe twisted sheaves}\label{subsec_gerbe_twisted_sheaf}

Let us fix $A$ to be an abelian group scheme over a smooth scheme $X$.  We follow \cite{Lieblich_Duke}, \cite{Caldararu} for the notion of twisted sheaves. 
Let $\pi: \XX\to X$ be a $A$-gerbe and we denote by $[\XX]$ to be the class in $H^2(X, A)$.  Let $E$ be a   sheaf on $\XX$, then one has a sheaf $\pi^*E$ on 
$\widetilde{\XX}$ by the topoi morphism $\pi: \widetilde{\XX}\to X$. Therefore there is a natural right group action
$$\mu: E\times I\XX\to E,$$
see \cite[Lemma 2.1.1.8]{Lieblich_Duke}. The following result is from  \cite[Lemma 2.1.1.13]{Lieblich_Duke}.

\begin{prop}\label{prop_trivial_action_inertia_sheaf}
Let $\XX\to X$ be an $A$-gerbe and $E$ is a  sheaf on $\XX$ such that the inertia action  $E\times I\XX\to E$ is trivial, then $E$ is naturally the pullback of a unique coherent sheaf on $X$ up to isomorphism. 
\end{prop}

The $A$-gerbe $\XX$-twisted coherent sheaves are defined as follows. Let $\chi: A\to \Gm$ be the character morphism.  Let $E$ be a coherent $\sO_{\XX}$-module, the module action $m: \Gm\times E\to E$ yields an associated right action $m^\prime: E\times\Gm\to E$ by 
$m^\prime(s,\varphi)=m(\varphi^{-1},s)$.

\begin{defn}\label{defn_gerbe_twisted_sheaf}
A $d$-fold $\chi$-twisted sheaf on $\XX$ is a  coherent $\sO_{\XX}$-module $E$ such that the natural action $\mu: E\times A\to E$ given by the $A$-gerbe structure
makes the diagram
\[
\xymatrix{
E\times A\ar[r]\ar[d]_{\chi^d}& E\ar[d]^{\id}\\
E\times \Gm\ar[r]^{m^\prime}& E
}
\]
commutes, where $\chi^d(s)=\chi(s)^d$. An $1$-fold twisted coherent sheaf will be called a twisted sheaf. 
\end{defn}
\begin{rmk}
The gerbe $\XX$-twisted sheaves naturally form a fibered  category of the classifying topos $\widetilde{\XX}$, and can be viewed as a fibred category over $X$ via the natural map $\widetilde{\XX}\to X$ of topoi.  This fibred category of $\XX$-twisted sheaves is naturally a stack over $X$. 
\end{rmk}

\begin{rmk}
In general, in \cite{TT-Adv}, for a $A$-gerbe $\XX\to X$, there is a disconnected stack $\widehat{\XX}$ (finite copies of the $A$-gerbe $\XX$) and a $\Gm$-gerbe $\mathbf{c}$ (called the B-field) on  $\widehat{\XX}$ such that 
the category of coherent sheaves on $\XX$ is equivalent to the category of  gerbe $\mathbf{c}$-twisted coherent sheaves on  $\widehat{\XX}$.   
The twisted sheaves of Lieblich are only on one component. 
Since we don't need this in this paper, we fix to Lieblich's definition and moduli on twisted sheaves. 
\end{rmk}

\begin{example}\label{example_Gm_gerbe_twist}
Let $A=\Gm$, and $\chi=\id$.  Then let $\{U_i\}$ be an open covering of $X$.  A  $\XX$-twisted sheaf on $X$ is given by:
\begin{enumerate}
\item a sheaf of modules $E_i$ on each $U_i$;
\item for each $i$ and $j$ an isomorphism of modules 
$g_{ij}: E_j|_{U_{ij}}\stackrel{\cong}{\longrightarrow}E_i|_{U_{ij}}$ such that on $U_{ijk}$, 
$E_k|_{U_{ijk}}\stackrel{\cong}{\longrightarrow}E_k|_{U_{ijk}}$ is equal to the multiplication by the scalar
$a\in \Gm(U_{ijk})$ giving the $2$-cocycle $[\XX]\in H^2(X,\Gm)$. 
\end{enumerate}
\end{example}

We are mainly interested in the $A=\mu_r$-gerbe (or $A=\Gm$-gerbe) twisted sheaves. Let 
$$\chi: \mu_r (\text{or~} \Gm)\to \Gm$$
be the natural inclusion of a subsheaf, a $\chi$-twisted sheaf on $\XX$ is just called an $\XX$-twisted sheaf.  We list the following obvious result for reference:
\begin{prop}\label{prop_stack_gerbe}
A $\Gm$-gerbe $\XX\to X$ is an Artin stack locally of finite presented over $X$.  A $\mu_r$-gerbe $\XX\to X$ is a DM stack. 
\end{prop}

\begin{defn}\label{defn_index_mur_gerbe}
Let $\XX\to X$ be a $\mu_r$-gerbe over $X$. The index $\ind(\XX)$ is the minimal rank of a locally free $\XX$-twisted sheaf over the generic scheme of $X$. 
\end{defn}

Let us recall the period of a  $\mu_r$-gerbe $\XX\to X$ in Definition \ref{defn_optimal_gerbe}.
From \cite[\S 2.2.5]{Lieblich_Duke}, for any scheme $X$, one has $\per(\XX)|\ind(\XX)$. When $X$ is a smooth surface, from a theorem of Je Jong, we have $\per(\XX)=\ind(\XX)$.  Thus over smooth surfaces $X$, the rank of any locally free $\XX$-twisted sheaf is divisible by $\ind(\XX)$.  The studying of moduli of twisted sheaves of rank $r$ on an optimal  $\mu_r$-gerbe $\XX\to X$ is a non-commutative analogue of the Picard scheme, in the sense that such sheaves are essentially rank one right modules over an Azumaya algebra on $X$. 

\subsection{Twisted stability and the moduli stack}\label{subsec_twisted_stability}

We review the twisted stability of $\XX$-twisted sheaves and construction of moduli spaces. 

\subsubsection{Twisted stability of Lieblich}
We first define the support of twisted sheaves. 
\begin{defn}\label{defn_support_twisted_sheaf}
Let $\XX\to X$ be a $\mu_r$-gerbe over $X$ and $E$ a $\XX$-twisted sheaf. The support $\supp(E)$ of $E$ is the closed substack of $\XX$ defined by the kernel of the map 
$\sO_{\XX}\to \sE nd_{\XX}(E)$, which is a quasi-coherent sheaf of ideals. The schematic support of $E$ is the scheme-theoretic image in $X$ of the support of $E$. 
\end{defn}
The support of $E$ is the preimage of the schematic support.  Then as in \cite{Lieblich_Duke} Lieblich defined the torsion subsheaves of a stack $\XX$. One can define a torsion filtration on a twisted sheaf for any algebraic stack.  
 We only use his definition and more details can be found in \cite{Lieblich_Duke}. There exists a unique maximal coherent torsion subsheaf of $E$.  The maximal torsion subsheaf of $E$ is called the torsion subsheaf of 
$E$ and is denoted by $T(E)$. 
A coherent sheaf $E$ is {\em pure} if $T(E)=0$.

Some properties of pure sheaves on stacks $\XX$ can be found in \cite[\S 2.2]{Lieblich_Duke}.   We define the stability condition of twisted sheaves and talk about the moduli spaces. 
We will define the moduli stack of semistable twisted sheaves on the $\mu_r$-gerbe $\XX$.  In order to fix topological invariants, let us take the Chow group $A_*(\XX)$ with $\qq$-coefficient as in \cite{Vistoli}.
Since $\XX$ is a smooth DM stack, $A_*(\XX)\cong A_*(X)$.  For DM stacks, the better way to count the Hilbert polynomial is to use the Grothendieck $K$-group
$K_0(\XX)$ of the abelian category of coherent sheaves.   There exists an orbifold Chern character morphism:
$$K_0(\XX)\stackrel{\cong}{\longrightarrow} H_{\CR}^*(\XX)=H^*(I\XX)$$
given by
$$E\mapsto \widetilde{\Ch}(E)$$
which is an isomorphism as $\qq$-vector spaces. Let $\widetilde{\Td}(T_{\XX})$ be the orbifold Todd class, then the orbifold Grothendieck-Riemann-Roch theorem says:
\begin{equation}\label{eqn_orbifold_GRR}
\chi(\XX, E)=\int_{I\XX}\widetilde{\Ch}(E)\cdot \widetilde{\Td}(T_{\XX})
\end{equation}
see for instance \cite{CIJS}. 

We recall the definition of geometric stability conditions of Lieblich in \cite{Lieblich_Duke}. 

\begin{defn}\label{defn_geometric_Euler_char}
The {\em geometric Euler characteristic} of a class $\alpha\in K_0(\XX)$ is defined as:
$$\chi^g(\alpha):=[I\XX:\XX]\cdot \deg(\Ch(\alpha)\cdot \Td_{\XX}).$$
Let $\sO_X(1)$ be a polarization of $X$. Then the geometric Hilbert polynomial of $\alpha$ is:
$$m\mapsto P_{\alpha}^g(m)=\chi^g(\alpha\otimes\sO_X(m)).$$
\end{defn}

Following \cite[\S 1.2]{HL},  for any $E$ with class $\alpha$, the Hilbert polynomial can be written down as:
$$P_{E}^g(m)=\sum_{i=0}^{\dim E}\alpha_i(E)\frac{m^i}{i!}.$$
Note that here $\alpha_i(E)$ do not need to be integers, since $\XX$ is a DM stack. 
If $E$ has dimension $d$, then the geometric rank is:
$$\rk(E):=\frac{\alpha_d(E)}{\alpha_d(\sO_{\XX})}$$
and the geometric degree of $E$ is given by:
$$\deg(E)=\alpha_{d-1}(E)-\rk(E)\cdot \alpha_{d-1}(\sO_{\XX}).$$

\begin{defn}\label{defn_length_dimension_zero_sheaf}
Let $E$ be a coherent sheaf on $\XX$ of dimension zero, the length of $E$ is defined as:
$$\ell(E)=\chi^g(E).$$
\end{defn}

\begin{defn}\label{defn_twisted_stability}
Let $\XX\to X$ be a $\mu_r$-gerbe over a smooth scheme $X$. 
An $\XX$-twisted sheaf $E$ of dimension $d$ is semistable (resp. stable) if any subsheaf $F\subset E$
$$\alpha_{d}(E)\cdot P_{F}^g\leq (resp. <) \alpha_d(F)\cdot P_E^g.$$
\end{defn}

A semistable coherent $\XX$-twisted sheaf $E$ must be pure.  So let 
$p_E^g=\frac{P_E^g}{\alpha_d(E)}$,  a gerbe $\XX$-twisted sheaf $E$ is semistable if and only if it is pure and for any $F\subset E$, 
$p_F^g\leq p_E^g$.

\begin{defn}\label{defn_slope_definition}
The slope of a coherent $\XX$-twisted sheaf  $E$ of dimension $d$ is defined by:
$$\mu(E):=\frac{\deg(E)}{\rk(E)}.$$
\end{defn}

\begin{defn}\label{defn_slope_stability}
Let $\XX\to X$ be a $\mu_r$-gerbe over a smooth scheme $X$. 
An $\XX$-twisted coherent sheaf $E$ of dimension $d$ is $\mu$-semistable (resp. $\mu$-stable) if $E$ is pure and for any subsheaf $F\subset E$
$$\mu(F)\leq (resp. <) \mu(E).$$
\end{defn}

Lieblich \cite{Lieblich_Duke} defined the moduli stack of semistable (resp. stable) $\XX$-twisted sheaves on $\XX$ using the above stability.

\subsubsection{The moduli stack of $\XX$-twisted semistable sheaves}

We use the geometric Hilbert polynomial to define the moduli stack of gerbe $\XX$-twisted sheaves studied in \cite{Lieblich_Duke}.  First we have:

\begin{prop}\label{prop_Lieblich_openness}( \cite[Corollary 2.3.2.11]{Lieblich_Duke})
Let $\XX\to X$ be a $\mu_r$-gerbe over a smooth scheme $X/R$, and $E$ is an $R$-flat family of coherent  $\XX$-twisted sheaves. The locus of $\mu$-semistable (resp. semistable), 
(resp. geometrically $\mu$-semistable or stable) fibers of $E$ is open in $R$. 
\end{prop}

Let us fix a Hilbert polynomial $P\in\qq[m]$, which is equivalent to fixing a $K$-group class $c\in K_0(\XX)$. 
From Proposition \ref{prop_Lieblich_openness}, there is an algebraic stack 
\begin{equation}\label{eqn_moduli_stack_twisted}
\sM^{ss,\tw}_{\XX/\kappa}: \Sch/\kappa\rightarrow (\text{Groupoids})
\end{equation}
of semistable twisted sheaves of fixed rank $\rk$ and Hilbert polynomial $P$.   This stack contains an open substack $\sM^{s,\tw}_{\XX/\kappa}$ of geometrically stable points.  Let $\Pic_{\XX/\kappa}$ be the Picard stack of $\XX$ parametrizing invertible sheaves on $\XX$.  There is a determinant $1$-morphism:
$$\sM^{\tw}_{\XX/\kappa}(\rk)\to \Pic_{\XX/\kappa}$$
where $\sM^{\tw}_{\XX/\kappa}(\rk)$ is the stack of $\XX$-twisted coherent sheaves of rank $\rk$ (locally of finite type Artin stack).  The stack of semistable $\XX$-twisted sheaves of rank $\rk$, determinant $L\in \Pic_{\XX/\kappa}$ and geometric Hilbert polynomial $P$ is defined by:
$$\sM^{ss,\tw}_{\XX/\kappa}(\rk, L,P):=\sM^{ss,\tw}_{\XX/\kappa}(\rk,P)\times_{\Pic_{\XX/\kappa}, \varphi_L}\kappa$$
where $\varphi_L: \kappa\to \Pic_{\XX/\kappa}$ is a morphism given by $L$.  

For any algebraic stack $\sM$, and suppose that $I\sM\to \sM$ is $\fppf$, then the big \'etale sheaf $\Sh(\sM)$ associated to $\sM$ is an algebraic space and $\sM\to \Sh(\sM)$ is a coarse moduli space, see \cite{LMB}, \cite{Stack_Project}. Therefore we have:

\begin{prop}\label{prop_coarse_moduli_algebraic_stack}(\cite[Proposition 2.3.3.4]{Lieblich_Duke})
The stack 
$$\sM^{s,\tw}_{\XX/\kappa}(\rk, L,P)\rightarrow \Sh(\sM^{s,\tw}_{\XX/\kappa}(\rk, L,P))$$
 is a $\mu_r$-gerbe over the algebraic space of finite type over $\kappa$.
\end{prop}

Next in the end of this section we collect some results for the moduli stacks for essentially trivial gerbes. 
Recall that from Definition \ref{defn_essential_trivial_gerbe},  the $\mu_r$-gerbe $\XX\to X$ is essentially trivial if it is in the kernel of the morphisms
$$H^1(X, \Gm)\rightarrow H^2(X, \mu_r)\rightarrow H^2(X, \Gm).$$
Such a $\mu_r$-gerbe $\XX$ can be given by a line bundle $\sL\to X$, i.e., $\XX=[\Tot(\sL^{\times})/\Gm]$ with the action 
$\lambda(x,e)=(x, \lambda^r\cdot e)$.  And also there exists an exact sequence:
\begin{equation}\label{eqn_Picard_Brauer}
0\to \Pic(X)/r\cdot \Pic(X)\longrightarrow H^2(X, \mu_r)\longrightarrow \Br(X)[r]\to 0.
\end{equation}
We have the following result as in \cite[Theorem 2.3.4.5]{Lieblich_Duke}:

\begin{prop}\label{prop_moduli_essential_trivial_coarse_moduli}
Suppose that $\XX\to X$ is an essentially trivial $\mu_r$-gerbe, then there exists a line bundle $\gamma\in \Pic(X)\otimes\frac{1}{r}\zz$, such that 
the moduli stack $\sM^{ss,\tw}_{\XX/\kappa}(\rk, L,P)$ is isomorphic to the stack 
$\sM^{ss,\gamma}_{X/\kappa}(\rk, L(r\gamma),P)$ of $\gamma$-twisted semistable sheaves on $X$ of rank $\rk$, determinant $L(r\gamma)$ and Hilbert polynomial $P$.
\end{prop}

The twisted stability on $X$ is from \cite{Matsuki-Wentworth}, and we recall it here. 
Let $\sL\in\Pic(X)\otimes\qq$ be a line bundle, and $H=\sO_X(1)$ is a polarization. 
A torsion-free sheaf $E$ on $X$ is said to be $\sL$-twisted $H$-Gieseker semistable for $\sL$ if and only if for $F\subset E$,
$$\frac{\chi(F\otimes \sL\otimes H^m)}{\rk(F)}\leq \frac{\chi(E\otimes \sL\otimes H^m)}{\rk(E)}$$
for $m>>0$.

\subsection{Moduli of $\XX$-twisted sheaves for optimal gerbes}\label{subsec_twisted_moduli_optimal}

In this section we review the case of optimal $\mu_r$-gerbes $\XX\to X$, i.e., for the class $[\XX]\in H^2(X,\mu_r)$, its image 
$o([\XX])$ in $H^2(X,\sO_X^*)$ has order $r$. 
We always assume that $X$ is a smooth projective surface.  
For a smooth surface $X$, and a $\mu_r$-gerbe $\XX\to X$ over $X$, fixing the geometric Hilbert polynomial $P$ of torsion free sheaves $E$ is the same as fixing the data $(\rk, L, c_2)\in H^*(X)$. 
Our goal is to show that the moduli stack $\sM^{s,\tw}_{\XX/\kappa}(\rk, L,c_2)$ is a $\mu_r$-gerbe over the moduli stack of twisted semistable sheaves on $X$.

From \S \ref{subsec_Brauer_optimal_gerbe},  the Brauer-Severi variety $P\to X$ associated with the gerbe  is a projective bundle over $X$ with fiber $\pp^{r-1}$.
From Definition \ref{defn_index_mur_gerbe}, and the paragraph following on, the index  $\ind(\XX)$ of a $\mu_r$-gerbe $\XX\to X$ is the minimal rank that there exists a locally free $\XX$-twisted sheaf over the generic scheme $X$.  For the surface $X$, $\ind(\XX)=\per(\XX)=r$. 
Let $\omega(P):=\delta^\prime(P)\in H^2(X,\mu_r)$. If there is a vector bundle $E\to X$ such that $P=P(E)$, then 
$$\omega(P)=c_1(E)\mod r.$$
For such a projective bundle $P$, there exists a vector bundle $G$ on $P$ which is given by the Euler sequence:
$$0\to \sO_P\longrightarrow G\longrightarrow T_P/X\to 0$$
and $G$ is a non-trivial extension of $T_{P/X}$ by $\sO_P$.

\subsubsection{Yoshioka's moduli of semistable twisted sheaves on $X$}\label{subsubsec_twisted_moduli_optimal_Yoshioka}
In \cite{Yoshioka2}, Yoshioka studied the moduli of twisted semistable sheaves on the Brauer-Severi variety $P\to X$, associated to the 
optimal $\mu_r$-gerbe $[\XX]\in H^2(X,\mu_r)$ such that under the morphism 
$$H^2(X,\mu_r)\longrightarrow H^2(X,\sO_X^*),$$
we have 
$$[\XX]\mapsto \alpha\in H^2(X,\sO_X^*)_{\tor}[r]=H^1(X,PGL_r).$$
Yoshioka \cite{Yoshioka2} defined the subcategory $\Coh(X;P)\subset \Coh(P)$ of coherent sheaves on $P$ to be the subcategory of sheaves 
$$E\in \Coh(X;P)$$
if and only if $$E|_{P_i}\cong p^*(E_i)\otimes \sO_{P_i}(\lambda_i)$$
where $\{U_i\}$ is an open covering of $X$, $P_i=U_i\times\pp^{r-1}$, $E_i\in\Coh(U_i)$; and there exists an equivalence:
$$\Coh(X;P)\stackrel{\cong}{\longrightarrow}\Coh(X,\alpha)$$
by
$$E\mapsto p_*(E\otimes L^{\vee}).$$
Here $\alpha=o([\XX])\in H^2(X,\sO_X^*)_{\tor}$ is the image of the $\mu_r$-gerbe $[\XX]$. The line bundle $L\in\Pic(P)$ is the line bundle with the property that 
$$L|_{p^{-1}(x)}=\sO_{p^{-1}(x)}(-1),$$
and $\Coh(X,\alpha)$ is the category of $\alpha$-twisted coherent sheaves on $X$, see Example \ref{example_Gm_gerbe_twist}.
We call $E\in \Coh(X;P)$ a $P$-sheaf.

A $P$-sheaf $E$ is of dimension $d$ if $p_*E$ is of dimension $d$ on $X$.  Yoshioka defined the Hilbert polynomial 
$$P_E^G(m)=\chi(p_*(G^\vee\otimes E)(m))=\sum_{i=0}^{d}\alpha_i^{G}(E)\cdot \mat{c} m+i\\
i\rix.$$

\begin{defn}\label{defn_G_twisted_semistable}
The $P$-sheaf $E$ of dimension $d$ is $G$-twisted semistable (with respect to $\sO_X(1)$) if $E$ is pure and for any proper subsheaf $F\subset E$
$$\frac{\chi(p_*(G^{\vee}\otimes F)(m))}{\alpha_d^G(F)}\leq \frac{\chi(p_*(G^{\vee}\otimes E)(m))}{\alpha_d^G(E)}$$
for $m>>0$.
\end{defn}
Then we have the following result for the moduli spaces:
\begin{thm}\label{thm_Yoshioka_moduli}(\cite[Theorem 2.1]{Yoshioka2})
There exists a coarse moduli space $M^{P,G}_{X/\kappa}(\rk,P)$ is $S$-equivalence classes of $G$-twisted semi-stable $P$-sheaves $E$ with $G$-twisted Hilbert polynomial 
$P$, and moreover $M^{P,G}_{X/\kappa}(\rk,P)$ is a projective scheme. 
\end{thm}

\subsubsection{Artin-de-Jong's smooth surface $Y$ over $X$ of degree $\ind(\XX)$}

The Brauer-Severi variety $P\to X$ represents the class $[\XX]\in H^1(X, \PGL_r)$.  From \S \ref{subsec_Brauer_optimal_gerbe},  this projective bundle 
$P$ corresponds to an Azumaya algebra $\sA_P$ over $X$.  From \cite{Lieblich_ANT}, one can define the moduli stack $\sM_{\rk}^{\XX}$ of generalized Azumaya algebras on $\XX\to X$, and the moduli stack of 
$\XX$-twisted sheaves $\sM^{\tw}_{\XX/\kappa}(\rk)$ is a cover over  $\sM_{\rk}^{\XX}$.  The construction and the generalization to Higgs version of Azumaya algebras can be found in \cite{Lieblich_ANT} and \cite{Jiang_2019-2}.

Thus any $P$-sheaf $E\in \Coh(X;P)$ from Yoshioka corresponds to a module $\sE$ over the corresponding Azumaya algebra $\sA_P$ and  in \cite{Simpson}  Simpson defined the stability for the underlying sheaf $\sE$ by taking the Hilbert polynomial 
$\chi(X, \sE(m))$.  From \cite[\S 2.3.4]{Lieblich_Duke}, \cite{Yoshioka}, the stability defined by it is the same as the $G$-twisted stability in Definition \ref{defn_G_twisted_semistable}. 

In \cite{Artin-de-Jong}, for the optimal  $\mu_r$-gerbe $\XX\to X$ with $[\XX]\in H^2(X,\sO_X^*)_{\tor}$, Artin-de-Jong constructed a smooth surface $Y$ and a morphism
$$\varphi: Y\to X$$
of degree $r$. 
We only recall the idea and more details can be found in \cite{Artin-de-Jong}. The surface 
$$Y\hookrightarrow \sA_P\otimes L=M_r(\sO_X)\otimes L$$
is cut of of the ``characteristic polynomial" of a general section of $\sA_P\otimes L$ for $L$ a sufficiently ample invertible sheaf on $X$.  In more detail, for any invertible sheaf $L$, the reduced norm yields an algebraic morphism
$$\sA_P\otimes L\to \Sym^{\bullet}L$$
with image the polynomial sections of degree $r$. The zero locus of such a polynomial function on $L^\vee$ gives the finite covering 
$Y\to X$ of degree $r$, and 
\[
\xymatrix{
Y\ar[rr]\ar[dr]&& X\\
&\XX\ar[ur]&
}
\]
the map factors through the gerbe $\XX$. From \cite[Proposition 3.2.2.6]{Lieblich_Duke}, 
\begin{prop}\label{prop_optimal_gerbe_Y}
Let $\XX\to X$ be an optimal gerbe over a surface $X$ such that $\ind(\XX)=\per(\XX)=r$. Then there exists a locally free $\XX$-twisted sheaf of rank $r$, and there exists a finite flat 
surjection of smooth surfaces 
$$\varphi: Y\to X$$
of degree $r$ such that 
\begin{enumerate}
\item there exists an invertible $\XX\times_{X}Y$-twisted sheaf. 
\item for every very ample invertible sheaf $\sO_X(1)$ on $X$, a general member has smooth preimage in $Y$.
\end{enumerate}
\end{prop}

From this proposition, and Lemma 3.2.3.2, Lemma 3.2.2.3 in  \cite{Lieblich_Duke},  Lieblich showed that any $\mu$-semistable twisted sheaf $E$ on $S$, the pullback $\varphi^*E$ is $\mu$-semistable. 
We hope that the moduli stack $\sM_{\XX/\kappa}^{ss,\tw}(\rk,P)$ on $\XX$ of $\XX$-twisted semistable sheaves on $\XX$ is isomorphic to the moduli stack of $\XX\times_{X}Y$-twisted sheaves on $Y$. We don't need this in this paper. 
We conclude:
\begin{thm}\label{thm_moduli_twisted_Yoshioka}
The moduli stack $\sM_{\XX/\kappa}^{s,\tw}(\rk, L, P)$ of stable $\XX$-twisted  sheaves on $\XX$ is  a $\mu_r$-gerbe over the moduli stack 
$M^{P,G}_{X/\kappa}(\rk, L,P)$ of the $G$-twisted semistable sheaves on $X$.
\end{thm}
\begin{proof}
From Proposition \ref{prop_coarse_moduli_algebraic_stack}, the stack 
$$\sM^{s,\tw}_{\XX/\kappa}(\rk, L,P)\rightarrow \Sh(\sM^{s,\tw}_{\XX/\kappa}(\rk, L,P))$$
 is a $\mu_r$-gerbe over the algebraic space of finite type over $\kappa$.
Here $\Sh(\sM^{s,\tw}_{\XX/\kappa}(\rk, L,P))$ is the big \'etale sheaf  associated to the algebraic stack $\sM^{s,\tw}_{\XX/\kappa}(\rk, L,P)$. 
We only need to show that this algebraic space is actually $M^{P,G}_{X/\kappa}(\rk, L,P)$ of the $G$-twisted semistable sheaves on $X$. 
Since from \cite{TT-Adv},  the category of coherent sheaves on $\XX$ is isomorphic to $\Gm$-gerbe twisted sheaves on a rigidified inertia stack $\widehat{\XX}=X\sqcup X\sqcup\cdots\sqcup X$.  The twisted sheaves are exactly corresponding to the $\Gm$-gerbe twisted sheaves on the  first component in  $\widehat{\XX}$, and this $\Gm$-gerbe is the Brauer class of the gerbe $\XX$. Thus  the $\Gm$-gerbe twisted sheaves
are actually the $G$-twisted semistable sheaf on $X$ of Yoshioka.
Therefore 
$\Sh(\sM^{s,\tw}_{\XX/\kappa}(\rk, L,P))$ is just 
$M^{P,G}_{X/\kappa}(\rk, L,P)$. 
\end{proof}

\subsection{Twisted Higgs sheaves and the moduli stack}\label{subsec_twisted_Higgs_sheaves}
From this section we work on a smooth projective surface $S$, and 
let $\SS\to S$ be a $\mu_r$-gerbe on a smooth surface $S$. 
In this section we define $\SS$-twisted Higgs sheaves on a surface $S$.  Although all arguments in this section work for higher dimensional projective schemes, we restrict to surfaces throughout.

Let us fix a polarization $\sO_{S}(1)$ on $S$ for the $\mu_r$-gerbe $\SS\to S$.  Let $\sL\in\Pic(S)$ be a line bundle, recall that  a Higgs sheaf on $S$ is a pair  $(E,\phi)$, where $E$ is a coherent sheaf on $S$ and 
$\phi: E\to E\otimes K_S$ is $\sO_X$-linear map called the ``Higgs field", see \cite{TT1}.  We consider the Higgs pairs on the $\mu_r$-gerbe $\SS$.  Let us fix a line bundle 
$\sL\in\Pic(\SS)$.  Later on we are interested in $\sL=K_{\SS}$, the canonical line bundle $K_{\SS}$.  

\begin{defn}\label{defn_Higgs_pair_XX}
An $\sL$-Higgs sheaf on $\SS$ is a pair $(E,\phi)$, where $E$ is a coherent sheaf on $\SS$, and $\phi: E\to E\otimes \sL$ a $\sO_{\SS}$-linear morphism. 
\end{defn}

Since $\SS$ is a surface DM stack, from \cite[Proposition 2.18]{JP}, there exists an equivalence:
$$\Higg_{\sL}(\SS)\cong \Coh_{c}(\Tot(\sL)),$$
where we denote by $\Tot(\sL)$ the total space of the line bundle $\sL$.  We denote by $\XX:=\Tot(K_{\SS})$, the canonical line bundle of $K_{\SS}$.  This $\XX$ is a Calabi-Yau threefold stack. 
Here $\Coh_{c}(\Tot(\sL))$ is the abelian category of compactly supported sheaves on $\Tot(\sL)$. Let us always fix the following diagram:
\begin{equation}\label{diagram_XX_SS}
\xymatrix{
\XX\ar[r]^{p}\ar[d]_{\pi}& X\ar[d]^{\pi}\\
\SS\ar[r]^{p}& S
}
\end{equation}
where $X=\Tot(K_S)$ is the total space of $K_S$, which is the coarse moduli space of $\XX$. 

\begin{prop}
The DM stack $\XX\to X$ is also a $\mu_r$-gerbe, and the class 
$[\XX]\in H^2(X,\mu_r)\cong H^2(S,\mu_r)$. 
\end{prop}
\begin{proof}
This can be seen \'etale locally since $\XX$ is an affine bundle over $\SS$.
We have
$$H^2(X,\mu_r)\cong H^2(S,\mu_r)$$
and 
$\XX\to X$ is the $\mu_r$-gerbe corresponding to the same class $[\SS]\in H^2(S,\mu_r)$.
\end{proof}

\begin{defn}\label{defn_twisted_Higgs_pair_XX}
An $\SS$-twisted $\sL$-Higgs sheaf on $\SS$ is a pair $(E,\phi)$, where $E$ is a $\SS$-twisted  coherent sheaf on $\SS$ as in Definition \ref{defn_gerbe_twisted_sheaf}, and $\phi: E\to E\otimes \sL$ a $\sO_{\SS}$-linear morphism such that the following diagram 
\[
\xymatrix{
E\times \mu_r\ar[r]\ar[d]_{\chi}& E\ar[d]^{\id}\ar[r]^{\phi}& E\otimes\sL\ar[d]^{\id}\\
E\times \Gm\ar[r]^{m^\prime}& E\ar[r]^{\phi}& E\otimes\sL
}
\]
commutes.
\end{defn}

Here is a result generalizing \cite[Proposition 2.18]{JP} and \cite[Proposition 2.2]{TT1}.
\begin{prop}\label{prop_equivalent_twisted_categories}(\cite[Proposition 2.18]{JP})
There exists an abelian category $\Higg^{\tw}_{K_{\SS}}(\SS)$ of $\SS$-twisted Higgs pairs on $\SS$ and an equivalence:
\begin{equation}\label{eqn_equivalence_twisted_categories}
\Higg^{\tw}_{K_{\SS}}(\SS)\stackrel{\sim}{\longrightarrow} \Coh^{\tw}_{c}(\XX)
\end{equation}
where $\Coh^{\tw}_{c}(\XX)$ is the category of compactly supported $\XX$-twisted coherent sheaves on $\XX$. 
\end{prop}
\begin{proof}
First \cite[Proposition 2.18]{JP} gives an equivalence:
$$\Higg_{K_{\SS}}(\SS)\cong \Coh_{c}(\Tot(\sL)),$$
where if we let $\eta$ be the tautological section of $\pi^*K_{\SS}$ on $\XX$,  which is linear on the fibers and cuts out 
$\SS\subset \XX$. Then 
$$\pi_*(\sO_{\XX})=\bigoplus_{i\geq 0}K_{\SS}^{-i}\cdot \eta^i,$$
and the category of coherent $\sO_{\XX}$-modules is equivalent to the category of coherent $\pi_*(\sO_{\XX})$-modules on $\SS$. 
Any $\pi_*(\sO_{\XX})$-coherent module $\pi_*\sE$ corresponds to a coherent 
$\sO_{\SS}$-module $E$, together with a $\sO_{\SS}$-linear morphism 
$$E\otimes K_{\SS}^{-1}\stackrel{\pi_*\eta}{\longrightarrow}E.$$
Thus a Higgs pair $(E,\phi)$ for $\phi=\pi_*\eta: E\to E\otimes K_{\SS}$ and vice versa. 
The construction is compatible with the twisted condition by the gerbe $\SS\to S$ and $\XX\to X$, and the gerbe $\XX$-twisted sheaf $\sE$ corresponds to 
\[
\xymatrix{
\sE\times \mu_r\ar[r]^-{m}\ar[d]_{\chi}& \sE\ar[d]^{\id}\\
\sE\times \Gm\ar[r]^-{m^\prime}& \sE
}
\]
and the $\XX$-twisted sheaf actions $\chi, m, m^\prime$ correspond to the $\SS$-twisting:
\[
\xymatrix{
E=\pi_*\sE\times \mu_r\ar[r]^--{m}\ar[d]_{\chi}& E\ar[d]^{\id}\\
E=\pi_*\sE\times \Gm\ar[r]^--{m^\prime}& E
}
\]
together with a Higgs field $\phi: E\to E\otimes K_{\SS}$. 
\end{proof}

We define twisted stability for Higgs pairs. 

\begin{defn}\label{defn_twisted_stability_Higgs}
Let $(E,\phi)$ be a torsion free $\SS$-twisted Higgs sheaf on $\SS$. Then we say $(E,\phi)$ is Gieseker semi(stable) if for any subsheaf 
$(F,\phi^\prime)\subset (E,\phi)$ we have 
$$p^g_F(m)\leq (<) p^g_E(m)$$ for $m>>0$.
\end{defn}

For a  torsion free $\SS$-twisted Higgs sheaf $(E,\phi)$ on $\SS$,  we still define 
$$\deg^g(E)=r\cdot \int_{\SS}c_1(E)\cdot c_1(\sO_S(1)).$$
and 
$\mu(E)=\frac{\deg^g(E)}{\rk(E)}$.  Then $(E,\phi)$ is $\mu$-semi(stable) if for any subsheaf 
$(F,\phi^\prime)\subset (E,\phi)$ we have 
$$\mu(F)\leq (<) \mu(E).$$

Let $\N^{ss,\tw}:=\N^{ss,\tw}_{\SS/\kappa}(\rk, L, c_2)$ be the moduli stack of $\SS$-twisted Higgs sheaves with topological data $(\rk, L, c_2)\in H^*(\SS,\qq)$. 
Then $\N^{ss,\tw}$ is an algebraic stack locally of finite type. The stable locus $\N^{s,\tw}:=\N^{s,\tw}_{\SS/\kappa}(\rk, L, c_2)$ is an open substack. 

One can micmic the construction of \cite{Lieblich_Duke} to show that the moduli stack $\N^{ss}$ exists using quot schemes and the GIT theory. We don't need to go through that construction based on the following:

\begin{prop}\label{prop_moduli_Higgs_pairs}(generalizing \cite[Proposition 2.20]{JP})
Based on the Diagram \ref{diagram_XX_SS}, and under the equivalence
$\Higg^{\tw}_{K_{\SS}}(\SS)\stackrel{\sim}{\longrightarrow} \Coh^{\tw}_{c}(\XX)$, the Gieseker (semi)stability of the $\SS$-twisted Higgs sheaves 
$(E,\phi)$ with respect to  $\sO_S(1)$ is equivalent to the Gieseker (semi)stability of the $\XX$-twisted torsion sheaves $\sE_{\phi}$ with respect to $\pi^*\sO_{S}(1)$. 
\end{prop}
\begin{proof}
The proof is similar to  \cite[Proposition 2.20]{JP}. 
We choose a generating sheaf $\Xi$ for $\SS$. 
From the equivalence $\Higg^{\tw}_{K_{\SS}}(\SS)\stackrel{\sim}{\longrightarrow} \Coh^{\tw}_{c}(\XX)$, any $(F,\phi^\prime)\subset (E,\phi)$ is equivalent to the $\XX$-twisted sheaves 
$\sF\subset \sE_{\phi}$ on $\XX$. We have:
$$\chi(\SS, E\otimes \Xi^{\vee})=\chi(\SS, \pi_*\sE_{\phi}\otimes\Xi^{\vee})=\chi(\XX, \sE_{\phi}\otimes\pi^*\Xi^{\vee}).$$
So the Gieseker stability are the same.  Note that we use the modified stability on $\SS$ and $\XX$ by choosing 
$\Xi=\sO_{\SS}\oplus\cdots\oplus\sO_{\SS}(r-1)$.  The  modified stability is the same as the stability on $\SS$ by using the geometric $\chi^g(E)$ by basic calculations.
\end{proof}

\section{Twisted Vafa-Witten invariants-stable case}\label{sec_POT_virtual_Class_VW_Invariants}

Let $\SS\to S$ be a $\mu_r$-gerbe over a smooth projective surface $S$. 
The obstruction theory of $\SS$-twisted Higgs sheaves, and the construction of  the virtual fundamental class on the moduli stack of stable  $\SS$-twisted Higgs sheaves are done in \cite{JK} following the construction of Tanaka-Thomas \cite{TT1}. We only review the basic definitions of the Vafa-Witten invariants.  More details can be found in \cite{TT1}, \cite{JP}.

\subsection{Virtual fundamental class}\label{subsec_Deformation_obstruction}

We fix $\pi: \XX:=\mbox{Tot}(K_{\SS})\to \SS$ to be the projection from the total space of the line bundle $K_{\SS}$ to $\SS$. 

We follow \cite[\S 3.2]{JP}, \cite[\S 3]{TT1} to work on families. 
Let $\SS\to B$ be a family of $\mu_r$-gerby surfaces  $\SS$, i.e., a smooth projective morphism with the fibre a $\mu_r$-gerbe surface, and let $\XX\to B$ be the total space of the a line bundle $K_{\SS/B}$. 
We have a diagram
\[
\xymatrix{
\XX\ar[rr]\ar[dr]&& B\\
&X\ar[ur]&
}
\]
such that $X=\Tot(K_S)$. 
Recall that we  $\N^{s,\tw}:=\N_{\SS/\kappa}^{s,\tw}(\rk, P)$,  the moduli stack of Gieseker stable Higgs pairs on the fibre of $\SS\to B$ with fixed rank $\rk> 0$ and Hilbert polynomial $P$.    
For simplicity, we denote by $\N^{\tw}:=\N^{s,\tw}$.

We pick a (twisted by the $\Gm$-action) universal sheaf $\rE$ over $\N^{\tw}\times_{B}\XX$.  
This universal sheaf may not exist due to the $\Gm$-action, but exists on a finite gerbe over $\N^{\tw}\times_{B}\XX$.
We use the same $\pi$ to represent the projection 
$$\pi: \XX\to \SS; \quad  \pi: \N^{\tw}\times_{B}\XX\to \N^{\tw}\times_{B}\SS.$$
Since $\rE$  is flat over $\N^{\tw}$ and $\pi$ is affine, 
$$\E:=\pi_*\rE \text{~on~} \N^{\tw}\times_{B}\SS$$
is flat over $\N^{\tw}$.  $\E$ is also coherent because it can be seen locally on $\N^{\tw}$, and is also a $\SS$-twisted sheaf. Therefore it defines a classifying map:
$$\Pi: \N^{\tw}\to \sM^{\tw}$$
by
$$\sE\mapsto \pi_*\sE; \quad  (E,\phi)\mapsto E,$$
where $\sM^{\tw}$ is the moduli stack of $\SS$-twisted coherent sheaves on the fibre of $\SS\to B$ with Hilbert polynomial $P$. 
For simplicity, we use the same $\E$ over $\sM^{\tw}\times \SS$ and $\E=\Pi^*\E$ on $\N^{\tw}\times \SS$. 
Let 
$$p_{\XX}:  \N^{\tw}\times_{B}\XX\to \N^{\tw}; \quad   p_{\SS}:  \N^{\tw}\times_{B}\SS\to \N^{\tw}$$
be the projections.  Then from \cite{JP} and \cite{TT1} we have an exact triangle:
\begin{equation}\label{eqn_deformation2}
R\cHom_{p_{\XX}}(\rE, \rE)\stackrel{\pi_*}{\longrightarrow}R\cHom_{p_{\SS}}(\E, \E)\stackrel{[\cdot, \phi]}{\longrightarrow}R\cHom_{p_{\SS}}(\E, \E\otimes K_{\SS}).
\end{equation}
Taking the relative Serre dual of the above exact triangle we get 
$$R\cHom_{p_{\SS}}(\E, \E)[2]\to R\cHom_{p_{\SS}}(\E, \E\otimes K_{\SS/B})[2]\to R\cHom_{p_{\XX}}(\rE, \rE)[3].$$
\begin{prop}(\cite[Proposition 2.21]{TT1})\label{prop_self_dual}
The above exact triangle  is the same as (\ref{eqn_deformation2}), just shifted.
\end{prop} 

Then the exact triangle  (\ref{eqn_deformation2}) fits into the following commutative diagram 
(\cite[Corollary 2.22]{TT1}):
\[
\xymatrix{
R\cHom_{p_{\SS}}(\E, \E\otimes K_{\SS/B})_{0}[-1]\ar[r]\ar@{<->}[d] &R\cHom_{p_{\XX}}(\rE, \rE)_{\perp}
\ar[r]\ar@{<->}[d] & R\cHom_{p_{\SS}}(\E, \E)_{0}\ar@{<->}[d]\\
R\cHom_{p_{\SS}}(\E, \E\otimes K_{\SS/B})[-1]\ar[r]\ar@{<->}[d]_{\id}^{\tr} &R\cHom_{p_{\XX}}(\rE, \rE)
\ar[r]\ar@{<->}[d] & R\cHom_{p_{\SS}}(\E, \E)\ar@{<->}[d]_{\id}^{\tr}\\
Rp_{\SS *}K_{\SS/B}[-1]\ar@{<->}[r]& Rp_{\SS *}K_{\SS/B}[-1]\oplus Rp_{\SS *}\sO_{\SS}\ar@{<->}[r]& Rp_{\SS *}\sO_{\SS}
}
\]
where $(-)_0$ denotes the trace-free Homs.  The $R\cHom_{p_{\XX}}(\rE, \rE)_{\perp}$ is the co-cone of the middle column and it will provide the symmetric obstruction theory of the moduli space $\N_{L,\tw}^{\perp}$ of stable trace free fixed determinant Higgs pairs.

In Appendix of  \cite{JP}  the authors generalized the perfect obstruction theory  as in \cite[\S 5]{TT1} of the moduli space of stable Higgs sheaves on the surface $S$ to 
surface DM stacks.  Since the $\mu_r$-gerbe $\SS\to S$ is a surface DM stack, and the homological algebra and deformation obstruction theory in \cite[\S 5]{TT1}, \cite[Appendix]{JP} work for twisted sheaves, we skip the detail. 
The truncation $\tau^{[-1,0]}R\cHom_{p_{\XX}}(\rE, \rE)$ defines a symmetric perfect obstruction theory on the moduli space $\N^{\tw}$.

We consider the natural $\Gm$-action on the total space $K_{\SS}/B$ with weight one on the fiber. 
The obstruction theory is naturally $\Gm$-equivariant. 
From \cite{GP}, the $\Gm$-fixed locus $\N^{\tw,\Gm}$ inherits a perfect obstruction theory
\begin{equation}\label{eqn_deformation_obstruction_fixed_locus}
\left(\tau^{[-1,0]}(R\cHom_{p_{\XX}}(\rE,\rE)[2])\Tt^{-1}\right)^{\Gm}\to  \ll_{\N^{\tw,\Gm}}
\end{equation}
by taking the fixed part.  Therefore it induces a virtual fundamental cycle 
$$[(\N^{\tw})^{\Gm}]^{\vir}\in H_*(\N^{\tw,\Gm}).$$
The virtual normal bundle is given 
$$N^{\vir}:=\left(\tau^{[-1,0]}(R\cHom_{p_{\XX}}(\rE,\rE)[2]\Tt^{-1})^{\mov}\right)^{\vee}
=\tau^{[0,1]}(R\cHom_{p_{\XX}}(\rE,\rE)[1])^{\mov}$$
which is the derived dual of the moving part.

The virtual  localized invariant is given by the following
$$\int_{[(\N^{\tw})^{\Gm}]^{\vir}}\frac{1}{e(N^{\vir})}.$$
Represent $N^{\vir}$ as a $2$-term complex $[E_0\to E_1]$ of locally free $\Gm$-equivariant sheaves with non-zero weights and define 
$$e(N^{\vir}):=\frac{c_{\topo}^{\Gm}(E_0)}{c_{\topo}^{\Gm}(E_1)}\in H^*((\N^{\tw})^{\Gm}, \zz)\otimes \qq[t, t^{-1}],$$
where $t=c_1(\Tt)$ is the generator of $H^*(B\Gm)=\zz[t]$, and $c_{\topo}^{\Gm}$ denotes the $\Gm$-equivariant top Chern class lying in $H^*((\N^{\tw})^{\Gm}, \zz)\otimes_{\zz[t]} \qq[t, t^{-1}]$. 

\begin{defn}\label{defn_twisted_VW1}
Let $\SS\to S$ be a $\mu_r$-gerbe over a smooth projective surface $S$. Fixing a geometric Hilbert polynomial $P$ associated with $\SS$. Let $\N$ be the moduli space of $\SS$-twisted stable Higgs sheaves with  Hilbert polynomial $P$.  Then the primitive twisted Vafa-Witten invariants of $\SS$ is defined as:
$$\widetilde{\VW}^{\tw}(\SS):=\int_{[(\N^{\tw})^{\Gm}]^{\vir}}\frac{1}{e(N^{\vir})}\in \qq.$$
\end{defn}

\begin{rmk}
For the $\mu_r$-gerbe $\SS\to S$, 
we have 
$$\Ext^\bullet_{\XX}(\sE_\phi, \sE_\phi)=H^{\bullet-1}(K_{\SS})\oplus H^{\bullet}(\sO_{\SS})\oplus 
\Ext^{\bullet}_{\XX}(\sE_\phi, \sE_\phi)_{\perp},$$
where $\Ext^{\bullet}_{\XX}(\sE_\phi, \sE_\phi)_{\perp}$ is the trace zero part with determinant $L\in \Pic(\SS)$. 
Hence the obstruction sheaf in the obstruction theory has a trivial summand $H^2(\sO_{\SS})$. 
So $[(\N^{\tw})^{\Gm}]^{\vir}=0$ is $h^{0,2}(S)>0$.  If $h^{0,1}(S)\neq 0$, then tensoring with flat line bundle makes the obstruction theory invariant. Therefore the integrand is the  pullback from $\N^{\tw}/\Jac(\SS)$, which is a lower dimensional space, hence zero. 
\end{rmk}

\subsection{$\SU(\rk)$-twisted Vafa-Witten invariants}\label{subsec_twisted_VW_invariants}
Instead we work on the moduli stack of $\SS$-twisted stable sheaves and Higgs sheaves with fixed determinant and trace zero. 
Let us now fix $(L, 0)\in \Pic(\SS)\times \Gamma(K_{\SS})$, and let $\N^{\perp,\tw}_{L}$ be the fibre of 
$$\N^{\tw}/\Pic(\SS)\times\Gamma(K_{\SS}).$$
Then moduli space $\N^{\perp,\tw}_{L}$ of stable Higgs sheaves $(E,\phi)$ with $\det(E)=L$ and trace-free $\phi\in \Hom(E,E\otimes K_{\SS})_0$ admits a symmetric obstruction theory 
$$R\cHom_{p_{\XX}}(\rE, \rE)_{\perp}[1]\Tt^{-1}\longrightarrow \ll_{\N^{\perp,\tw}_{L}}$$
from \cite[Proposition A.6]{JP}. 

\begin{defn}\label{defn_SU_twisted_VW_invariants}
Let $\SS\to S$ be a $\mu_r$-gerbe over a smooth projective surface $S$. Fixing a geometric Hilbert polynomial $P$ associated with $\SS$. Let $\N^{\perp,\tw}_{L}$ be the moduli space of $\SS$-twisted stable Higgs sheaves with Hilbert polynomial $P$.  Then define
$$\VW^{\tw}(\SS):=\int_{[(\N^{\perp,\tw}_{L})^{\Gm}]^{\vir}}\frac{1}{e(N^{\vir})}.$$
We call them the twisted  Vafa-Witten invariants for the gauge group $\SU(\rk)/\zz_{\rk}$.
\end{defn}

\begin{rmk}
In \cite[\S 6]{TT1}, Tanaka-Thomas defined the Vafa-Witten invariants $\VW(S):=\int_{[(\N^{\perp}_{L})^{\Gm}]^{\vir}}\frac{1}{e(N^{\vir})}$ for smooth projective surface $S$, here $N_L^{\perp}$ is the moduli space of stable Higgs sheaves on $S$ with fixed data $(\rk, L, c_2)$, whose structure group is $SL_{\rk}(\kappa)$. This is the Vafa-Witten invariants for the gauge group $\SU(\rk)$.  From \cite{Lieblich_ANT}, for an optimal  $\mu_r$-gerbe $\SS$-twisted Higgs sheaf 
$(E,\phi)$, the $\SS$-twisted torsion free sheaf $E$ corresponds to an Azumaya algebra on $S$, and therefore determines a $\PGL_{\rk}$-bundle on $S$,  see \cite{Jiang_2019-2} for twisted Higgs sheaves.  In gauge group level, this corresponds to the gauge group $\SU(\rk)/\zz_{\rk}$.
\end{rmk}

For the $\mu_r$-gerbe $\SS$, it maybe better to fix the K-group class $\mathbf{c}\in K_0(\SS)$ such that the Hilbert polynomial of $\mathbf{c}$ is $P$.  Then 
$\VW^{\tw}_{\mathbf{c}}(\SS)=\int_{[(\N^{\perp,\tw}_{L})^{\Gm}]^{\vir}}\frac{1}{e(N^{\vir})}$ is Vafa-Witten invariant corresponding to $\mathbf{c}$. 
The inertia stack $I\SS$ is $r$ copies of the gerbe $\SS$, and the K-group class $\mathbf{c}\in K_0(\SS)$ is determined by the orbifold Grothendieck-Riemann-Roch formula (\ref{eqn_orbifold_GRR}), therefore one can fix a data
$(\rk, L, c_2)\in H^*(\SS,\qq)$.  Thus we have 
\begin{equation}\label{eqn_SU_twisted_VW_invariants}
\VW^{\tw}_{(\rk,L,c_2)}(\SS):=\int_{[(\N^{\perp,\tw}_{L})^{\Gm}]^{\vir}}\frac{1}{e(N^{\vir})}.
\end{equation}

\subsection{The $\Gm$-fixed loci}\label{subsec_CStar_fixed_locus}

The group $\Gm$ naturally acts on $K_{\SS}$ by scaling the fiber, therefore it induces an action on the moduli stack $\N^{\perp,\tw}_{L}$. 
We discuss the $\Gm$-fixed loci for the moduli stack $\N^{\perp,\tw}_{L}$. 

\subsubsection{Case I-Instanton Branch:}\label{subsubsec_first_type} 
For the Higgs pairs $(E,\phi)$ such that $\phi=0$, and is $\Gm$-fixed, the $\SS$-twisted sheaf $E$ must be stable.    Therefore the $\Gm$-fixed locus is exactly the moduli space $\sM_{\SS/\kappa}^{s,\tw}:=\sM^{s,\tw}_{\SS/\kappa}(\rk,L,c_2)$ of 
$\SS$-twisted Gieseker stable sheaves on $\SS$ with rank $\rk$, fixed determinant $L$ and second Chern class $c_2$.  The exact triangle in (\ref{eqn_deformation2}) splits the obstruction theory
$$R\cHom_{p_{\XX}}(\rE, \rE)_{\perp}[1]\Tt^{-1}\cong R\cHom_{p_{\SS}}(\E, \E\otimes K_{\SS})_{0}[1]\oplus 
R\cHom_{p_{\SS}}(\E, \E)_{0}[2]\Tt^{-1}$$
where $\Tt^{-1}$ represents the moving part of the $\Gm$-action. Then the $\Gm$-action induces a perfect obstruction theory 
$$E_{\sM}^{\bullet}:=R\cHom_{p_{\SS}}(\E, \E\otimes K_{\SS})_{0}[1]\to \ll_{\sM_{\SS/\kappa}^{s,\tw}}.$$
The virtual normal bundle 
$$N^{\vir}=R\cHom_{p_{\SS}}(\E, \E\otimes K_{\SS})_{0}\Tt=E_{\sM}^{\bullet}\otimes \Tt[-1].$$
So the twisted Vafa-Witten  invariant contributed from $\sM_{\SS/\kappa}^{s,\tw}$ (we can let $E_{\sM}^{\bullet}$ is quasi-isomorphic to $E^{-1}\to E^0$) is: 
\begin{align*}
\int_{[\sM_{\SS/\kappa}^{s,\tw}]^{\vir}}\frac{1}{e(N^{\vir})}&=\int_{[\sM_{\SS/\kappa}^{s,\tw}]^{\vir}}\frac{c_s^{\Gm}(E^0\otimes \Tt)}{c_r^{\Gm}(E^{-1}\otimes \Tt)}\\
&=\int_{[\sM_{\SS/\kappa}^{s,\tw}]^{\vir}}\frac{c_s(E^0)+\Tt c_{s-1}(E^0)+\cdots}{c_r(E^{-1})+\Tt c_{r-1}(E^{-1})+\cdots}
\end{align*}
Here we assume $r$ and $s$ are the ranks of $E^{-1}$ and $E^0$ respectively, and $r-s$ is the virtual dimension of 
$\sM_{\SS/\kappa}^{s,\tw}$.  By the virtual dimension consideration, only $\Tt^0$ coefficient contributes and we may let $\Tt=1$, so
\begin{align}\label{eqn_virtual_Euler_number}
\int_{[\sM_{\SS/\kappa}^{s,\tw}]^{\vir}}\frac{1}{e(N^{\vir})}&=\int_{[\sM_{\SS/\kappa}^{s,\tw}]^{\vir}}\Big[\frac{c_{\bullet}(E^0)}{c_{\bullet}(E^{-1})}\Big]_{\vd}
=\int_{[\sM_{\SS/\kappa}^{s,\tw}]^{\vir}}c_{\vd}(E_{\sM}^{\bullet})\in \zz.
\end{align}
This is the signed virtual Euler number of Ciocan-Fontanine-Kapranov/Fantechi-G\"ottsche, see \cite{FG}.  The general case of this invariant is studied in \cite{JT}.
Hence we have the following result: 
\begin{prop}\label{prop_K_Sleqzero_fixed_locus}
 If $\deg K_{S}\leq 0$, then any stable $\SS$-twisted $\Gm$-fixed Higgs pair 
$(E,\phi)$ has Higgs field $\phi=0$. Therefore if we fix $c=(\rk,L,c_2)$, the twisted invariant
$\VW^{\tw}_{c}(\SS)$ is the same as the signed virtual Euler number in (\ref{eqn_virtual_Euler_number}).
\end{prop}
\begin{proof}
We use the geometric Gieseker stability to replace the general stability in the proof of \cite[Proposition 7.4]{TT1}. Given a $\SS$-twisted Higgs sheaf 
$(E,\phi)$ which is $\Gm$-fixed,  the kernel  $\ker(\phi)$ and $\Im(\phi)$ are both $\phi$-invariant twisted subsheaves.  From stability we have 
$$p^g_E(m)< p^g_{\Im(\phi)}(m)<p^g_{E\otimes K_{\SS}}(m)$$
for $m>>0$ unless $\Im(\phi)=0$ or $E\otimes K_{\SS}$.
The degree $\deg(K_S)\leq 0$ implies that $\deg(K_{\SS})\leq 0$, therefore 
$p^g_{E\otimes K_{\SS}}(m)<p^g_{E}(m)$ for $m>>0$.  So either 
$\Im(\phi)=0$ or the whole  $E\otimes K_{\SS}$.
The Higgs pair has fixed determinant and trace zero, and it can not be the whole  $E\otimes K_{\SS}$. Therefore $\phi=0$.
\end{proof}

Also we have:

\begin{prop}\label{prop_semistable_Higgs_Gm}
If $\deg K_{S}< 0$, then any semistable $\Gm$-fixed $\SS$-twisted  Higgs pair $(E,\phi)$ has Higgs field $\phi=0$. 
\end{prop}
\begin{proof}
This is the same as Proposition \ref{prop_K_Sleqzero_fixed_locus}.
\end{proof}

\subsubsection{Case II-Monopole Branch:}\label{subsec_second_fixed_loci}
The second component $\sM^{(2)}$ corresponds to the case that in the $\Gm$-fixed $\SS$-twisted stable Higgs sheaf $(E,\phi)$, the Higgs field $\phi\neq 0$. 
Let $(E,\phi)$ be a $\Gm$-fixed $\SS$-twisted  stable Higgs pair.  Since the corresponding $\Gm$-fixed $\XX$-twisted stable sheaves $\sE_{\phi}$ on $\XX $ are simple,  
we  use \cite[Proposition 4.4]{Kool}, \cite{GJK} to make this stable sheaf 
$\Gm$-equivariant.   But we work on twisted sheaves.  We have  the twisted version of Martijn  \cite[Proposition 4.4]{Kool}.
\begin{lem}\label{lem_twisted_Martijn}
Let $\sE$ be a $\XX$-twisted stable sheaf on $\XX\to X$ which is $\Gm$-invariant, then it is $\Gm$-equivariant. 
\end{lem}
\begin{proof}
Since  a $\XX$-twisted stable sheaf $E$ is a coherent sheaf on the DM stack $\XX$, Martijn's result holds for DM stacks. 
\end{proof}

The cocycle condition in the  $\Gm$-equivariant definition for the $\SS$-twisted Higgs sheaf $(E,\phi)$ corresponds to a $\Gm$-action 
$$\psi: \Gm\to \Aut(E)$$
such that 
\begin{equation}\label{eqn_cocycle_condition}
\psi_{t}\circ \phi\circ \psi_t^{-1}=t\phi
\end{equation}  
With respect to the $\Gm$-action on $E$,  the torsion free $\SS$-twisted sheaf $E$ splits into a direct sum of eigenvalue subsheaves
$E=\oplus_{i}E_i$
where $E_i$ is the weight space such that $t$ has by $t^i$, i.e., $\psi_t=\mbox{diag}(t^i)$. 
Conversely, for any twisted pair $(E,\phi)$, if the $\Gm$ action on $E$ induces a weight one action on the Higgs field $\phi$, then it is fixed by the original $\Gm$ action. 

The fact that the  $\Gm$-action on the canonical line bundle $K_{\SS}$ has weight $-1$ makes the Higgs field $\phi$ decrease the weights, and it maps the lowest weight torsion subsheaf to zero, hence zero by stability. 
So each $E_i$ is torsion free and have rank $> 0$.  Thus $\phi$ acts blockwise through morphisms
$$\phi_i: E_i\to E_{i-1}.$$
These are flags of torsion-free twisted sheaves on $\SS$.

We are working on $\mu_r$-gerbes  $\SS\to S$ over the surface $S$, the decomposition of   $\SS$-twisted sheaves $E=\oplus_{i}E_i$ depends on the gerbe structure.   If the $\mu_r$-gerbe $\SS\to S$ is essentially trivial,  then in the decomposition $E=\oplus_{i}E_i$, we believe that the case that
all $E_i$ have rank one can happen.  Then $\phi_i$ define nesting of ideals and this is the nested Hilbert scheme on $\SS$,  see \cite[\S 8]{TT1}. 

If the $\mu_r$-gerbe $\SS\to S$ is not essentially trivial, for instance, it is optimal, then $\ind(\XX)=r$, which means that the minimal rank of locally free $\SS$-twisted sheaf is rank $r$.  Then in this case the decomposition 
$E=\oplus_{i}E_i$ must satisfies that each $E_i$ has rank at least $r$. 

\begin{rmk}
It is interesting to study the decomposition for the $\mu_r$-gerbes over quintic surfaces, and compare with the results in \cite[\S 8]{TT1}.
\end{rmk}

\subsection{Invariants defined by the Behrend function}\label{subsec_Behrend_function}

In \cite{Behrend},  over any DM stack $X$ Behrend constructed an integer value constructible function $\nu_X$, which is called the Behrend function.  If a scheme or a DM stack $X$ admits a symmetric obstruction theory
$E_{X}^{\bullet}$ and $X$ is proper then from \cite[Theorem 4.18]{Behrend}, 
$$\int_{[X]^{\vir}}1=\chi(X,\nu_X).$$
This shows that the Donaldson-Thomas invariants for Calabi-Yau threefolds are weighted Euler characteristics. 
More details of Behrend function and Behrend's theorem can be found in \cite{Behrend}, \cite{Jiang}.

Let $\SS\to S$ be a $\mu_r$-gerbe over a smooth surface $S$.  On the moduli stack $\N_{L}^{\perp,\tw}$ (the moduli stack of $\SS$-twisted stable Higgs sheaves with data $(\rk,L,c_2)$), we have the Behrend function 
$$\nu_{\N}: \N_{L}^{\perp,\tw}\to \zz.$$
The weighted Euler characteristic of $\N^{\tw}$ is defined by:
$$\chi(\N_{L}^{\perp,\tw}, \nu_{\N})=\sum_{i\in \zz}i\cdot \chi(\nu_{\N}^{-1}(i)).$$
The Behrend function is constant on the nontrivial $\Gm$-orbit for the $\Gm$ action on $\N^{\tw}$, thus we have the localized Behrend function invariant
$$\chi(\N_{L}^{\perp,\tw}, \nu_{\N})=\chi((\N_{L}^{\perp,\tw})^{\Gm}, \nu_{\N}|_{(\N_{L}^{\perp,\tw})^{\Gm}}).$$

\begin{defn}\label{defn_SU_twisted_vw_invariants}
Let $\SS\to S$ be a $\mu_r$-gerbe over a smooth projective surface $S$. Fixing a geometric Hilbert polynomial $P$ associated with $\SS$. Let $\N^{\perp,\tw}_{L}$ be the moduli space of $\SS$-twisted stable Higgs sheaves with Hilbert polynomial $P$.  Then define
$$\vw^{\tw}(\SS):=\chi(\N_{L}^{\perp,\tw}, \nu_{\N}).$$
We call them the small twisted Vafa-Witten invariants for the gauge group $\SU(\rk)/\zz_{\rk}$.
\end{defn}

\begin{rmk}
From calculations in \cite[\S 8]{TT1}, \cite{TT2}, Tanaka-Thomas have confirmed that the right  Vafa-Witten invariants $\VW(S)$ for a surface $S$ is $\VW(S):=\int_{[(\N^{\perp}_{L})^{\Gm}]^{\vir}}\frac{1}{e(N^{\vir})}$.  They also defined the small Vafa-Witten invariants $\vw(S):=\chi(\N_{L}^{\perp}, \nu_{\N})$. The invariants $\vw=\VW$ for Fano and K3 surfaces. But they are not the same for general type surfaces.  We also define both invariants, and later on do the calculations for the small twisted $\vw$ invariants. 
\end{rmk}

\section{Joyce-Song twisted stable pairs and generalized twisted Vafa-Witten invariants}\label{sec_Joyce-Song}

Still fix a $\mu_r$-gerbe $\SS\to S$ over a smooth projective surface $S$. 
In this section we include the Joyce-Song method to count the semistable $\SS$-twisted Higgs sheaves. 
We then define Vafa-Witten invariants for strictly semistable 
$\SS$-twisted Higgs sheaves.  Although this is a parallel theory comparing with Tanaka-Thomas in \cite{TT2}, we include the detail construction since we will use the generalized twisted Vafa-Witten invariants to define the $\SU(r)/\zz_r$-Vafa-Witten invariants for the Langlands dual group $\SU(r)/\zz_r$. We also prove the S-duality conjecture for $\pp^2$ in rank two.

We fix our 
$$\XX=\Tot(K_{\SS})\to X$$
which is a $\mu_r$-gerbe over $X=\Tot(K_S)$ with class $[\XX]\in H^2(X,\mu_r)$.

\subsection{Background to count semistable objects}\label{subsec_count_semistable}

Let $(E,\phi)$ be a $\SS$-twisted Higgs sheaf on a $\mu_r$-gerbe $\SS\to S$. Recall that from Proposition \ref{prop_equivalent_twisted_categories}, there is a spectral $\XX$-twisted sheaf
$\sE_{\phi}$ on $\XX=\Tot(K_{\SS})$ with respect to the polarization $\sO_X(1)=\pi^*\sO_S(1)$.

From Diagram \ref{diagram_XX_SS} and Proposition \ref{prop_moduli_Higgs_pairs},  the Gieseker (semi)stability of the twisted Higgs pair $(E,\phi)$ is equivalent to the Gieseker 
 (semi)stability of $\sE_{\phi}$. 
 
 Recall if we fix on the surface $S$, the following data:
 $$\rank(E)=\rk; \quad c_1(E)=c_1; \quad c_2(E)=c_2.$$
 for $(E,\phi)$, then from \cite[\S 2.1]{TT2}, let $\iota: S\hookrightarrow X$ be the inclusion,  we have:
 $$
 \begin{cases}
 c_1(\sE_{\phi})=\rk\cdot [S]\\
 c_2(\sE_{\phi})=-\iota_*(c_1+\frac{\rk(\rk+1)}{2}c_1(S))\\
 c_3(\sE_{\phi})=\iota_*(c_1^2-2c_2+(\rk+1)c_1\cdot c_1(S)+\frac{\rk(\rk+1)(\rk+2)}{6}c_1(S)^2)
 \end{cases}
 $$
 in $H_c^*(X,\qq)$. Here $[S]$ represents the Poincar\'e dual. Let $\SS\to S$ be a $\mu_r$-gerbe over a smooth projective surface $S$.  If the $\mu_r$-gerbe 
 $\SS$ is trivial, i.e., $\SS\cong [S/\mu_r]$ with $\mu_r$ globally trivial action, then one can use the above data to count $(E,\phi)$. 
 
 In general for the $\mu_r$-gerbe $\SS\to S$, one can use the $K$-group class
 $c\in K_0(\SS)\cong K_0(\XX)$ to measure the topological data.  By the orbifold Grothendieck-Riemann-Roch formula for stacks in (\ref{eqn_orbifold_GRR}), $c$ is determined by the Chen-Ruan cohomology class in 
 $H^*_{\CR}(\SS)\cong H^*_{\CR}(\XX)$.  For $\XX=\Tot(K_{\SS})$, we use the modified Hilbert polynomial 
 $P_{\Xi}(\sE_{\phi}(m))$ and  $p_{\Xi}(\sE_{\phi}(m))$. We only care about classes (charges) $\alpha,\beta\in K_0(\XX)$.  We follow from \cite{TT2} to assume that 
 \begin{equation}\label{eqn_assumption}
  p_{\Xi}(\beta(m))=\text{constant}\cdot   p_{\Xi}(\alpha(m)) \Rightarrow \beta=\text{constant}\cdot \alpha.
 \end{equation}
 to make the Joyce-Song wall crossing formula easy. 
 Aso for $\alpha,\beta\in K_0(\XX)$, the Calabi-Yau condition is satisfied:
 $$\chi(\alpha,\beta)=\sum_i \dim \Ext^i(\alpha, \beta)=0$$
 which is skew-symmetric. 
 
 For the Calabi-Yau threefold DM stack $\XX$, and a $K$-group class $\alpha\in K_0(\XX)$,  there is a moduli stack $\sM^{ss,\tw}_{\XX/\kappa}(\rk, L, c_2)$ of $S$-equivalence classes of 
 $\XX$-twisted semistable torsion sheaves on $\XX$, which is isomorphic to the moduli stack $\sM^{ss,\tw}_{\SS/\kappa}(\rk, L, c_2)$ of 
 $\SS$-twisted semistable Higgs sheaves on $\SS$.
 We would use Joyce-Song techniques to count them.

\subsection{Hall algebras and the integration map}\label{subsec_Hall_algebra}

In this section we review the definition and construction of the motivic Hall algebra of Joyce and Bridgeland in \cite{Joyce07}, \cite{Bridgeland10}. Then we review the integration map. 
We briefly review the notion of motivic Hall algebra in \cite{Bridgeland10}, more details can be found in \cite{Bridgeland10}, \cite{Joyce07}. 

\begin{defn}\label{defn_Grothendieck_stack}
The Grothendieck ring of stacks  $K(\St/\kappa)$ is defined to be the $\kappa$-vector space spanned by isomorphism classes of Artin stacks of finite type over $\kappa$ with affine stabilizers, modulo the relations:
\begin{enumerate}
\item for every pair of stacks $\XX_1$ and $\XX_2$ a relation:
$$[\XX_1\sqcup\XX_2]=[\XX_1]+[\XX_2];$$
\item for any geometric bijection $f: \XX_1\to \XX_2$, $[\XX_1]=[\XX_2]$;
\item for any Zariski fibrations $p_i: \XX_i\to \YY$ with the same fibers, $[\XX_1]=[\XX_2]$.
\end{enumerate}
\end{defn}
Let $[\aaa_{\kappa}^1]=\ll$ be the Lefschetz motive.  If $S$ is a stack of finite type over $\kappa$, we define the relative Grothendieck ring of stacks $K(\St/S)$ as follows:

\begin{defn}\label{relative:Grothendieck:group}
The relative Grothendieck ring of stacks  $K(\St/\kappa)$ is defined to be the $\kappa$-vector space spanned by isomorphism classes of morphisms
$$[\XX\stackrel{f}{\rightarrow}S],$$
with $\XX$ an Artin stack over $S$ of finite type with affine stabilizers, modulo the following relations:
\begin{enumerate}
\item for every pair of stacks $\XX_1$ and $\XX_2$ a relation:
$$[\XX_1\sqcup\XX_2\stackrel{f_1\sqcup f_2}{\longrightarrow}S]=[\XX_1\stackrel{f_1}{\rightarrow}S]+[\XX_2\stackrel{f_2}{\rightarrow}S];$$
\item for any diagram:
$$
\xymatrix{
\XX_1\ar[rr]^{g}\ar[dr]_{f_1}&&\XX_2\ar[dl]^{f_2}\\
&S,&
}
$$
where $g$ is a 
geometric bijection, then $[\XX_1\stackrel{f_1}{\rightarrow}S]=[\XX_2\stackrel{f_2}{\rightarrow}S]$;
\item for any pair of Zariski fibrations 
$$\XX_1\stackrel{h_1}{\rightarrow} \YY; \quad \XX _2\stackrel{h_2}{\rightarrow} \YY;$$
with the same fibers, and $g: \YY\to S$, a relation 
$$[\XX_1\stackrel{g\circ h_1}{\longrightarrow} S]=[\XX_2\stackrel{g\circ h_2}{\longrightarrow} S].$$
\end{enumerate}
\end{defn}
The motivic Hall algebra  in \cite{Joyce07} and \cite{Bridgeland10} is defined as follows.
Let $\sM_{\XX/\kappa}^{\tw}$ be the moduli stack of $\XX$-twisted coherent sheaves on $\XX$. It is an algebraic stack, locally of finite type over $\kappa$, see \cite{Lieblich_Duke}. The motivic Hall algebra is the vector space 
$$H(\sA^{\tw})=K(\St/\sM^{\tw}_{\XX/\kappa})$$
equipped with a non-commutative product given by the rule:
$$[\XX_1\stackrel{f_1}{\longrightarrow} \sM^{\tw}_{\XX/\kappa}]\star[\XX_2\stackrel{f_2}{\longrightarrow} \sM^{\tw}_{\XX/\kappa}]=[\ZZ\stackrel{b\circ h}{\longrightarrow} \sM^{\tw}_{\XX/\kappa}],$$
where $h$ is defined by the following Cartesian square:
\[
\xymatrix{
\ZZ\ar[r]^{h}\ar[d]&\sM^{(2)}\ar[r]^{b}\ar[d]^{(a_1,a_2)}&\sM^{\tw}_{\XX/\kappa} \\
\XX_1\times\XX_2\ar[r]^--{f_1\times f_2}& \sM^{\tw}_{\XX/\kappa}\times \sM^{\tw}_{\XX/\kappa},&
}
\]
with $\sM^{(2)}$ the stack of short exact sequences in $\sA^{\tw}$, and 
the maps $a_1, a_2, b$ send a short exact sequence
$$0\rightarrow A_1\longrightarrow B\longrightarrow A_2\rightarrow 0$$
to sheaves $A_1$, $A_2$, and $B$ respectively. Then $H(\sA^{\tw})$ is an algebra over 
$K(\St/\kappa)$. 

Then on the Hall algebra  $H(\sA^{\tw})$, we are interested in the elements:
$$\mathbb{1}_{\N^{ss,\tw}_{\alpha}}: \N^{ss,\tw}_{\alpha}\hookrightarrow \Higg^{\tw}_{K_{\SS}}(\SS)\cong \Coh_c^{\tw}(\XX),$$
where $\N^{ss,\tw}_{\alpha}$ is the stack of Gieseker semistable $\XX$-twisted Higgs sheaves $(E,\phi)$ of class $\alpha$, and 
$\mathbb{1}_{\N^{ss,\tw}_{\alpha}}$ is the inclusion into the stack of all twisted Higgs pairs on $\SS$. 
We consider its ``logarithm":
\begin{equation}\label{eqn_epsilon_alpha}
\epsilon(\alpha):=\sum_{\substack{\ell\geq 1, (\alpha_i)_{i=1}^{\ell}: \alpha_i\neq 0, \forall i\\
p_{\alpha_i}=p_{\alpha},\sum_{i=1}^{\ell}\alpha_i=\alpha}} \frac{(-1)^{\ell}}{\ell}\mathbb{1}_{\N_{\alpha_1}^{ss,\tw}}*\cdots *\mathbb{1}_{\N_{\alpha_\ell}^{ss,\tw}}.
\end{equation}
The key point is that $\epsilon(\alpha)$ is virtually indecomposable in the Hall algebra, see 
\cite[Theorem 8.7]{Joyce03}. There is a proof using operators on inertia stacks, see \cite{BP}.
Then the  $\epsilon(\alpha)$ is a stack function with algebra stabilizers,
$$\epsilon(\alpha)\in \overline{\text{SF}}^{\ind}_{\text{al}}(\Coh^{\tw}_c(\XX),\qq)$$
in Joyce's notation.  From \cite[Proposition 3.4]{JS}, $\epsilon(\alpha)$ can be written as 
$$\sum_i a_i\cdot Z_i\times B\Gm$$ 
where $a_i\in \qq$, and $B\Gm$ is the classifying stack.  After removing the factor $B\Gm$ on can pullback the Behrend function to write:
\begin{equation}\label{eqn_epsilon_alpha_stack_function}
\epsilon(\alpha):=\sum_{i}a_i\left(f_i: Z_i\times B\Gm\to \Coh_c^{\tw}(\XX)\right).
\end{equation}
Then \cite[Eqn 3.22]{JS} defines generalized Donaldson-Thomas invariants 
\begin{equation}\label{eqn_generalized_JSalpha}
\JS^{\tw}_{\alpha}(\XX):=\sum_{i} a_i \chi(Z_i, f_i^*\nu)\in \qq
\end{equation}
where $\nu: \sM^{\tw}_{\XX/\kappa}\to \zz$ is the Behrend function.  This invariants are defined by applying the integration map 
to the element $\epsilon(\alpha)$ in (\ref{eqn_epsilon_alpha}).  More details of the integration map can be found in \cite{Joyce07}, \cite{JS} and \cite{Bridgeland10}.  
For DM stacks, see \cite{Jiang_DT_Flop}.

As explained in \cite[\S 2]{TT2}, the $\Gm$-action on $\XX\to \SS$ induces an action on the moduli stack 
$\sM^{\tw}_{\XX/\kappa}$ of twisted sheaves on $\XX$.  The $\Gm$-action can be extended to the Hall algebra $H(\sA^{\tw})$, and 
hence $\epsilon(\alpha)$ carries a $\Gm$-action covering the one on 
$\Coh_c^{\tw}(\XX)$. Also in the decomposition
$$\epsilon(\alpha):=\sum_{i} a_i\left( f_i: Z_i\times B\Gm\to \Coh_c^{\tw}(\XX)\right),$$
the $Z_i$'s admit $\Gm$-action so that they are $\Gm$-equivariant and the proof uses Kretch's stratification of finite type algebraic stacks
with affine geometric stabilizers. Then the Behrend function is also $\Gm$-equivariant, and 
\begin{equation}\label{eqn_generalized_JSalpha_Gm}
\JS^{\tw}_{\alpha}(\XX)=(\JS^{\tw}_{\alpha}(\XX))^{\Gm}(\XX)=\sum_{i} a_i \chi(Z_i^{\Gm}, f_i^*\nu|_{Z_i^{\Gm}})\in \qq
\end{equation}

\subsection{Application to semistable $\SS$-twisted Higgs sheaves}\label{subsec_semistable_twisted_Higgs}

We first define:

\begin{defn}\label{defn_semistable_vw}
Let $\SS\to S$ be a $\mu_r$-gerbe; and $\XX=\Tot(K_{\SS})$.  For $\alpha\in K_0(\XX)$, we define the $\SS$-twisted $U(\rk)$-Vafa-Witten invariants of 
$\SS$ by:
$$\widetilde{\vw}_{\alpha}^{\tw}(\SS):=(\JS^{\tw}_{\alpha}(\XX))^{\Gm}(\XX).$$
\end{defn}
\begin{rmk}
If the semistability and stability coincides, then 
$$\widetilde{\vw}_{\alpha}^{\tw}(\SS)=\chi(\N_{\alpha}^{s,\tw},\nu_{\N})=\chi\left((\N_{\alpha}^{s,\tw})^{\Gm},\nu_{\N}|_{(\N_{\alpha}^{s,\tw})^{\Gm}}\right).$$
In this case, the invariants vanish only when $h^1(\sO_S)=0$. We need to define the twisted $\SU(\rk)$-Vafa-Witten invariants $\vw^{\tw}$.
\end{rmk}

\subsubsection{Joyce-Song twisted stable pairs}\label{subsubsec_Joyce-Song_twisted_pair}

We introduce the Joyce-Song pair for $\XX$-twisted sheaves.   In order to define the Joyce-Song pair, we fix a minimal rank $\XX$-twisted locally free 
sheaf $\xi$ on $\XX$ such that $\Xi=\oplus_{i=0}^{r-1}\xi^{\otimes i}$ is a generating sheaf.  
Here a generating sheaf $\Xi$ on $\XX$ is $p$-very ample and contains every irreducible representations of the cyclic group $\mu_r$. 
Recall that $p: \XX\to X$ is the map to its coarse moduli space, and $p_*: \Coh(\XX)\to \Coh(X)$ is an exact functor.

Fixing a $K$-group class $\alpha\in K_0(\XX)\cong K_0(\SS)$, for $m>>0$, a Joyce-Song twisted pair consists of the following data
\begin{enumerate}
\item a compactly supported $\XX$-twisted coherent sheaf $\sE$ with $K$-group class $\alpha\in K_0(\XX)$, and \\
\item a nonzero section $s\in H^0(p_*(\Xi^{\vee}\otimes\sE)(m))$, i.e., a map $s: p_*\Xi\to p_*\sE(m)$.
\end{enumerate}

\begin{defn}\label{defn_stable_JS_twisted_pair}
A Joyce-Song twisted pair $(\sE,s)$ is stable if and only if 
\begin{enumerate}
\item $\sE$ is Gieseker semistable with respect to $\sO_S(1)$.\\
\item if $\sF\subset \sE$ is a proper subsheaf which destabilizes $\sE$, then $s$ does not factor through $p_*\sF(m)\subset p_*\sE(m)$.
\end{enumerate}
\end{defn}

We denote such a pair by $I^\bullet=\{\Xi(-m)\to \sE\}$.
For our $\mu_r$-gerbe $\SS\to S$, and the Diagram \ref{diagram_XX_SS},  $\XX\to X$ is also a $\mu_r$-gerbe.  For any $\XX$-twisted coherent sheaf 
$\sE$ on $\XX$, we have $H^{\geq 1}(p_*\sE(m))=0$
for large $m>>0$.  We then take 
$\sP^{\tw}:=\sP^{\tw}_{(\rk,L,c_2)}(\XX)$ to be the moduli stack of $\XX$-twisted stable Joyce-Song pairs 
$$I^\bullet=\{\Xi(-m)\to \sE\}$$
on $\XX$.  This moduli stack can be constructed as the moduli stack of simple complexes over the gerbe $\XX$, see \cite{Andreini}, \cite{JS}. 

The moduli stack 
$\sP^{\tw}$ is not compact, but admits a symmetric obstruction theory since we can take it as the moduli space of simple complexes on a Calabi-Yau threefold DM stack $\XX$. 
The invariant we use is Behrend's weighted Euler characteristic 
$\widetilde{\sP}^{\tw}_{(\rk,L,c_2)}(m):=\chi(\sP^{\tw}, \nu_{\sP})$.  Under the $\Gm$-action:
$$\widetilde{\sP}^{\tw}_{(\rk,L,c_2)}(m)=\widetilde{\sP}^{\tw,\Gm}_{(\rk,L,c_2)}(m)=\chi\left(\sP^{\Gm}, \nu_{\sP}|_{\sP^{\Gm}}\right).$$
For generic polarization, Joyce-Song invariants $\JS^{\tw}_{\alpha}(\XX)\in \qq$ satisfy the following identities \cite[Theorem 5.27]{JS}:
\begin{equation}\label{eqn_Joyce-Song_wall_crossing}
\widetilde{\sP}^{\tw}_{(\rk,L,c_2)}(m)=\sum_{\substack{\ell\geq 1, (\alpha_i=\delta_i \alpha)_{i=1}^{\ell}: \\
\delta_i>0,\sum_{i=1}^{\ell}\delta_i=1}} \frac{(-1)^{\ell}}{\ell !}\prod_{i=1}^{\ell}
(-1)^{\chi(\alpha_i(m))}\cdot \chi(\alpha_i(m))\cdot \JS^{\tw}_{\alpha_i}(\XX).
\end{equation}
Since $\widetilde{\sP}^{\tw}_{(\rk,L,c_2)}(m)$ is deformation invariant, which implies that $\JS_{\alpha}^{\tw}$ is also deformation invariant. 
For general 
$\sO_{S}(1)$, the wall crossing formula is complicated as in \cite[Theorem 5.27]{JS}.  
If semistability coincides with stability for twisted sheaves $\sE$, the moduli stack $\sP^{\tw}_{(\rk,L,c_2)}$ is a $\pp^{\chi(\alpha(m))-1}$-bundle over the moduli stack 
$\N^{\tw}_{(\rk, L, c_2)}$ of $\XX$-twisted torsion sheaves $\sE$. 
The Behrend function $\nu_{\sP}$ of $\sP^{\tw}_{(\rk,L,c_2)}(\XX)$  is the pullback of the Behrend function on 
$\N^{\tw}_{(\rk, L, c_2)}$ multiplied by $(-1)^{\chi(\alpha(m))-1}$. Therefore: 
\begin{equation}\label{eqn_Joyce-Song_wall_crossing_stable}
\widetilde{\sP}^{\tw}_{(\rk,L,c_2)}(m)=(-1)^{\chi(\alpha(m))-1}\cdot \chi(\alpha(m))\cdot \widetilde{\vw}_{(\rk, L, c_2)}(\SS)
\end{equation}
which is the first term $\ell=1$ in (\ref{eqn_Joyce-Song_wall_crossing}). In general the formula expresses $\widetilde{\sP}^{\tw}_{(\rk,L,c_2)}(m)$ in terms of rational corrections 
from semistable $\XX$-twisted sheaves $\sE$ on $\XX$. 

\begin{rmk}
Before we go to  generalized twisted $\SU(\rk)$-Vafa-Witten invariants,  it is worth mentioning that the wall crossing formula  (\ref{eqn_Joyce-Song_wall_crossing}) above for 
$\widetilde{\sP}^{\tw}_{(\rk,L,c_2)}(m)$  works for $\XX$-twisted sheaves in \cite[Theorem 5.27]{JS}. The reason is that Joyce-Song prove it for categories 
$\sB_{p_{\alpha}}$, whose objects are $\XX$-twisted semistable sheaves $\sE$ with reduced Hilbert polynomial $p_{\alpha}$, and a vector space $V$ together with a linear map
$$V\to H^0(\sE(m))$$
with the same Euler forms $\overline{\chi}$.  Joyce-Song's techniques \cite[Theorem 5.27]{JS} work for abelian categories with stability conditions, thus the wall crossing techniques work for this twisted category. 
\end{rmk}

\subsubsection{Generalized twisted $\SU(\rk)$-Vafa-Witten $\vw$-invariants}

For the $\mu_r$-gerbe $\SS\to S$, if $h^{0,1}(S)>0$, then $\widetilde{\vw}^{\tw}$ defined before vanish because of the action of $\Jac(S)$.
We follow from \cite{TT2} to fix the determinant $\det(E)$. 

Let us fix a line bundle $L\in \Pic(\SS)$, and use the map:
$$\Coh_c^{\tw}(\XX)\stackrel{\det\circ \pi_*}{\longrightarrow}\Pic(\SS).$$
Let us define 
$$\Coh_c^{\tw}(\XX)^{L}:=(\det\circ \pi_*)^{-1}(L).$$
We can restrict Joyce's stack function to this restricted category.  
For any stack function $F:=\left(f: U\to \Coh_c^{\tw}(\XX)\right)$ we define:
$$F^{L}=\left(f: U\times_{\tiny\Coh_c^{\tw}(\XX)}\Coh_c^{\tw}(\XX)^{L}\to \Coh_c^{\tw}(\XX)\right)$$
which is $\mathbb{1}_{\tiny\Coh_c^{\tw}(\XX)^{L}}\cdot F$, where $\cdot$ is the ordinary product \cite[Definition 2.7]{JS}.
Then $\epsilon(\alpha)$ in (\ref{eqn_epsilon_alpha_stack_function}) becomes:
\begin{equation}\label{eqn_epsilon_alpha_L}
\epsilon(\alpha)^L:=\sum_{i} a_i\left( f_i: Z_i\times_{\tiny\Coh_c^{\tw}(\XX)} \Coh_c^{\tw}(\XX)^{L}/\Gm\to \Coh_c^{\tw}(\XX)\right).
\end{equation}
Then applying the integration map again and we get the fixed determinant generalized Donaldson-Thomas invariants:
$$\JS^{L,\tw}_{\alpha}(\XX):=\sum_{i}a_i \chi\left(Z_i\times_{\tiny\Coh_c^{\tw}(\XX)} \Coh_c^{\tw}(\XX)^{L}, f_i^*\nu \right).$$
We also compute it using localization:
\begin{equation}\label{eqn_localized_JS_twisted}
\JS^{L,\tw}_{\alpha}(\XX)=(\JS^{L,\tw}_{\alpha})^{\Gm}(\XX)=\sum_{i}a_i \chi\left(Z^{\Gm}_i\times_{\tiny\Coh_c^{\tw}(\XX)} \Coh_c^{\tw}(\XX)^{L}, f_i^*\nu \right).
\end{equation}

\begin{defn}\label{defn_generalized_twisted_vw_L}
We define the twisted $\SU(\rk)$ generalized Vafa-Witten invariants as:
$$\vw^{\tw}_{(\rk, L,c_2)}(\SS):=(-1)^{h^0(K_{\SS})}\JS^{L,\tw}_{(\rk,L,c_2)}(\XX)\in \qq.$$
\end{defn}

As explained in \cite[\S 4]{TT2}, we did not restrict to the trace zero $\tr\phi=0$, but in the $SU(\rk)$-moduli stack case we did.  The difference  is only a Behrend function sign $(-1)^{\dim H^0(K_{\SS})}$. 
If $h^{0,1}(S)=0$ and $\sO_S(1)$ is generic, then the wall crossing formula (\ref{eqn_Joyce-Song_wall_crossing}) becomes:
\begin{equation}\label{eqn_Joyce-Song_wall_crossing_L}
\widetilde{\sP}^{\tw}_{\alpha}(m)=\sum_{\substack{\ell\geq 1, (\alpha_i=\delta_i \alpha)_{i=1}^{\ell}: \\
\delta_i>0,\sum_{i=1}^{\ell}\delta_i=1}} \frac{(-1)^{\ell}}{\ell !}\prod_{i=1}^{\ell}
(-1)^{\chi(\alpha_i(m))+h^0(K_{\SS})}\cdot \chi(\alpha_i(m))\cdot \vw^{\tw}_{\alpha_i}(\SS).
\end{equation}
If $h^{0,1}>0$, and $\sO_S(1)$ generic, we should use Joyce-Song pairs $(\sE,s)$ with fixed determinant
$$\det(\pi_*\sE)\cong L\in\Pic(\SS).$$
Joyce-Song use the category $\sB_{p_{\alpha}}$ before, and we apply the integration map
$\widetilde{\Psi}^{\sB_{p_{\alpha}}}$ to:
$$\epsilon^L_{\alpha,1}:=\overline{\epsilon}_{\alpha,1}\times_{\tiny\Coh_c^{\tw}(\XX)}\Coh_c^{\tw}(\XX)^L$$
where $\overline{\epsilon}_{\alpha,1}$ is the stack function (13.25), (13.26) in \cite[Chapter 13]{JS}.  
Since $\overline{\epsilon}_{\alpha,1}$ is virtually indecomposable, we get 
$$\sP^{\tw}_{\alpha=(\rk,L,c_2)}(m):=\chi(\sP^{\tw}, \nu_{\sP}),$$
where $\sP^{\tw}:=\sP^{\tw}_{\alpha}$ is the moduli stack of stable Joyce-Song twisted pairs $(\sE,s)$ with $\det(\pi_*\sE)=L$. 
Hence 
\begin{equation}\label{eqn_Palpha_weighted_Euler}
\sP^{\tw}_{\alpha}(m)=\sP^{\Gm}_{\alpha}(m)=\chi((\sP^{\tw})^{\Gm}_{\alpha}, \nu_{\sP}|_{\sP^{\Gm}_{\alpha}}).
\end{equation}
For  generic $\sO_S(1)$, we get (\cite[Proposition 4.4]{TT2}): if $h^{0,1}(S)>0$, the invariants $(\sP^{\tw})^{\Gm}_{\alpha}(m)$ determines the twisted 
$SU(\rk)$-Vafa-Witten invariants $\vw^{\tw}$ of Definition \ref{defn_generalized_twisted_vw_L} by:
\begin{equation}\label{eqn_Palpha_L_stable_case}
\sP^{\tw}_{\alpha}(m)=(-1)^{h^0(K_{\SS})}(-1)^{\chi(\alpha(m))-1}\chi(\alpha(m))\vw^{\tw}_{\alpha}(\SS).
\end{equation}
\begin{rmk}
The proof of (\ref{eqn_Palpha_L_stable_case}) is the same as \cite[Proposition 4.4]{TT2}, and the basic reason is that if $h^{0,1}(S)>0$, then the Jacobian $\Jac(S)$ is an abelian variety which acts on the moduli stacks to force the Euler characteristic to be zero. 
\end{rmk}

\subsection{Generalized twisted $\SU(\rk)$-Vafa-Witten invariants $\VW^{\tw}$}\label{subsec_generalized_twisted_VW}

Fix a $\mu_r$-gerbe $\SS\to S$, and the $\mu_r$-gerbe $\XX\to X$ as in the diagram (\ref{diagram_XX_SS}).
In this section we generalize the arguments in \cite[\S 6]{TT2} to twisted generalized $SU(\rk)$-Vafa-Witten invariants $\VW^{\tw}$. 

For $m>>0$, fixing $\alpha=(\rk,L,c_2)$, $L\in\Pic(\SS)$, the Joyce-Song twisted pair $(\sE, s)$ is defined in \S \ref{subsubsec_Joyce-Song_twisted_pair}, where $\sE$ is a $\XX$-twisted sheaf on $\XX$ corresponding to a $\SS$-twisted Higgs sheaf $(E,\phi)$ on $\SS$. Let 
$$\sP^{\perp, \tw}_{\alpha}\subset \sP^{\tw}_{\alpha}$$
be the moduli stack of twisted stable pairs with $\det(E)\det(\pi_*\sE)=L$, $\tr\phi=0$. 
On $\sP^{\tw}_{\alpha}$, there is a symmetric obstruction theory, see \cite[Chapter 12]{JS}. The tangent-obstruction complex is given by:
$R\cHom_{\XX}(I^\bullet, I^{\bullet})_0[1]$ at the point $I^{\bullet}=\{\sO_{\XX}\stackrel{s}{\longrightarrow}\sE\}$.  By \cite[\S 5]{TT1}, or \cite[Appendix]{JP} we can modify it to a symmetric obstruction theory on 
$\sP^{\perp, \tw}_{\alpha}$.  Actually we have:
$$R\cHom_{\XX}(I^\bullet, I^{\bullet})=R\cHom_{\XX}(I^\bullet, I^{\bullet})_{\perp}\oplus H^*(\sO_{\XX})\oplus H^{\geq 1}(\sO_{\SS})\oplus H^{\leq 1}(K_{\SS})[-1],$$ 
where we actually removed the deformation-obstruction theory of 
$\det(I^{\bullet})\cong \sO_{\XX}$, of $\det(\pi_*\sE)\cong \sO_{\SS}$ and of $\tr\phi\in \Gamma(K_{\SS})$.  Since R.Thomas mentioned that this can be done in families, we omit the details. 

We apply the $\Gm$-virtual localization \cite{GP} to define:
\begin{defn}\label{defn_JS_twisted_pair_invaraint}
$$\sP^{\perp,\tw}_{\alpha}(m):=\int_{[(\sP^{\perp,\tw}_{\alpha})^{\Gm}]^{\vir}}\frac{1}{e(N^{\vir})}.$$
\end{defn}
We micmic \cite{TT2} to conjecture:
\begin{con}\label{con_JS_wall_crossing_VW}
If $H^{0,1}(S)=H^{0,2}(S)=0$, there exist rational numbers $\VW^{\tw}_{\alpha_i}(\SS)$ such that 
$$
\sP^{\perp,\tw}_{\alpha}(m)=\sum_{\substack{\ell\geq 1, (\alpha_i=\delta_i \alpha)_{i=1}^{\ell}: \\
\delta_i>0,\sum_{i=1}^{\ell}\delta_i=1}} \frac{(-1)^{\ell}}{\ell !}\prod_{i=1}^{\ell}
(-1)^{\chi(\alpha_i(m))}\cdot \chi(\alpha_i(m))\cdot \VW^{\tw}_{\alpha_i}(\SS).
$$
for $m>>0$. When either of $H^{0,1}(S)$ or $H^{0,2}(S)$ is nonzero, we only take the first term in the sum:
$$\sP^{\perp,\tw}_{\alpha}(m)=
(-1)^{\chi(\alpha(m))-1}\cdot \chi(\alpha(m))\cdot \VW^{\tw}_{\alpha}(\SS).$$
\end{con}

This conjecture is similar to Conjecture 6.5 in \cite{TT2}. We hope that Laarakker's method \cite{Laarakker2} can prove this conjecture for gerbes.  The reason to make  this conjecture can be found in \cite[\S 6]{TT2}. 
We expect that the conjecture makes sense for $\mu_r$-gerbes $\SS\to S$, since the deformation and obstruction theory for sheaves on gerbes behave like sheaves on surfaces although the summand in the conjecture should depend on the $\mu_r$-gerbe structure on $S$. In \cite{Jiang_Tseng}, we prove this conjecture for $\mu_r$-gerbes on K3 surfaces and also the equality $\VW^{\tw}_{\alpha}(\SS)=\vw^{\tw}_{\alpha}(\SS)$ for a K3 gerbe. 

Finally in this section we provide two results for the $\mu_r$-gerbe $\SS\to S$ generalizing \cite[Proposition 6.8, 6.17]{TT2}, which will be useful later for the calculations. 

\begin{prop}\label{prop_con_true_stable_case}
If the semistable $\XX$-twisted  sheaves in $\N_{\alpha}^{\perp, \tw}$ coincide with stable $\XX$-twisted  sheaves, then Conjecture \ref{con_JS_wall_crossing_VW} is true and $\VW^{\tw}_{\alpha}\in \qq$
defined by $\int_{[\N^{\perp, \tw}_{\alpha}]^{\vir}}\frac{1}{e(N^{\vir})}$.
\end{prop}
\begin{proof}
The proof is the same as \cite[Proposition 6.8]{TT2}, since the deformation of Joyce-Song twisted pairs $I^{\bullet}$ and $\sE$ are similar to the general pairs and sheaves on the DM stack $\XX$. 
\end{proof}

\begin{prop}\label{prop_KS<0_VW_vw}
Suppose that  $K_S<0$, then Conjecture \ref{con_JS_wall_crossing_VW} is true and $\VW^{\tw}_{\alpha}(\SS)=\vw_{\alpha}^{\tw}(\SS)$.
\end{prop}
\begin{proof}
The proof is also the same as  \cite[Proposition 6.17]{TT2}.
\end{proof}

\subsection{The $\SU(r)/\zz_r$-Vafa-Witten invariants and proof in rank two for $\pp^2$}\label{subsec_SUr/Zr}

In this section we define the $\SU(r)/\zz_{r}$-Vafa-Witten invariants using the twisted $\SU(\rk)$-Vafa-Witten invariants $\VW^{\tw}$. 

We fix a smooth projective surface $S$.  Let $r\in \zz_{>0}$ be an integer. The $\mu_r$-gerbes $\SS\to S$ are classified by 
$H^2(S,\mu_r)$.  For any $g\in H^2(S,\mu_r)$, let $\SS_g\to S$ be the $\mu_r$-gerbe corresponding to this $g$.  Since the \'etale cohomology group $H^2(S,\mu_r)$
is finite, we have finite number of equivalent $\mu_r$-gerbes over $S$.

Now fixing a $K$-group data $K_0(\SS_g)$ given by $\alpha=(r, L, c_2)\in H^*(\SS_g, \qq)$, we have  the moduli stack $\N^{\perp,\tw}_{L}(\SS_g)$ of $\SS_g$-twisted semistable Higgs sheaves with data
$\alpha$; and we have the generalized twisted Vafa-Witten invariants 
$\VW^{\tw}_{\alpha}$ defined in Definition \ref{defn_SU_twisted_VW_invariants} in the stable case, and Conjecture \ref{con_JS_wall_crossing_VW} in the semistable case. 
We use all the semistable invariants $\VW^{\tw}_{\alpha}$. 
\begin{defn}\label{defn_SU/Zr_VW}
Fix an $r\in\zz_{>0}$, for any $\mu_r$-gerbe $p: \SS_g\to S$
corresponding to $g\in H^2(S,\mu_r)$, let $\sL\in\Pic(\SS_g)$
let 
$$Z_{r,\sL}(\SS_g, q):=\sum_{c_2}\VW^{\tw}_{(r,\sL, c_2)}(\SS_g)q^{c_2}$$
be the generating function of the twisted Vafa-Witten invariants. 

Let us fix a line bundle 
$L\in\Pic(S)$, and define for any essentially trivial $\mu_r$-gerbe $\SS_g\to S$ corresponding to the line bundle $\sL_g\in\Pic(S)$, 
$L_g:=p^*L\otimes \sL_g$; 
for all the other $\mu_r$-gerbe $\SS_g\to S$, keep the same $L_g=p^*L$. 
For $L\in \Pic(S)$, let $\overline{L}\in H^2(S,\mu_r)$ be the image under the morphism  
$H^1(S, \sO_S^*)\to H^2(S,\mu_r)$. 
We define
$$Z_{r,L}(S, \SU(r)/\zz_r;q):=\sum_{g\in H^2(S,\mu_r)}e^{\frac{2\pi i g\cdot \overline{L}}{r}}Z_{r,L_g}(\SS_g, q).$$
We call it the partition function of $\SU(r)/\zz_r$-Vafa-Witten invariants. 
\end{defn}

\begin{con}\label{con_S_duality}
For a smooth projective surface $S$, the partition function 
of $\SU(r)$-Vafa-Witten invariants 
$$Z_{r,L}(S, \SU(r);q)=\sum_{c_2}\VW_{(r,L, c_2)}(S)q^{c_2}$$ 
 and 
the partition function of $\SU(r)/\zz_r$-Vafa-Witten invariants 
$Z_{r,L}(S, \SU(r)/\zz_r;q)$ satisfy  the S-duality conjecture in Conjecture (\ref{eqn_S_transformation_2}). 
\end{con}

We include a proof of this conjecture for $\pp^2$ in rank two. 

\subsubsection{The partition function for $\pp^2$.}
We consider the projective plane $\pp^2$.  Let $M_{\pp^2}(2, c_1, c_2)$ be the moduli space of stable torsion free sheaves
of rank $2$, first Chern class $c_1$ and second Chern class $c_2$.  Since $K_{\pp^2}<0$, any semistable Higgs sheaf $(E,\phi)$ will have $\phi=0$. Therefore the moduli space of stable Higgs sheaves $\N_{\pp^2}(2, c_1, c_2)$ is isomorphic to $M_{\pp^2}(2, c_1, c_2)$. Also the space $M_{\pp^2}(2, c_1, c_2)$ is smooth and the Vafa-Witten invariant  is just (up to a sign) the Euler characteristic of the moduli space.  Then we introduce 
\begin{equation}\label{eqn_partition_P2}
\widehat{Z}_{c_1}^{\pp^2}(q)=\sum_{c_2}\chi(M_{\pp^2}(2, c_1, c_2))q^{c_2}
\end{equation}
Let $N_{\pp^2}(2, c_1, c_2)$ be the moduli space of stable vector bundles 
of rank $2$, first Chern class $c_1$ and second Chern class $\chi$.  Let 
$$Z_{c_1}^{\vb, \pp^2}(q)=\sum_{c_2}e(N_{\pp^2}(2, c_1, c_2))q^{c_2}$$
be the partition function. 
Then from \cite{Yoshioka}, \cite{Kool}, 
$$\widehat{Z}_{c_1}^{\pp^2}(q)=\frac{q^{\frac{1}{8}}}{\eta(q)^{\chi(\pp^2)}}\cdot Z_{c_1}^{\vb, \pp^2}(q)$$
where $\eta(q)$ is the Dedekind eta function. 

To state the result we introduce some notations. First let $H(\Delta)$ be the Hurwitz class numbers, i.e., 
$H(\Delta)$ is the number of positive definite integer 
 binary quadratic forms $AX^2+BXY+CY^2$ such that $B^2-4AC=-\Delta$ and weighted by the size of its automorphisms group. Let 
$\sigma_0(n)$ be the divisor function. 

\begin{thm}\label{thm_partition_P2}(\cite{Klyachko}, \cite{Yoshioka}, \cite{Kool})
We have:
$$
Z_{c_1}^{\vb, \pp^2}(q)=\begin{cases}
q^{\frac{1}{4}c_1^2+\frac{3}{2}c_1+2}\cdot \sum_{n=1}^{\infty}3 H(4n-1)q^{\frac{1}{4}-n};  & (c_1 \text{~odd});\\
q^{\frac{1}{4}c_1^2+\frac{3}{2}c_1+2}\cdot \sum_{n=1}^{\infty}3 \left(H(4n)-\frac{1}{2}\sigma_0(n)\right)q^{-n};  & (c_1 \text{~even}).
\end{cases}
$$
\end{thm}
In the case $c_1$ is odd, by the work of D. Zagier \cite{Zagier} we see that $Z_{c_1}^{\vb, \pp^2}(q)$ is the holomorphic part of a modular form of weight $3/2$ for $\Gamma_0(4)$ (up to
replacing $q$ by $q^{-1}$ and up to an overall power of $q$ in front). In the case $c_1$ is even one
only obtains modularity after correctly adding strictly semistable sheaves to the moduli
space. Their contribution turns out to cancel the sum of divisors term. 

In order to compare with other partition functions later we introduce:
\begin{equation}\label{eqn_partition_odd_P2}
Z_{c_1, \odd}^{\vb, \pp^2}(q):=q^{\frac{1}{4}c_1^2+\frac{3}{2}c_1+2}\cdot \sum_{n=1}^{\infty}3 H(4n-1)q^{\frac{1}{4}-n};
\end{equation}
and 
\begin{equation}\label{eqn_partition_even_P2}
Z_{c_1, \even}^{\vb, \pp^2}(q):=q^{\frac{1}{4}c_1^2+\frac{3}{2}c_1+2}\cdot \sum_{n=1}^{\infty}3 \left(H(4n)-\frac{1}{2}\sigma_0(n)\right)q^{-n}.
\end{equation}

\subsubsection{The partition function for $\pp(2,2,2)$.}
In \cite{GJK}, the authors generalize the calculation of the moduli of stable torsion free sheaves on smooth toric variety to weighted projective spaces $\pp(a, b, c)$, which is  a special toric DM stack. 
The calculation uses and generalizes  the toric method in \cite{Kool} to this toric DM stack. 
We omit the detail calculation and only 
 include the calculation results for  $\pp(2,2,2)$.

In general the weighted projective plane $\pp(a,b,c)$ is a $\mu_d$-gerbe $\pp(\frac{a}{d},\frac{b}{d},\frac{c}{d})$ where 
$d=\gcd(a,b,c)$.  In the case of $\pp(2,2,2)$ which is a $\mu_2$-gerbe over $\pp^2$,  the partition function will depend on the choice of the component of the inertia stack $I\pp(2,2,2)=\pp(2,2,2)\cup \pp(2,2,2)$. We use $\lambda=0, \text{or~} 1$ to distinct these two components. 

Now we list the results for $\pp(2,2,2)$. In this case the first Chern class $c_1$ is always even. 
Let $\lambda\in\{0,1\}$ index the component in the inertia stack $I\pp(2,2,2)=\pp(2,2,2)\cup\pp(2,2,2)$.
Let  $N_{\pp(2,2,2)}(2, c_1, \chi)$ be the moduli space of stable vector bundles
of rank $2$, first Chern class $c_1$ and second Chern class $c_2$.  Let 
$$Z_{c_1,\lambda}^{\vb, \pp(2,2,2)}(q)=\sum_{c_2}\chi(N_{\pp(2,2,2)}(2, c_1, c_2))q^{c_2}$$
be the partition function. 

\begin{thm}\label{thm_partition_P222}(\cite[Theorem 1.2]{GJK})
Since $c_1$ is even, there are two cases $c_1\equiv 0 (\text{mod~} 4)$ or $c_1\equiv 2 (\text{mod~} 4)$.
We have 
\begin{align*}
&Z_{c_1, 0}^{\vb, \pp(2,2,2)}(q)=\\
&\begin{cases}
Z_{\frac{c_1}{2}}^{\vb, \pp^2}(q)=q^{\frac{1}{16}c_1^2+\frac{3}{4}c_1+2}\cdot \sum_{n=1}^{\infty}3(H(4n)-\frac{1}{2}\sigma_0(n))q^{-n};  & (c_1\equiv 0  \text{~mod~} 4);\\
Z_{\frac{c_1}{2}}^{\vb, \pp^2}(q)=q^{\frac{1}{16}c_1^2+\frac{3}{4}c_1+2}\cdot \sum_{n=1}^{\infty}3(H(4n-1)q^{\frac{1}{4}-n};  & (c_1\equiv 2  \text{~mod~} 4)
\end{cases}
\end{align*}
and 
\begin{align*}
&Z_{c_1, 1}^{\vb, \pp(2,2,2)}(q)=\\
&\begin{cases}
Z_{\frac{c_1}{2}+1}^{\vb, \pp^2}(q)=q^{\frac{1}{4}(\frac{c_1}{2}+1)^2+\frac{3}{2}(\frac{c_1}{2}+1)+2}\cdot \sum_{n=1}^{\infty}
3(H(4n-1)q^{\frac{1}{4}-n};  & (c_1\equiv 0  \text{~mod~} 4);\\
Z_{\frac{c_1}{2}+1}^{\vb, \pp^2}(q)=q^{\frac{1}{4}(\frac{c_1}{2}+1)^2+\frac{3}{2}(\frac{c_1}{2}+1)+2}\cdot  \sum_{n=1}^{\infty}3(H(4n)-\frac{1}{2}\sigma_0(n))q^{-n};  & (c_1\equiv 2  \text{~mod~} 4).
\end{cases}
\end{align*}
\end{thm}

\begin{rmk}
In the case $c_1$ is even,  from \cite{GJK} one only obtains modularity after correctly adding strictly semistable sheaves to the moduli space. Their contribution turns out to cancel the sum of divisors term.  Thus in the following when checking the S-duality, we can ignore the divisor functions. 
\end{rmk}

\subsubsection{S-duality}

From (4.30) of \cite[\S 4]{VW},  by a result of Zagier \cite{Zagier}, let 
$$f_0=\sum_{n\geq 0}3H(4n)q^n+6\tau_2^{-\frac{1}{2}}\sum_{n\in \zz}\beta(4\pi n^2\tau_2)q^{-n^2}$$
and
$$f_1=\sum_{n>0}3H(4n-1)q^{n-\frac{1}{4}}+6\tau_2^{-\frac{1}{2}}\sum_{n\in \zz}\beta(4\pi (n+\frac{1}{2})^2\tau_2)q^{-(n+\frac{1}{2})^2}$$
where $q^{2\pi i \tau}$, and $\tau_2=\Im(\tau)$, and 
$$\beta(t)=\frac{1}{16\pi}\int_{1}^{\infty}u^{-\frac{3}{2}}e^{-ut}du.$$
From \cite{Zagier},  these functions are modular, but not holomorphic.  Hence $Z_{c_1}^{\vb, \pp^2}$ ($c_1$ even or odd) is the homomorphic part of the non-holomorphic modular functions above.  Also by Zagier, see \cite[Formula (4.31)]{VW}, under $\tau\mapsto -\frac{1}{\tau}$, 
we have:
\begin{equation}\label{eqn_S_transformation_P2}
\mat{c} f_0(-\frac{1}{\tau})\\
f_1(-\frac{1}{\tau})\rix=
(\frac{\tau}{i})^{\frac{3}{2}}\cdot \left(-\frac{1}{\sqrt{2}}\right)
\mat{cc} 1&1\\
1&-1\rix\mat{c} f_0(\tau)\\
f_1(\tau)\rix.
\end{equation}
This is the transformation conjecture (\ref{eqn_S_transformation}).  We know that $f_0$ is invariant under $T$ and $f_1$ is invariant under $T^4$. Therefore $f_0$ is invariant under $ST^4S$. 

To check the S-duality, we choose the case $c_1=0$ or $2$, from Theorem \ref{thm_partition_P222} and Theorem \ref{thm_partition_P2} we calculate:
$$
\begin{cases}
Z_{0,0}^{\vb, \pp(2,2,2)}(q)=Z_{0}^{\vb,\pp^2}(q)=q^2\sum_{n=1}^{\infty}3(H(4n)-\frac{1}{2}\sigma_0(n))q^{-n}; \\
Z_{2,0}^{\vb, \pp(2,2,2)}(q)=Z_{1}^{\vb,\pp^2}(q)= q^{\frac{15}{4}}\cdot \sum_{n=1}^{\infty}3(H(4n-1)q^{\frac{1}{4}-n};\\
Z_{0,1}^{\vb, \pp(2,2,2)}(q)=Z_{1}^{\vb,\pp^2}(q)= q^{\frac{15}{4}}\cdot \sum_{n=1}^{\infty}3(H(4n-1)q^{\frac{1}{4}-n};\\
Z_{2,1}^{\vb, \pp(2,2,2)}(q)=Z_{2}^{\vb,\pp^2}(q)=q^6\sum_{n=1}^{\infty}3(H(4n)-\frac{1}{2}\sigma_0(n))q^{-n}.
\end{cases}
$$

Then we check that under transformation $\tau\mapsto -\frac{1}{\tau}$, we have
\begin{equation}\label{eqn_S_P2-P222_1}
q^{-2}\cdot Z_{0}^{\vb,\pp^2}(q)\mapsto 
q^{-2}Z_{0,0}^{\vb, \pp(2,2,2)}(q)+ q^{-\frac{15}{4}}\cdot Z_{0,1}^{\vb, \pp(2,2,2)}(q).
\end{equation}
and 
\begin{equation}\label{eqn_S_P2-P222_2}
q^{-\frac{15}{4}}\cdot Z_{1}^{\vb,\pp^2}(q)\mapsto 
q^{-6}Z_{2,1}^{\vb, \pp(2,2,2)}(q)- q^{-\frac{15}{4}}\cdot Z_{2,0}^{\vb, \pp(2,2,2)}(q).
\end{equation}

\begin{defn}\label{defn_partition_SO3}
We define 
$$Z^{\pp^2}_0(\SU(2)/\zz_2; \tau):=\frac{1}{2}\cdot\left(q^{-2}\cdot Z_{0,0}^{\vb, \pp(2,2,2)}(q)+q^{-\frac{15}{4}}\cdot Z_{0,1}^{\vb, \pp(2,2,2)}(q)\right)$$ and 
$$Z^{\pp^2}_1(\SU(2)/\zz_2; \tau):=\frac{1}{2}\cdot\left(q^{-6}Z_{2,1}^{\vb, \pp(2,2,2)}(q)- q^{-\frac{15}{4}}\cdot Z_{2,0}^{\vb, \pp(2,2,2)}(q)\right)$$
\end{defn}
Then from the above calculations in (\ref{eqn_S_P2-P222_1}), we have
\begin{thm}\label{thm_S-duality_P2}
We define 
$$Z_0^{\pp^2}\left(\SU(2); \tau\right)=q^{-2}\cdot Z_{0}^{\vb,\pp^2}(q)$$ and 
$$Z_1^{\pp^2}\left(\SU(2); \tau\right)=q^{-\frac{15}{4}}\cdot Z_{1}^{\vb,\pp^2}(q)$$ 
Under the $S$-transformation  $\tau\mapsto -\frac{1}{\tau}$, we have:
$$
Z_0^{\pp^2}\left(\SU(2); -\frac{1}{\tau}\right)=\pm 2^{-\frac{3}{2}}\left(\frac{\tau}{i}\right)^{\frac{3}{2}}Z_0^{\pp^2}(\SU(2)/\zz_2; \tau)
$$ 
and 
$$
Z_1^{\pp^2}\left(\SU(2); -\frac{1}{\tau}\right)=\pm 2^{-\frac{3}{2}}\left(\frac{\tau}{i}\right)^{\frac{3}{2}}Z_1^{\pp^2}(\SU(2)/\zz_2; \tau)
$$ 
And the S-duality conjecture (\ref{eqn_S_transformation}) holds. 
\end{thm}
\begin{proof}
The group is $H^2(\pp^2, \mu_2)\cong \mu_2$, where $0\in \mu_2$ corresponds to the trivial $\mu_2$-gerbe $\pp^2/\mu_2=\pp^2\times B\mu_2$, and 
$1\in \mu_2$ corresponds to the nontrivial $\mu_2$-gerbe $\pp(2,2,2)$ over $\pp^2$ given by the line bundle $\sO_{\pp^2}(-1)$.  

We prove the first transformation formula. 
The partition function $Z_{0,0}^{\vb, \pp(2,2,2)}(q)$ is the same as the partition function 
$Z_{0,0}^{\vb, [\pp^2/\mu_2]}(q)$
of the  trivial 
$\mu_2$-gerbe $[\pp^2/\mu_2]$. This is because  the partition function $Z_{0,0}^{\vb, [\pp^2/\mu_2]}(q)$ corresponds to counting torsion free rank two sheaves with trivial 
$\mu_2$-action. 
The partition function $Z_{0,1}^{\vb, \pp(2,2,2)}(q)$ is the same as the partition function 
$Z_{2,0}^{\vb,\pp(2,2,2)}(q)$ from Theorem \ref{thm_partition_P222} and the above calculation. 
We explain this a bit.  The inertia stack $I\pp(2,2,2)=\pp(2,2,2)_0\cup \pp(2,2,2)_1$, where the first component corresponds to 
$0\in \mu_2$, and the second component corresponds to $1\in\mu_2$.  Our moduli of twisted sheaves actually counts 
sheaves on the first component.  The moduli space of  $\mu_2$-gerbe $\pp(2,2,2)$-twisted sheaves with data $(2,p^*\sO(-1),c_2)$ is isomorphic to the 
moduli space of  twisted sheaves with data $(2, 0,c_2)$ on the second component.
The second formula is similar. 
Thus
from the above $S$-transformation in (\ref{eqn_S_transformation_P2}) and (\ref{eqn_S_P2-P222_1}), the theorem follows. 
\end{proof}

\section{S-duality conjecture for K3 surfaces}\label{sec_S-duality_K3}

In this section we study Conjecture \ref{con_S_duality} for smooth $K3$ surfaces. We prove the S-duality conjecture of Vafa-Witten in all the  prime  ranks $r$.  

\subsection{K3 surfaces- Picard numbers, Mukai Lattice and Brauer group}\label{subsec_K3_basic_materials}

This section serves some basic knowledges on smooth K3 surfaces.  Our reference is \cite{Huybrechts}.
Let $S$ be a smooth projective K3 surface.  The cohomology $H^*(S,\zz)$  is zero on odd degrees, and $H^0(S,\zz)=H^4(S,\zz)\cong \zz$, $H^2(S,\zz)= \zz^{22}$. The Hodge diamond for the Hodge numbers 
$h^{p,q}=\dim H^q(S,\Omega_S^p)$ is given by:
\[
\begin{array}{cccccc}
&&h^{0,0}&&\\
&h^{1,0}&&h^{0,1}&\\
h^{2,0}&& h^{1,1}&& h^{0,2}\\
&h^{2,1}&& h^{1,2}&&\\
&&h^{2,2}&&
\end{array}
=
\begin{array}{cccccc}
&&1&&\\
&0&&0&\\
1&& 20 && 1\\
&0&& 0&&\\
&&1&&
\end{array}
\]
Let $\Pic(S)$ be the Picard group of $S$. 
The N\'eron-Severi group of $S$ is the quotient $\NS(S) = \Pic(S)/\Pic^0(S)$, where $\Pic^0(S)$  is the subgroup of all line bundles that are algebraically equivalent to zero.
The group $\Num(S)$ is defined as the quotient of the Picard group $\Pic(S)$ by the kernel of the intersection form on $H^*(S,\zz)$. We say that line bundles from the kernel are numerically trivial.
The rank of $\NS(S)$ is called the Picard number and is denoted by $\rho(S)$. For algebraic K3 surface $S$. 
$$\Pic(S)\cong \NS(S)\cong \Num(S).$$
The Picard number $\rho(S)$ satisfies 
$$0\leq \rho(S)\leq 20$$
for a smooth complex K3 surface $S$. 
Over character zero field $\kappa$, a smooth K3 surface $S$ with Picard number $\rho=20$ is called a {\em singular K3}.  Over the field $\kappa$ with  positive characteristic $p>0$, a K3 surface always has Picard rank $\rho(S)\geq 1$, and a K3 surface  with Picard number $\rho(S)=22=h^{1,1}$ is called a {\em supersingular K3} surface, see \cite{Maulik}.

On the cohomology $H^*(S,\zz)$, we review the Mukai lattice as in \cite{Yoshioka1}.
\begin{defn}\label{defn_Mukai_lattice}
Define a symmetric bilinear form on $H^*(S,\zz)$:
$$\langle x,y\rangle=-\int_{S}x^{\vee}y$$
where $x^{\vee}$ is the dual of the class $x$.  We call this lattice the Mukai lattice. 
\end{defn}
\begin{rmk}
If $x=x_0+x_1+x_2$, for $x_i\in H^{2i}(S,\zz)$, $y=y_0+y_1+y_2$, for $y_i\in H^{2i}(S,\zz)$, then 
$\vee: H^*(S,\zz)\to H^*(S,\zz)$ is the homomorphism sending $x$ to $x_0-x_1+x_2$.  Then 
$$\langle x,y\rangle=x_1y_1-x_0y_2-x_2y_0.$$
\end{rmk}

Let $E$ be a coherent sheaf on $S$, we define 
$$v(E):=\Ch(E)\sqrt{\td}_{S}=\Ch(E)(1+\omega)$$
where $\omega$ is the fundamental class of $S$ and $\Ch(E)$ is the Chern character of $E$.  Then Riemann-Roch theorem says that 
$$\chi(E, F)=-\langle v(E), v(F)\rangle,$$
where $E$ and $F$ are coherent sheaves on $S$.  This Mukai lattice has a decomposition:
$$H^*(S,\zz)=H^2(S,\zz)\oplus H^0(S,\zz)\oplus H^4(S,\zz)=(-E_8)^{\oplus 2}\oplus H^{4},$$
where $H$ is the hyperbolic lattice.  Some properties of the Mukai lattices like isometries can be found in \cite{Yoshioka1}.

We talk about the Brauer group of $S$.  According to \cite{HS}, also from de Jong, 
$$\Br(S)=\Br^\prime(S)=H^2(S,\sO_S^*)_{\tor}.$$
Consider the exact sequence
$$1\to \mu_r\longrightarrow \sO_S^*\stackrel{(\cdot)^r}{\longrightarrow}\sO_S^*\to 1$$
and taking long cohomology we get an exact sequence:
$$1\to H^1(S,\sO_S^*)/r\cdot H^1(S,\sO_S^*)\longrightarrow H^2(S,\mu_r)\longrightarrow H^2(S,\sO_{S}^*)_{r-\tor}\to 1.$$
This is true for any $r\in \zz_{\geq 0}$. Therefore 
$\Br(S)=(\qq/\zz)^{22-\rho(S)}$, where $\rho(S)$ is the Picard number.  Finally we have 
$$H^2(S,\mu_r)\cong (\zz_r)^{22}.$$

Let us pay attention to the case $r=p$  for a prime number $p$.  Then in this case, 
$$|H^2(S,\mu_p)|=|(\zz_p)^{22}|=p^{22}.$$

\begin{prop}\label{prop_non_essential_trivial}
The number of equivalent classes of non-essentially trivial $\mu_p$-gerbes on $S$ is $p^{22}-p^{\rho(S)}$.
\end{prop}
\begin{proof}
The essentially trivial $\mu_p$-gerbes on $S$ is given by the Picard number $\rho(S)$, which is $p^{\rho(S)}$.  Thus the result just follows.
\end{proof}
\begin{rmk}\label{rmk_optimal_mu2}
In the  case $p=2$, the number of equivalent optimal $\mu_2$-gerbes on $S$ is exactly $2^{22}-2^{\rho(S)}$.
\end{rmk}

\subsection{Tanaka-Thomas's Vafa-Witten invariants for K3 surfaces}\label{subsec_VW_Tanaka-Thomas_K3}

Let $S$ be a smooth K3 surface.   Then fixing a class $\alpha=(r,L, c_2)\in H^*(S,\zz)$, the Vafa-Witten invariants (stable ones or the generalized ones)
$\VW_{\alpha}(S)$, $\vw_{\alpha}(S)$ are defined in \cite{TT1}, \cite{TT2}.  They have the same definitions as in Definition \ref{defn_SU_twisted_VW_invariants},
 Definition \ref{defn_SU_twisted_vw_invariants}, Definition \ref{defn_generalized_twisted_vw_L}, and Conjecture \ref{con_JS_wall_crossing_VW} when taking no gerbe structures on $S$ and omit the twist.
 Also there is a similar conjecture for $\VW_{\alpha}(S)$ as in Conjecture 1.2 in \cite{TT2}.  This conjecture is true for K3 surfaces \cite{MT}. Thus 
 $$\VW_{\alpha}(S)=\vw_{\alpha}(S).$$
 
 We recall some results of \cite[\S 5]{TT2}.
 
 \begin{prop}\label{prop_Higgs_pair_Gm}(\cite[Proposition 5.1]{TT2}, generalizing the stable Higgs pair case)
 If $(E,\phi)$ is a $\Gm$-fixed Higgs pair by the action of scaling the Higgs field $\phi$, then $E$ admits an algebraic $\Gm$-action 
 $$\Psi: \Gm\to \Aut(E)$$
 such that $\Psi_t\circ\phi\circ \Psi^{-1}=t\phi$ for all $t\in\Gm$.
 \end{prop}
 
 \begin{prop}\label{prop_rank_2_Higgs_pair_TT2}(\cite[Lemma 5.6]{TT2})
 Let $(E,\phi)$ be a $\Gm$-fixed Gieseker semistable Higgs pair, then $E$ is Gieseker semistable. Moreover, either $\phi=0$ or
 $c_2(E)=2k$ is even and, up to twist by some power of $\mathfrak{t}$, we have
\begin{equation}\label{eqn_E_semistable_decomposition}
E=I_{Z}\oplus I_Z\cdot \mathfrak{t}^{-1},  \phi=\mat{cc} 0&0\\
1&0\rix
\end{equation}
for some length $k$ subscheme $Z\subset S$.
 \end{prop}
 
  \begin{prop}\label{prop_rank_2_Higgs_pair_TT2_Behrend_function}(\cite[Proposition 5.9]{TT2}, \cite{MT})
 Let $(E,\phi)$ be a $\Gm$-fixed Gieseker semistable Higgs pair on a smooth K3 surface $S$ with fixed determinant $L$ and $\tr\phi=0$,  then 
 the Behrend function at this point 
 $$\nu_{\N}=-1.$$
 \end{prop}

\subsection{Twisted Vafa-Witten invariants-essentially trivial gerbes}\label{subsec_VW_twisted_essential}

In this section we study the twisted Vafa-Witten invariants for essentially trivial $\mu_r$-gerbes $\SS\to S$ over the K3 surface $S$. 

\subsubsection{Trivial $\mu_r$-gerbes on $S$}
Recall that a $\mu_r$-gerbe $\SS\to S$ is called essentially trivial if the class $[\SS]$ is in the image of the map:
$H^1(S,\sO_S^*)\to H^2(S,\mu_r)$. Thus from \S \ref{subsec_gerbe_essential}, the gerbe $\SS$ is given by a line bundle $\sL\in\Pic(S)$ on $S$.
If $\sL$ is non-trivial, we call $\SS$ a {\em nontrivial essentially trivial }$\mu_r$-gerbe, otherwise, if $\sL$ is a trivial line bundle on $S$, the gerbe $\SS=[S/\mu_r]$ is trivial, where $\mu_r$ acts globally trivial on $S$.  For such a trivial $\mu_r$-gerbe $\SS\to S$, any $\SS$-twisted sheaf $E$ on $\SS$ is actually a sheaf and is pullback from $S$.  We have:

\begin{prop}\label{prop_trivial_mur_gerbe_K3}
Let $\SS=[S/\mu_r]\to S$ is a trivial $\mu_r$-gerbe over a K3 surface $S$. 
Then the  moduli stack $\N^{ss,\tw}_{\SS/\kappa}(r, \sO, c_2)$ of $\SS$-twisted semistable Higgs sheaves is isomorphic to the moduli stack $\N^{ss}_{S/\kappa}(r,\sO,c_2)$ of Gieseker semistable Higgs sheaves on $S$.
\end{prop}
\begin{proof}
This is from the fact that any $\SS$-twisted semistable sheaf or Higgs sheaf on $\SS$ is a sheaf on $\SS$, which is a pullback from a semistable sheaf on $S$. So the moduli spaces are isomorphic. 
\end{proof}

Then the twisted  Vafa-Witten invariants $\VW^{\tw}(\SS)=\vw^{\tw}(\SS)$ for trivial $\mu_r$-gerbe $\SS$ are the same as the Vafa-Witten invariants 
$\VW(S)$ and $\vw(S)$ as in \cite{TT2}.

\subsubsection{Non-trivial essentially trivial $\mu_r$-gerbes}

Let $\SS\to S$ be a non-trivial essentially trivial $\mu_r$-gerbe over a K3 surface $S$.

First we make Proposition \ref{prop_moduli_essential_trivial_coarse_moduli} rigorous in the K3 surface case.
\begin{prop}\label{prop_moduli_essential_trivial_coarse_moduli_K3}
Suppose that $p:\SS\to S$ is a non-trivial essentially trivial $\mu_r$-gerbe corresponding to the line bundle $\sL\in\Pic(S)$,  then  
the moduli stack $\N^{ss,\tw}_{\SS/\kappa}(r,\sL, c_2)$ of $\SS$-twisted Higgs sheaves on $\SS$ is isomorphic to the moduli stack 
 $\N^{ss,\tw}_{S/\kappa}(r,\sO, c_2)$ of $\SS$-twisted Higgs sheaves on $S$.
\end{prop}
\begin{proof}
Let $\SS=[\sL]^{\frac{1}{r}}=[\Tot(\sL^{\times})/\Gm]$, be the $r$-th root stack for the line bundle $\sL$, which is the gerbe $\SS$ we want. 
Let $L$ be the universal $r$-th root of the line bundle $\sL$. Then the functor 
$$E\mapsto E\otimes L^{\vee}$$
yields an equivalence of categories from $\SS$-twisted sheaves on $\SS$ to sheaves on $S$. Hence also 
yields an equivalence of categories from $\SS$-twisted Higgs sheaves on $\SS$ to Higgs sheaves on $S$.
The semistability of a $\SS$-twisted Higgs sheaf $(E,\phi)$ translates into $\sL(r\cdot \gamma)$-twisted semistability of $E\otimes L^{\vee}$. 
Here the line bundle $\gamma\in \Pic(S)\otimes\frac{1}{r}\zz$ is the line bundle $L^{\vee}$ on $\SS$.  

Thus 
the moduli stack $\sM^{ss,\tw}_{\SS/\kappa}(r, \sL,c_2)$ is isomorphic to the stack 
$\sM^{ss,\gamma}_{S/\kappa}(r, \sL(r\gamma),c_2)$ of $\gamma$-twisted semistable sheaves on $S$ of rank $r$, determinant $\sL(r\gamma)$ and second Chern class $c_2$. This is  also true for the  moduli stack of Higgs sheaves.
\end{proof}

Now we restrict ourself to the rank $2$ case.  Let $\SS\to S$ be a non-trivial essentially trivial $\mu_2$-gerbe.

\begin{prop}\label{prop_rank_2_splitting}
Let $(E,\phi)$ be a $\SS$-twisted semistable rank $2$ Higgs sheaf with class $(2, \sL, c_2)$.  Then the torsion free sheaf $E$ can not split into sum of line bundles. 
\end{prop}
\begin{proof}
First since our torsion free sheaf $E$ has determinant $\det(E)=\sL$, where $\sL$ is the defining line bundle for the $\mu_r$-gerbe $\SS$.  
Our gerbe $[\SS]\in H^2(S,\mu_2)$ is a nontrivial element.  Thus 
$$c_1(E)=c_1(\det(E))\equiv [\SS] \mod 2.$$
One can think of this as the second Stiefel-Whitney class of $E$ on $S$.

Suppose that $E$ can be split into 
$E=E_1\oplus E_2$ such that each $E_i$ is of rank one.  
Also $\det(E)=\sL$
implies that $E=E_1\oplus E_2=\sL^{\frac{1}{2}}\oplus\sL^{\frac{1}{2}}$, in which $\sL^{\frac{1}{2}}$ is also non-trivial. 
Then 
$c_1(E)=c_1(E_1)+c_1(E_2)$ and 
after modulo $\mod 2$ we get that $c_1(E)\equiv 0 \mod 2$, which is a contradiction. 
\end{proof}
\begin{rmk}
Another way to see this is that $E$, as a $\SS$-twisted torsion-free sheaf or vector bundle,  can be taken as an $SO(3)=\SU(2)/\zz_2$-bundle. Then an $SO(3)$-bundle with non-trivial 
 Stiefel-Whitney class $w_2(E)=c_1(E)\mod 2$ can not split.  Thanks to Yang Li for talking about this point. 
\end{rmk}

Now we are ready to calculate the twisted Vafa-Witten invariants.

\begin{prop}\label{prop_twisted_vw_essential_trivial}
Let $\SS\to S$ be a non-trivial essentially trivial $\mu_2$-gerbe over the K3 surface $S$ corresponding to the line bundle $\sL\in\Pic(S)$. Then
we have 
$$\vw^{\tw}_{2,\sL, 0}(\SS)=0; \quad   \vw^{\tw}_{2,\sL, 1}(\SS)=0.$$
For $k\geq 1$
we have 
$$
\begin{cases}
\vw^{\tw}_{2,\sL, 2k}(\SS)=\chi(\Hilb^{4k-3}(S));\\ \\
\vw^{\tw}_{2,\sL, 2k+1}(\SS)=\chi(\Hilb^{4k-1}(S)),
\end{cases}
$$
where $\chi(\Hilb^{n}(S))$ is the topological Euler number of the Hilbert scheme of $n$ points on $S$.
\end{prop}
\begin{proof}
From Definition \ref{defn_generalized_twisted_vw_L}, 
$$\vw^{\tw}_{(2, \sL,c_2)}(\SS):=(-1)^{h^0(K_{\SS})}\JS^{\sL,\tw}_{(2,\sL,c_2)}(\XX).$$
Thus we need to study the Joyce-Song invariants $\JS^{\sL,\tw}_{2,\sL,c_2}(\XX)$ of the local K3 gerbe for $\SS\to S$.
The Joyce-Song twisted invariants count Gieseker semistable $\XX$-twisted sheaves on $\XX$.  
And also from Proposition \ref{prop_rank_2_Higgs_pair_TT2_Behrend_function}, the Behrend function is alway $-1$.  
Note that the proof of Proposition \ref{prop_rank_2_Higgs_pair_TT2_Behrend_function} in \cite{TT2} only used the local d-critical scheme structure of the moduli space, does not depend on the 
gerbe structure $\SS\to S$.  The result is true for the moduli space of $\SS$-twisted sheaves on 
$\SS$. 
So since 
$(-1)^{h^0(K_{\SS})}=-1$, the sign cancels and we need to count the Joyce-Song twisted invariants $\JS^{\sL,\tw}_{2,\sL,c_2}(\XX)$ in the Euler characteristic level. 

Next we use a gerby version of Toda's result in \cite{Toda_JDG} to express $\JS^{\sL,\tw}_{2,\sL,c_2}(\XX)$ as the Euler characteristic of 
the moduli stack $\sM^{ss,\tw}_{\SS/\kappa}(2,\sL,c_2)$ of $\SS$-twisted semistable sheaves on $\SS$. 
Toda's method in  \cite{Toda_JDG} works for category with stability conditions.  We will write down more details of Toda's method for the category of 
$\SS$-twisted sheaves in future work. 
Since our $\alpha=(2,\sL,c_2)$ can not split, it is always true, see \cite[Formula (7)]{Toda_JDG}.  
And then from Proposition \ref{prop_moduli_essential_trivial_coarse_moduli_K3} this moduli stack is isomorphic to the moduli stack
$\sM^{ss,\gamma}_{S/\kappa}(2, \sO, c_2^\prime)$ of $\gamma$-twisted semistable sheaves on $S$ of rank $2$, determinant $\sO$ and second Chern class $c_2^\prime$, where 
$\gamma=\sL$. From \cite{Yoshioka1},  it is deformation equivalent to the Hilbert scheme $\Hilb^{2c_2^\prime-3}(S)$ of points on $S$.    

The first vanishing  is trivial since if $c_2=0$, then $E=\sL\oplus\sO$ a contradiction with the non-splitting condition in Proposition \ref{prop_rank_2_splitting}. 
The second vanishing is from the negative number $2-3=-1$ and $\Hilb^{-1}(S)=\emptyset$.  All other results from the Euler characteristic of $\Hilb^{2c_2^{\prime}-3}(S)$.
\end{proof}

\subsection{Twisted Vafa-Witten invariants-optimal gerbes}\label{subsec_VW_twisted_optimal}

In this section we study the twisted Vafa-Witten invariants for optimal $\mu_2$-gerbes $\SS\to S$ over a K3 surface $S$. 
From \S \ref{subsec_Brauer_optimal_gerbe}, these $\mu_2$-gerbes correspond to nontrivial order $2$ Brauer classes in $H^2(S,\sO_S^*)_{2-\tor}$, and from Proposition \ref{prop_non_essential_trivial} and Remark \ref{rmk_optimal_mu2}, there are totally 
$$2^{22}-2^{\rho(S)}$$
number of nontrivial order $2$ Brauer classes.  

Let $\SS\to S$ be a $\mu_2$-optimal gerbe over the K3 surface $S$. 
Let $[\SS]\in H^2(S,\sO_S^*)_{\tor}$ be its class, then $[\SS]$ represents a class in $H^1(S,PGL_2)$ since 
$\Br^{\prime}(S)=\Br(S)$ and  $H^2(S,\sO_S^*)_{2-\tor}\cong H^1(S,PGL_2)$. We denote by 
$$p: P\to S$$
the corresponding projective $\pp^1$-bundle over $S$, which is the Brauer-Severi variety. 
Then from Theorem \ref{thm_moduli_twisted_Yoshioka}:
\begin{prop}\label{prop_moduli_twisted_Yoshioka_K3}
The moduli stack $\sM_{\SS/\kappa}^{s,\tw}(2, \sO, c_2)$ of stable $\SS$-twisted  Higgs sheaves on $\SS$ is  a $\mu_2$-gerbe over the moduli stack 
$M^{P,G}_{S/\kappa}(2, \sO,c_2)$ of the $G$-twisted semistable Higgs sheaves on $S$.
\end{prop}

We explain the vector bundle $G$ on $P$ and its twisted moduli of Yoshioka.  From \S \ref{subsec_twisted_moduli_optimal}, the vector bundle $G$ over $P$ is defined by the Euler sequence
$$0\to \sO_P\rightarrow G\rightarrow T_{P/S}\to 0$$
of the projective bundle $P\to S$. 

\subsubsection{Integral structure on $H^*(S,\qq)$ following Huybrechts-Stellari}

The order $|[\SS]|=2$ in $H^2(S,\sO_S^*)_{\tor}$ is the index $\ind(\SS)$ of the $\mu_2$-gerbe $\SS$.
Therefore $2$ is the minimal rank on $\SS$
such that there exists a rank $2$ $\SS$-twisted locally free sheaf $E$ on the generic scheme $S$. Recall that 
$\langle, \rangle$ the Mukai pairing on $H^*(S,\zz)$ as in Definition \ref{defn_Mukai_lattice}.
Recall that in 
\S \ref{subsubsec_twisted_moduli_optimal_Yoshioka} we defined $P$-sheaves following Yoshioka. 

\begin{defn}\label{defn_Mukai_vector_Psheaf}
For a $P$-sheaf $E$, define a Mukai vector of $E$ as:
$$v_{G}(E):=\frac{\Ch(Rp_*(E\otimes G^{\vee}))}{\sqrt{\Ch(Rp_*(G\otimes G^\vee))}}\sqrt{\td}_S=(\rk, \zeta, b)\in H^*(S,\qq),$$
where $p^*(\zeta)=c_1(E)-\rk(E)\frac{c_1(G)}{\rk(G)}, b\in \qq$.
\end{defn}
One can check that 
$$\langle v_{G}(E_1), v_{G}(E_2)\rangle=-\chi(E_1, E_2).$$
We define for $\xi\in H^2(S,\zz)$, an injective homomorphism:
$$T_{-\frac{\xi}{2}}: H^*(S,\zz)\longrightarrow H^*(S,\qq)$$
by
$$x\mapsto e^{-\frac{\xi}{2}}\cdot x$$
and 
$T_{-\frac{\xi}{2}}$ preserves the bilinear form  $\langle, \rangle$. The following result is from \cite[Lemma 3.3]{Yoshioka2}. 

\begin{prop}\label{prop_repn_Mukai_vector}(\cite[Lemma 3.3]{Yoshioka2})
Let $\xi\in H^2(S,\zz)$ be a representation of $\omega(G)\in H^2(S,\mu_2)$, where $\rk(G)=2$. Set 
$$(\rk(E), D, a):=e^{\frac{\xi}{2}}\cdot v_G(E).$$
Then $(\rk(E), D, a)\in H^*(S,\zz)$ and $D \mod 2=\omega(E)$.
\end{prop}

In \cite{H-St}, Huybrechts and  Stellari defined a weight $2$ Hodge structure on the lattice 
$(H^*(S,\zz), \langle,\rangle)$ as:
$$
\begin{cases}
H^{2,0}(H^*(S,\zz)\otimes \aaa^1_{\kappa})):=T_{-\frac{\xi}{2}}^{-1}(H^{2,0}(S));\\
H^{1,1}(H^*(S,\zz)\otimes \aaa^1_{\kappa})):=T_{-\frac{\xi}{2}}^{-1}(\oplus_{p=0}^{2}H^{p,p}(S));\\
H^{0,2}(H^*(S,\zz)\otimes \aaa^1_{\kappa})):=T_{-\frac{\xi}{2}}^{-1}(H^{0,2}(S)).
\end{cases}
$$
and this polarized Hodge structure is denoted by 
$$\left(H^*(S,\zz), \langle, \rangle, -\frac{\xi}{2}\right).$$
From \cite[Lemma 3.4]{Yoshioka2}, this Hodge structure 
$\left(H^*(S,\zz), \langle, \rangle, -\frac{\xi}{2}\right)$
depends only on the Brauer class 
$o([\xi \mod 2])$, where 
$o(\xi)$ is the image under the map $o: H^2(S,\mu_2)\to H^2(S,\sO_S^*)$.

\begin{defn}\label{defn_integral_structure_K3}
For the projective bundle $P\to S$ and $G$ the locally free $P$-sheaf. Let 
$\xi\in H^2(S,\zz)$ be a lifting of 
$\omega(G)\in H^2(S,\mu_2)$, where $\rk(G)=2$. 
\begin{enumerate}
\item Define an integral Hodge structure of $H^*(S,\qq)$ as:
$$T_{-\frac{\xi}{2}}\left(\left(H^*(S,\zz), \langle, \rangle, -\frac{\xi}{2}\right)\right).$$
\item $v=(\rk, \zeta, b)$ is a Mukai vector if $v\in T_{-\frac{\xi}{2}}\left(H^*(S,\zz)\right)$ adn $\zeta\in \Pic(S)\otimes \qq$.
Moreover, if $v$ is primitive in  $T_{-\frac{\xi}{2}}\left(H^*(S,\zz)\right)$, then $v$ is primitive. 
\end{enumerate}
\end{defn}

In \cite{Yoshioka2}, the moduli stack $\sM^{P,G}_{H,ss}(v)$ (or $\sM^{P,G}_{H,s}(v)$)
of $G$-twisted semistable (or stable) $P$-sheaves $E$ with $v_G(E)=v$ is defined.  Let 
$M^{P,G}_{H,ss}(v)$ (or $M^{P,G}_{H,s}(v)$) be its coarse moduli space.  The moduli stack $\N^{P,G}_{H,ss}(v)$ (or $\N^{P,G}_{H,s}(v)$)
of $G$-twisted semistable (or stable) Higgs $P$-sheaves $(E,\phi)$ with $v_G(E)=v$ is similarly defined, where  
$E$ is a $G$-twisted $P$-sheaf and 
$$\phi: p_*E\to p_*(E)\otimes K_{\SS}$$
is the Higgs field.  This makes sense since $p_*(E\otimes L^{\vee})$ is a $G$-twisted sheaf on $S$. 
Let 
$N^{P,G}_{H,ss}(v)$ (or $N^{P,G}_{H,s}(v)$) be its coarse moduli space. 
 From \cite[Theorem 2.1]{Yoshioka2}, 
 
 \begin{thm}\label{thm_moduli_N_mu2_gerbe}
 The coarse moduli space $N^{P,G}_{H,s}(v)$ exists, and it is a quasi-projective scheme.  Moreover, the natural map 
 $\N^{s,\tw}_{\SS/\kappa}(v)\to N^{P,G}_{H,s}(v)$ is a $\mu_2$-gerbe. 
 \end{thm}
 \begin{proof}
 The moduli space $\N^{P,G}_{H,s}(v)$ is actually isomorphic to the the coarse moduli space of the moduli stack of $\mu_2$-gerbe $\SS$-twisted Higgs sheaves on $S$, since 
 the $\mu_2$-gerbe $\SS$ has class $[\SS]\in H^2(S,\sO_S^*)$ is exactly the twisting gerbe in \cite{Yoshioka2}.  Therefore the coarse moduli space exists and it is the underlying space of the $\mu_2$-gerbe, which is isomorphic to  $\N^{s,\tw}_{\SS/\kappa}(v)$, the moduli stack of $\SS$-twisted Higgs sheaves $(E,\phi)$ with $v_G(E)=v$.
 \end{proof}
 
 Now we can calculate the twisted Vafa-Witten invariants:
 
 \begin{prop}\label{prop_rank2_vw_optimal}
 Let $\SS\to S$ be an optimal $\mu_2$-gerbe over a K3 surface $S$. Fixing a Mukai vector $v=(2, \zeta, b)\in H^*(S,\qq)$ of rank two.  Then the twisted Vafa-Witten invariant
 $$\vw^{\tw}_{v}(\SS)=\VW^{\tw}_{v}(\SS)=\frac{1}{2}\chi(M^{P,G}_{H,ss}(v)).$$
 \end{prop}
 \begin{proof}
 Since $\SS\to S$ is an optimal gerbe, the minimal rank of $\SS$-twisted locally free sheaf over generic scheme $S$ is $2$.  Thus any rank $2$ $\SS$-twisted torsion free sheaf is stable. Therefore for any $\SS$-twisted semistable Higgs sheaf $(E,\phi)$, $E$ is stable.  Hence in this case, the moduli stack $\N^{s,\tw}_{\SS/\kappa}(v)$ is an affine cone over the moduli stack $\sM^{s,\tw}_{\SS/\kappa}(v)$ of $\SS$-twisted stable sheaves.  Since from standard obstruction theory
 $$\Ext^2(E,E)_0=0,$$
 this moduli stack $\sM^{s,\tw}_{\SS/\kappa}(v)$ is smooth and  $\N^{s,\tw}_{\SS/\kappa}(v)$ is an affine bundle.  Thus the twisted Vafa-Witten invariants
 $$\vw^{\tw}_{v}(\SS)=\chi(\sM^{s,\tw}_{\SS/\kappa}(v))$$
 due to the Behrend function value $-1$, and there is a $-1$ sign for $h^0(K_{\SS})$.  From Theorem \ref{thm_moduli_N_mu2_gerbe}, $\sM^{s,\tw}_{\SS/\kappa}(v)\to M^{P,G}_{H,ss}(v)$ is a $\mu_2$-gerbe, we get 
  $$\vw^{\tw}_{v}(\SS)=\frac{1}{2}\chi(M^{P,G}_{H,ss}(v)).$$
 \end{proof}

\begin{cor}\label{cor_vw_twisted_optimal}
Let $\SS\to S$ be an optimal $\mu_2$-gerbe over $S$.
Let $v=(2,0,-\frac{k}{2})$ be a Mukai vector for $H^*(S,\qq)$. Then 
$$\vw^{\tw}_{v}(\SS)=\frac{1}{2}\chi(\Hilb^{k+1}(S)).$$
\end{cor}
\begin{proof}
From \cite[Theorem 3.16]{Yoshioka2}, the moduli space $M^{P,G}_{H,ss}(v)$ of $G$-twisted stable sheaves on $S$ with Mukai vector $v=(2,0,-b=-\frac{k}{2})$ is deformation equivalent to 
the moduli space $M^{S,\sO_S}_{S/\kappa}(v^\prime)$ of $G$-twisted semistable sheaves on $S$ with Mukai vector 
$v^\prime=(2, D, c)$ such that 
$$2\cdot 2 b=D^2-2\cdot 2 c.$$
And the moduli space  $M^{S,\sO_S}_{S/\kappa}(v^\prime)$ is isomorphic to the Hilbert scheme
$\Hilb^{\frac{1}{2} (v^\prime)^2 +1}(S)$, where $\frac{1}{2} (v^\prime)^2 +1=2b+1$.

In order to get all the positive integers $2b+1$,  we have to choose $D$'s such that 
$D^2=-2$ or $D^2=0$.  When $D^2=-2$, we get $2b=-1, 1, 3,\cdots $, all the odd integers; and 
when $D^2=0$, we get $2b=0,2,4,\cdots$, all the even integers.  From Proposition  \ref{prop_rank2_vw_optimal},
 $$\vw^{\tw}_{v}=\VW^{\tw}_{v}=\frac{1}{2}\chi(M^{P,G}_{H,ss}(v)).$$
 Thus the result follows.
\end{proof}

\subsection{Proof of S-duality formula in rank two}\label{subsec_proof_S_duality}

In this section we prove the S-duality conjecture Conjecture  (\ref{eqn_S_transformation_2}) for K3 surfaces  first in rank $2$, and then talk about the higher ranks. 
Let $S$ be a smooth projective K3 surface.  Following  Conjecture \ref{con_S_duality}, 
and fixing the data $(2,\sO, c_2)$, we consider the partition function
$$Z(S, \SU(2)/\zz_2; q)=\sum_{g\in H^2(S,\mu_2)}e^{\frac{2\pi i g\cdot \overline{\sO}}{2}}Z_{2, \sL_g}(\SS_g, q),$$
where 
$$Z_{2, \sL_g}(\SS_g, q)=\sum_{c_2}\vw^{\tw}_{2,\sL_g, c_2}(\SS_g)q^{c_2}$$
is the partition function of the twisted Vafa-Witten invariants for the $\mu_2$-gerbe $\SS_g$ corresponding to $g$. 
Note that $e^{\frac{2\pi i g\cdot \overline{\sO}}{2}}=1$.
We let 
$$
\begin{cases}
\sL_g \text{~is the corresponding defining line bundle for~} \SS_g  & \text{if~} \SS_g \text{~essentially trivial};\\
\sL_g=\sO    &  \text{if~} \SS_g \text{~is an optimal gerbe}.
\end{cases}
$$

\begin{thm}\label{thm_SU2Z2_partition_function}
Let $S$ be a smooth projective K3 surface with Picard number $\rho(S)$. Then 
$$Z(S, \SU(2)/\zz_2; q)=\frac{1}{4}q^2 G(q^2)+q^2 \left(2^{21}\cdot G(q^{\frac{1}{2}})+2^{\rho(S)-1}\cdot 
G(-q^{\frac{1}{2}})\right). $$
\end{thm}
\begin{proof}
Since the Picard number of $S$ is $\rho(S)$, there are $2^{\rho(S)}$ number of essentially trivial $\mu_2$-gerbes $\SS\to S$ over $S$.  Let 
$\SS_0\to S$ be the trivial $\mu_2$-gerbe.   Also $ |H^2(S,\mu_2)|=2^{22}$, and from Remark \ref{rmk_optimal_mu2}, there are totally 
$$2^{22}-2^{\rho(S)}$$
number of nontrivial order $2$ Brauer classes, i.e., $2^{22}-2^{\rho(S)}$  number of optimal $\mu_2$-gerbes.   

First for the trivial $\mu_2$-gerbe $\SS_0\to S$, from Proposition \ref{prop_trivial_mur_gerbe_K3} and calculations in  \cite[\S 5.1]{TT2}, we have:
\begin{align*}
Z_{2, \sO}(\SS_0, q)&=\sum_{k}\vw_{2,0,k}(\SS_0)q^k\\
&=\sum_{k}\chi(\Hilb^{4k-3}(S))q^{2k}+\sum_{k}\chi(\Hilb^{4k-1}(S))q^{2k+1}+\frac{1}{4}\sum_{k}\chi(\Hilb^{k}(S))q^{2k}\\
&=\frac{1}{4}q^2 G(q^2)+\frac{1}{2}q^2\left(G(q^{\frac{1}{2}})+G(-q^{\frac{1}{2}})\right).
\end{align*}

Let $\SS_g\to S$ be a non-trivial essentially trivial $\mu_2$-gerbe over $S$, which is given by the line bundle $\sL_g\in \Pic(S)$.  Then from Proposition \ref{prop_twisted_vw_essential_trivial}, we have:
\begin{align*}
Z_{2, \sL_g}(\SS_g, q)&=\sum_{k}\vw^{\tw}_{2,\sL_g,k}(\SS_g)q^k\\
&=\sum_{k}\chi(\Hilb^{4k-3}(S))q^{2k}+\sum_{k}\chi(\Hilb^{4k-1}(S))q^{2k+1}\\
&=\frac{1}{2}q^2\left(G(q^{\frac{1}{2}})+G(-q^{\frac{1}{2}})\right).
\end{align*}

Let $\SS_g\to S$ be an optimal $\mu_2$-gerbe over $S$.  Then from Proposition \ref{prop_rank2_vw_optimal} and Corollary \ref{cor_vw_twisted_optimal}, and note that  in the Mukai vector, 
$v=(2,0,-b=-\frac{k}{2})$, we have $-b=-c_2+2$ and here we understand that $c_2\in \frac{1}{2}\zz$,
and the extra $2$ comes form the fundamental class $2\cdot \omega$, 
we have 
$$c_2=b+2$$ 
for $k\in\zz$. Thus we calculate 
\begin{align*}
Z_{2, \sO}(\SS_g, q)&=\sum_{c_2}\vw^{\tw}_{2,\sO,c_2}(\SS_g)q^{c_2}\\
&=\sum_{c_2}\vw^{\tw}_{v}(\SS_g)q^{c_2}\\
&=\frac{1}{2}\sum_{\frac{k}{2}: k\in\zz_{\geq -1}}\chi(\Hilb^{k+1}(S))q^{\frac{k}{2}}\cdot q^2\\
&=\frac{1}{2}\sum_{\frac{k}{2}: k\in\zz_{\geq 0}}\chi(\Hilb^{k}(S))q^{\frac{k}{2}}\cdot q^{\frac{3}{2}}\\
&=\frac{1}{2}q^2 G(q^{\frac{1}{2}}).
\end{align*}
Therefore 
\begin{align*}
&Z(S, \SU(2)/\zz_2; q) \\
&=\sum_{g\in H^2(S,\mu_2)}Z_{2, \sL_g}(\SS_g, q)\\
&=\frac{1}{4}q^2 G(q^2)+2^{\rho(S)}\cdot \frac{1}{2}q^2\left(G(q^{\frac{1}{2}})+G(-q^{\frac{1}{2}})\right)+(2^{22}-2^{\rho(S)})\cdot \frac{1}{2}q^2 G(q^{\frac{1}{2}})\\
&=\frac{1}{4}q^2 G(q^2)+q^2\left(2^{21} G(q^{\frac{1}{2}})+2^{\rho(S)-1} G(-q^{\frac{1}{2}})\right).
\end{align*}
\end{proof}

\begin{cor}\label{cor_S-duality_K3}
Let $S$ be a smooth projective K3 surface with Picard number $\rho(S)=11$.  Then 
$$Z(S, \SU(2)/\zz_2; q)=\frac{1}{4}q^2 G(q^2)+q^2 \left(2^{21}\cdot G(q^{\frac{1}{2}})+2^{10}\cdot G(-q^{\frac{1}{2}})\right).$$
\end{cor}
\begin{rmk}
The formula in Corollary \ref{cor_S-duality_K3} is predicted by Vafa-Witten in \cite[Formula (4.18)]{VW} for any complex K3 surface $S$.  But our result Corollary \ref{cor_S-duality_K3} depends on the complex structure of the K3 surface $S$ (it depends on the Picard number).  The reason is that we sum all the $\mu_2$-gerbes $\SS\to S$ and the type of $\mu_2$-gerbes on $S$ depends on the Picard number. 
\end{rmk}

Recall that a K3 surface $S$ over characteristic $p>0$ is called {\em supersingular} if its Picard rank is $\rho(S)=22$.  Then if all the arguments in this paper work for positive characteristic, we get the following result:

\begin{cor}\label{cor_S-duality_K3_super}
Let $S$ be a supersingular K3 surface over positive characteristic $p>0$.  Then if all the results in the above section work over positive characteristic, we have  
$$Z(S, \SU(2)/\zz_2; q)=\frac{1}{4}q^2 G(q^2)+q^2 \left(2^{21}\cdot G(q^{\frac{1}{2}})+2^{21}\cdot G(-q^{\frac{1}{2}})\right).$$
\end{cor}

\begin{rmk}
We define $\vw^{\SU(2)/\zz_2}_{\alpha}(S)$ to be the Vafa-Witten invariants for the $\SU(2)/\zz_2$-theory for the surface $S$.  In (\ref{eqn_partition_function_SU2/Z2_maple}) we calculate some invariants 
$$\vw^{\SU(2)/\zz_2}_{2,0}(S)=\frac{1}{4}, \quad  \vw^{\SU(2)/\zz_2}_{2,\frac{3}{2}}(S)=2096128.$$
We calculate the coefficient of $q^2$.  It is 
$$(2^{11} \times 24)+(2^{21}-2^{10})\times 24+6=50356224+6=50356230.$$
Let us explian how to get this invariant. 
Let $\SS_0\to S$ be a trivial $\mu_2$-gerbe, then from \cite[\S 5]{TT2},
$$\vw^{\tw}_{2, 0, 2}(\SS_0)=\frac{1}{4}\chi(\Hilb^1(S))+\chi(\Hilb^1(S))=6+24=30.$$
Let $\SS_g\to S$ be a nontrivial essentially trivial $\mu_2$-gerbe, then 
$$\vw^{\tw}_{2, \sL_g, 2}(\SS_g)=\chi(\Hilb^1(S))=24.$$
Let $\SS_g\to S$ be an optimal  $\mu_2$-gerbe, then 
$$\vw^{\tw}_{2, 0, 2}(\SS_g)=\frac{1}{2}\cdot \chi(\Hilb^1(S))=\frac{1}{2}\cdot 24.$$
Then we have 
$$\vw^{\SU(2)/\zz_2}_{2,0, 2}(S)=6+(2^{11}\times 24)+\frac{1}{2}(2^{22}-2^{11})\times 24=50356230.$$

From the calculation in Theorem \ref{thm_SU2Z2_partition_function}, all the fractional degrees of $q$ come form the twisted Vafa-Witten  invariants of optimal $\mu_2$-gerbes 
$\SS_g\to S$.   We calculated  $\vw^{\SU(2)/\zz_2}_{2,\frac{3}{2}}(S)$ and 
$$ \vw^{\SU(2)/\zz_2}_{2,\frac{5}{2}}(S)=\frac{1}{2}(2^{22}-2^{11})\times 324=679145472$$
where $\chi(\Hilb^2(S))=324$ is the Euler number of the Hilbert scheme of $2$-points on $S$ and 
$$ \vw^{\SU(2)/\zz_2}_{2,\frac{7}{2}}(S)=\frac{1}{2}(2^{22}-2^{11})\times 25650=53765683200$$
where $\chi(\Hilb^4(S))=25650$ is the Euler number of the Hilbert scheme of $4$-points on $S$.
These numbers match the expansion in (\ref{eqn_partition_function_SU2/Z2_maple}). 
\end{rmk}

\subsection{Vafa-Witten's S-duality conjecture for K3 surfaces}\label{sub_VW_con_K3}

Vafa-Witten  actually predicted the formula  in \cite[Formula 4.18]{VW} for any complex K3 surface $S$, which does not depend on the complex structure of $S$.    We modify Conjecture \ref{con_S_duality} a bit for K3 surfaces  and show that trivial and optimal $\mu_2$-gerbes on $S$ give  Vafa-Witten's formula.  We follow the notations in \S \ref{subsec_SUr/Zr}.

We fix a smooth projective K3 surface $S$.  The $\mu_2$-gerbes $\SS\to S$ are classified by 
$H^2(S,\mu_2)$.  Let $\SS_0\to S$ be the trivial $\mu_2$-gerbe;  $\SS_{\ess}\to S$ be a nontrivial essentially  trivial 
$\mu_2$-gerbe corresponding to a  line bundle $\sL_g$; and $\SS_{\opt}\to S$ be an optimal 
$\mu_2$-gerbe.

Let us first give a more detailed explanation of Vafa-Witten \cite[Page 55]{VW}.  Vafa-Witten actually sum over the topological data for the  $\PGL_2$-bundles or Higgs sheaves $(E,\phi)$ with first Chern class $g=c_1(E)\in H^2(S,\zz)/2\cdot H^2(S,\zz)$.  From Page 55 in \cite{VW}, the square $g^2$ satisfies $g^2=c_1(E)^2\equiv 0, \text{or~} 2\mod 4$.  They call 
$g$ {\em even} if $g^2\equiv 0\mod 4$; and  {\em odd} if $g^2\equiv 2\mod 4$.  There are three type of partition functions on a K3 surface $S$ corresponding to $g=0$, even but not zero and odd.  Let $n_0, n_{\even}, n_{\odd}$ be the number of values of $g$ that are respectively, trivial, even but nontrivial, and odd.  Then 
$$
\begin{cases}
n_0=1;\\
n_{\even}=\frac{2^{22}+2^{11}}{2}-1;\\
n_{\odd}=\frac{2^{22}-2^{11}}{2}.
\end{cases}
$$
From our formula in Corollary \ref{cor_S-duality_K3} we modify the calculation: 
\begin{align}\label{eqn_K3_formula_key}
&Z(S, \SU(2)/\zz_2; q)=
Z_{2,\sO}(\SS_0, q)+(2^{11}-1)Z_{r, \sL_g}(\SS_{\ess}, q)+
(2^{22}-2^{11})Z_{2, \sO}(\SS_{\opt}, q) \\ \nonumber
&=\frac{1}{4}q^2 G(q^2)+ \frac{1}{2}q^2\left(G(q^{\frac{1}{2}})+G(-q^{\frac{1}{2}})\right)+
(2^{11}-1)\frac{1}{2}q^2\left(G(q^{\frac{1}{2}})+G(-q^{\frac{1}{2}})\right)\\  \nonumber
&+(2^{22}-2^{11})\frac{1}{2}q^2\left(G(q^{\frac{1}{2}})\right)\\  \nonumber
&=\Big[\frac{1}{4}q^2 G(q^2)+ \frac{1}{2}q^2\left(G(q^{\frac{1}{2}})+G(-q^{\frac{1}{2}})\right)\Big]+
\left((2^{11}-1)+\frac{2^{22}-2^{11}}{2}\right)\frac{1}{2}q^2\left(G(q^{\frac{1}{2}})+G(-q^{\frac{1}{2}})\right) \\  \nonumber
&+(2^{22}-2^{11})\frac{1}{2}q^2\left(G(q^{\frac{1}{2}})\right)-
\left(\frac{2^{22}-2^{11}}{2}\right)\frac{1}{2}q^2\left(G(q^{\frac{1}{2}})+G(-q^{\frac{1}{2}})\right)\\  \nonumber
&=\Big[\frac{1}{4}q^2 G(q^2)+ \frac{1}{2}q^2\left(G(q^{\frac{1}{2}})+G(-q^{\frac{1}{2}})\right)\Big]+
\left(\frac{2^{22}+2^{11}}{2}-1\right)\frac{1}{2}q^2\left(G(q^{\frac{1}{2}})+G(-q^{\frac{1}{2}})\right) \\  \nonumber
&+\left(\frac{2^{22}-2^{11}}{2}\right)\frac{1}{2}q^2\left(G(q^{\frac{1}{2}})-G(-q^{\frac{1}{2}})\right)\\  \nonumber
&=Z_0(q)+ n_{\even}Z_{\ess}(q)+n_{\odd}Z_{\odd}(q)
\end{align}
where we let 
$$Z_0(q)=\frac{1}{4}q^2 G(q^2)+ \frac{1}{2}q^2\left(G(q^{\frac{1}{2}})+G(-q^{\frac{1}{2}})\right);$$
$$Z_{\ess}(q)=\frac{1}{2}q^2\left(G(q^{\frac{1}{2}})+G(-q^{\frac{1}{2}})\right)$$ and 
$$Z_{\odd}(q)=\frac{1}{2}q^2\left(G(q^{\frac{1}{2}})-G(-q^{\frac{1}{2}})\right)$$
which are the formula in \cite[Formula (4.11), (4.12)]{VW}. 

From Proposition \ref{prop_trivial_mur_gerbe_K3}  and Proposition \ref{prop_twisted_vw_essential_trivial}, we have $Z_{2,\sO}(\SS_0, q)=Z_0(q)$ and $Z_{2, \sL_g}(\SS_{\ess}, q)=Z_{\ess}(q)$.  
For optimal $\mu_2$-gerbes 
$\SS_{\opt}\to S$, note that for a Mukai vector $v=(2,0,-b=-\frac{k}{2})$, the second Chern classes $c_2=2+b$ take values in half integers, we have the following result 
\begin{lem}\label{lem_optimal_mu2}
Define 
$$Z_{2,\sO}(\SS_{\opt}, (-1)^2\cdot q):=\sum_{c_2}\vw^{\tw}_{v}(\SS_{\opt})(-1)^{2c_2}q^{c_2}.$$
 Then we have 
$$Z_{2,\sO}(\SS_{\opt}, (-1)^2\cdot q)=\frac{1}{2}q^2 G(-q^{\frac{1}{2}}).$$
\end{lem}
\begin{proof}
From the calculation in the proof of Theorem \ref{thm_SU2Z2_partition_function},
\begin{align*}
Z_{2, \sO}(\SS_{\opt}, (-1)^2\cdot q)&=\sum_{c_2}\vw^{\tw}_{2,\sO,c_2}(\SS_{\opt})(-1)^{2c_2}q^{c_2}\\
&=\frac{1}{2}\sum_{\frac{k}{2}: k\in\zz_{\geq -1}}\chi(\Hilb^{k+1}(S))(-1)^{k+4}\cdot q^{\frac{k}{2}}\cdot q^2\\
&=\frac{1}{2}\sum_{\frac{k}{2}: k\in\zz_{\geq 0}}\chi(\Hilb^{k}(S))(-1)^{k+3}\cdot q^{\frac{k}{2}}\cdot q^{\frac{3}{2}}\\
&=\frac{1}{2}q^2 G(-q^{\frac{1}{2}}).
\end{align*}
\end{proof}

There are two types of non-zero even classes $g\in H^2(S, \mu_2)$. The first type consists of classes $g$ that are algebraic classes (i.e. $(1,1)$-classes) and the number of these classes is $n_{\even}^1:=2^{\rho(S)}-1$. The second type consists of classes $g$ that are non-algebraic classes and the number of these classes is $n^2_{\even}:=n_{\even}-(2^{\rho(S)}-1)$.  

\begin{defn}\label{defn_even_odd_partition_function}
For the type of even classes $g$, we define
$$
 Z_{\even}(q):=Z_{2,\sO}(\SS_{\opt}, q)+Z_{2,\sO}(\SS_{\opt}, (-1)^2 q)=\frac{1}{2}q^2 \left( G(q^{\frac{1}{2}})+G(-q^{\frac{1}{2}})\right).$$

The partition function $Z_{\odd}(q)$ is obtained from optimal $\mu_2$-gerbes by:
$$Z_{\odd}(q)=\frac{1}{2}q^2 \left( G(q^{\frac{1}{2}})-G(-q^{\frac{1}{2}})\right)
=Z_{2,\sO}(\SS_{\opt}, q)+(-1)^{2\pi i\frac{1}{2}}\cdot Z_{2,\sO}(\SS_{\opt}, (-1)^2 q).$$
\end{defn}

Let $\sL_g\in\Pic(S)$ be a nontrivial line bundle such that it defines a nontrivial essentially trivial $\mu_2$-gerbe on $S$. 
Then $c_1(\sL)=g \mod 2$ and $g\neq 0$.  We have
\begin{prop}\label{prop_essentially_trivial_Tanaka-Thomas}
Let $Z_{2, \sL_g}(S, q)$ be the generating function of Tanaka-Thomas Vafa-Witten invariants for $S$ with rank two and with determinant $\sL_g$.  Then 
$$Z_{2, \sL_g}(S, q)=Z_{2, p^*\sL_g}(\SS_{\ess}, q)=Z_{\ess}(q)= Z_{\even}(q),$$
where $p: \SS_{\ess}\to S$ is the map to its coarse moduli space.  Moreover the formula $Z_{2, \sL_g}(S, q)$ only depends on the first Chern class $c_1\in H^2(S, \zz)$ such that $0\neq g=c_1\mod 2$.
\end{prop}
\begin{proof}
The proof of Proposition \ref{prop_rank_2_splitting} and Proposition \ref{prop_twisted_vw_essential_trivial} hold for semistable Higgs shaves $(E,\phi)$ on the K3 surface $S$ such that 
$\det(E)=\sL_g$. Thus we get the same formula for the generating function in the calculation in the proof of Theorem \ref{thm_SU2Z2_partition_function}. The last equality is from definition. 
\end{proof}

\begin{rmk}\label{rmk_algebraic_class_optimal-1}
Let  $L\in \Pic(S)$ be a line bundle such that its image in $H^2(S,\mu_2)$ is nonzero. 
For the optimal $\mu_2$-gerbe $\SS_{\opt}$ over $S$, 
from \cite[Theorem 3.16]{Yoshioka2}, the moduli stack $\sM^{\tw}_{(2,L,c_2)}(\SS_{\opt})$ of stable twisted sheaves with fixed determinant $L\in\Pic(S)$ is  isomorphic  to the moduli stack $\sM^{\tw}_{(2,\sO,c^\prime_2)}(\SS_{\opt})$ with trivial determinant for suitable $c_2^\prime$.  Thus their corresponding stable twisted Higgs sheaves are also deformation equivalent, which shows that Definition \ref{defn_even_odd_partition_function} is consistent for the two type of  even classes. 
\end{rmk}

\begin{rmk}\label{rmk_algebraic_class_optimal-2}
From Proposition \ref{prop_twisted_vw_essential_trivial}, for a nontrivial essentially trivial $\mu_2$-gerbe 
$\SS_{\ess}\to S$ corresponding to a line bundle $\sL$, we have 
\begin{align*}
Z_{2, \sL}(\SS_{\ess}, q)&=\sum_{c_2}\VW^{\tw}_{2,\sL, c_2}(\SS_{\ess})q^{c_2}\\
&=\sum_{k\geq 1}\chi(\Hilb^{4k-3}(S))q^{2k}+\sum_{k\geq 1}\chi(\Hilb^{4k-1}(S))q^{2k+1}\\
&=\frac{1}{2}q^2 \left( G(q^{\frac{1}{2}})+G(-q^{\frac{1}{2}})\right)
\end{align*}
From the proof of Theorem \ref{thm_SU2Z2_partition_function}, and Lemma \ref{lem_optimal_mu2}, the above formula is equal to 
\begin{align*}
&Z_{2, \sO}(\SS_{\opt},  q)+Z_{2, \sO}(\SS_{\opt}, (-1)^2\cdot q)\\
&=\frac{1}{2}\sum_{\frac{k}{2}: k\in\zz_{\geq 0}}\chi(\Hilb^{k}(S)) q^{\frac{k}{2}}\cdot q^{\frac{3}{2}}+
\frac{1}{2}\sum_{\frac{k}{2}: k\in\zz_{\geq 0}}\chi(\Hilb^{k}(S))(-1)^{k+3}\cdot q^{\frac{k}{2}}\cdot q^{\frac{3}{2}}\\
&=\frac{1}{2}q^2 \left(G(q^{\frac{1}{2}})+q^2 G(-q^{\frac{1}{2}})\right).
\end{align*}
Therefore the moduli stack we choose for even classes does not depend on the algebraic or non-algebraic classes, which is all deformation equivalent to the Hilbert scheme $\Hilb^{k}(S)$ of points on $S$. 
\end{rmk}

\begin{thm}\label{thm_SU2/Z2_Vafa-Witten_K3}
Let $S$ be a smooth complex K3 surface, and  define
\begin{align*}
Z^{\prime}_{2,\sO}(S, \SU(2)/\zz_2;q):=Z_0(q)+ n_{\even}Z_{\even}(q)+
n_{\odd}Z_{\odd}(q).
\end{align*}
We call it the partition function of $SU(2)/\zz_2$-Vafa-Witten invariants, and we have:
$$Z^{\prime}(S, \SU(2)/\zz_2; q)=\frac{1}{4}q^2 G(q^2)+q^2 \left(2^{21}\cdot G(q^{\frac{1}{2}})+2^{10}\cdot G(-q^{\frac{1}{2}})\right).$$
This proves the prediction of Vafa-Witten in \cite[\S 4]{VW} for the gauge group $\SU(2)/\zz_2$ and the S-duality conjecture (\ref{eqn_S_transformation_2}).
\end{thm}
\begin{proof}
The proof of the formula $Z^{\prime}(S, \SU(2)/\zz_2; q)$ is given in (\ref{eqn_K3_formula_key}), since 
$$Z_{\even}(q)=\frac{1}{2}q^2 \left( G(q^{\frac{1}{2}})+G(-q^{\frac{1}{2}})\right).$$  Since for a complex K3 surface $S$, the Picard number $\rho(S)$ satisfies $0\leq \rho(S)\leq 20$.  
We have the inequality  $n_{\even}> 2^{\rho(S)}=n_{\even}^{1}+1$.  
Therefore 
there definitely exist optimal $\mu_2$-gerbes on $S$ such that their corresponding classes 
$g\in H^2(S,\mu_2)$ have two parts. One part corresponds to even classes $g$ that are not algebraic,  and the other corresponds to odd classes.   Then from Definition \ref{defn_even_odd_partition_function} and Remark \ref{rmk_algebraic_class_optimal-1} it is not hard to see that the formula $Z^{\prime}(S, \SU(2)/\zz_2; q)$ is independent of the complex structure on $S$. 
\end{proof}

\begin{rmk}
For the even class $g\in H^2(S,\zz_2)$ such that  $g\neq 0$, we define 
$$Z^{\prime}_{2,\even}(S, \SU(2)/\zz_2;q):=Z_0(q)+ (2^{10}-1)Z_{\ess}(q)-2^{10}Z_{\odd}(q).$$
For odd classes $g\in H^2(S,\zz_2)$, we define 
$$Z^{\prime}_{2,\odd}(S, \SU(2)/\zz_2;q):=Z_0(q)+ (-2^{10}-1)Z_{\ess}(q)+2^{10}Z_{\odd}(q).$$
Then 
from \cite[(4.14), (4.15)]{VW}  under the $S$-transformation $\tau\mapsto -\frac{1}{\tau}$ we have:
$$Z_{\even}(q)\mapsto 2^{-11}\left(\frac{\tau}{i}\right)^{w/2}Z^{\prime}_{2,\even}(S, \SU(2)/\zz_2;q);$$
and 
$$Z_{\odd}(q)\mapsto 2^{-11}\left(\frac{\tau}{i}\right)^{w/2}Z^{\prime}_{2,\odd}(S, \SU(2)/\zz_2;q).$$
\end{rmk}
\begin{rmk}
If the K3 surface $S$ is projective, then the Picard number satisfies the condition $1\leq \rho(S)\leq 20$. 
In this case we can define $Z_{\ess}(q)=Z_{2, p^*\sL_g}(\SS_{\ess},q)$ just using one nontrivial  essentially trivial 
$\mu_2$-gerbe $\SS_{\ess}\to S$, since there exists a nontrivial line bundle $\sL_g\in \Pic(S)$ such that 
$c_1(\sL_g)\equiv g\mod 2$ and $0\neq g\in H^2(S,\mu_2)$. 
\end{rmk}

\subsection{Discussion on higher rank case}\label{subsec_higher_rank_S-duality}

Let $S$ be a smooth K3 surface with Picard number $\rho(S)$.
Let us first fix to the case the rank $r$ is a prime integer.  Then like $r=2$, we have similar results for essentially trivial $\mu_r$-gerbes $\SS\to S$.
For the trivial $\mu_r$-gerbe $\SS=[S/\mu_r]\to S$, the twisted Vafa-Witten invariants are the same as the Vafa-Witten invariants in \cite{TT2}. 

For non-trivial essentially trivial $\mu_r$-gerbes
$\SS\to S$, we have:

\begin{prop}\label{prop_rank_r_splitting}
Let $(E,\phi)$ be a $\SS$-twisted semistable rank $r$ Higgs sheaf with class $(r, \sL, c_2)$.  Then the torsion free sheaf $E$ can not split into sum of subbundles. 
\end{prop}
\begin{proof}
The proof is similar to Proposition \ref{prop_rank_2_splitting}. 
The reason is that if  $E$ can be split into 
$E=E_1\oplus E_2$ such that each $E_i$ is of higher rank less than $r$,  
the determinant is $\det(E)=\sL$, which 
implies that $\det(E)=\det(E_1\oplus E_2)$.  Either $E_1$ or $E_2$ is nontrivial, and 
$\det(E_i)=\sL^{\frac{s}{r}}$ for some $1\leq s\leq r$.  The basic reason is that $r$ do not have nontrivial split.  Thus   
$c_1(E)=c_1(E_1)+c_1(E_2)$ and 
after modulo $\mod r$ we get that $c_1(E)\equiv 0 \mod r$, which is a contradiction. 
\end{proof}

The following is a similar result as in Proposition \ref{prop_twisted_vw_essential_trivial}, we omit the proof. 

\begin{prop}\label{prop_twisted_vw_essential_trivial_higher_rank}
Let $\SS\to S$ be a non-trivial essentially trivial $\mu_r$-gerbe over the K3 surface $S$ corresponding to the line bundle $\sL\in\Pic(S)$. Then we have 
$$
\begin{cases}
\vw^{\tw}_{r,\sL, rk}(\SS)=\chi(\Hilb^{r^2\cdot k-(r^2-1)}(S));\\ 
\vw^{\tw}_{r,\sL, rk+1}(\SS)=\chi(\Hilb^{r^2\cdot k-(r^2-r)}(S));\\
\vdots \quad\quad \quad  \vdots  \quad \quad \quad  \vdots \quad\quad\quad \vdots\\
\vw^{\tw}_{r,\sL, rk+(r-1)}(\SS)=\chi(\Hilb^{r^2\cdot k-1}(S)),
\end{cases}
$$
where $\chi(\Hilb^{n}(S))$ is the topological Euler number of the Hilbert scheme of $n$ points on $S$.
\end{prop}

Since $r$ is a prime number, $H^2(S,\mu_r)=\zz_r^{22}$,  all other non essentially trivial 
$\mu_r$-gerbes on $S$ is an optimal $\mu_r$-gerbe. 
Let $\SS\to S$ be an optimal $\mu_r$-gerbe over $S$.  All of the arguments in \S \ref{subsec_VW_twisted_optimal} works for this $\mu_r$-gerbe $\SS$, since we only count rank $r$ $\SS$-twisted Higgs sheaves, and the minimal rank of locally free $\SS$-twisted  sheaf is $r$, hence must be stable. Especially we have a similar result comparing with Corollary \ref{cor_vw_twisted_optimal}:
\begin{cor}\label{cor_vw_twisted_optimal_higher_rank}
Let $\SS\to S$ be an optimal $\mu_r$-gerbe over $S$ for a prime number $r$.
Let $v=(r,0,-\frac{k}{r})$ be a Mukai vector for $H^*(S,\qq)$. Then 
$$\vw^{\tw}_{v}(\SS)=\frac{1}{r}\chi(\Hilb^{k}(S)).$$
\end{cor}

Combining with Proposition \ref{prop_twisted_vw_essential_trivial_higher_rank} and Proposition \ref{cor_vw_twisted_optimal_higher_rank}, we have the following result:

\begin{thm}\label{thm_SU2Z2_partition_function_higher_prime_rank}
Let $S$ be a smooth projective K3 surface with Picard number $\rho(S)$. Then 
$$Z(S, \SU(r)/\zz_r; q)=\frac{1}{r^2}q^r G(q^r)+q^r\left(r^{21} G(q^{\frac{1}{r}})+r^{\rho(S)-1} \left(\sum_{j=1}^{r-1}G\left(e^{\frac{2\pi ij}{r}}q^{\frac{1}{r}}\right)\right) \right). $$
\end{thm}
\begin{proof}
We do similar calculation as in Theorem \ref{thm_SU2Z2_partition_function}. 
The Picard number of $S$ is $\rho(S)$, there are $r^{\rho(S)}$ number of essentially trivial $\mu_r$-gerbes $\SS\to S$ over $S$.  Let 
$\SS_0\to S$ be the trivial $\mu_2$-gerbe.   Also $ |H^2(S,\mu_r)|=r^{22}$, and  there are totally 
$$r^{22}-r^{\rho(S)}$$
number of nontrivial order $r$ Brauer classes, i.e., these number of optimal $\mu_r$-gerbes.   

For the trivial $\mu_r$-gerbe $\SS_0\to S$, from Proposition \ref{prop_trivial_mur_gerbe_K3} and calculations in  \cite[\S 5.1]{TT2}, we have:
\begin{align*}
Z_{r, \sO}(\SS_0, q)&=\sum_{k}\vw_{r,0,k}(\SS_0)q^k\\
&=\frac{1}{r^2}q^r G(q^r)+\frac{1}{r}q^r\left(\sum_{j=0}^{r-1}G\left(e^{\frac{2\pi ij}{r}}q^{\frac{1}{r}}\right)\right).
\end{align*}
And this is the result in \cite[\S 5]{TT2}.

Let $\SS_g\to S$ be a non-trivial essentially trivial $\mu_r$-gerbe over $S$, which is given by the line bundle $\sL_g\in \Pic(S)$.  Then from Proposition \ref{prop_twisted_vw_essential_trivial_higher_rank}, we have:
\begin{align*}
Z_{r, \sL_g}(\SS_g, q)&=\sum_{k}\vw^{\tw}_{r,\sL_g,k}(\SS_g)q^k=\sum_{k}\left(\sum_{j=0}^{r-1}\vw^{\tw}_{r,\sL_g,rk+j}(\SS_g)\right)q^k
=\frac{1}{r}q^r\left(\sum_{j=0}^{r-1}G\left(e^{\frac{2\pi ij}{r}}q^{\frac{1}{r}}\right)\right).
\end{align*}

Let $\SS_g\to S$ be an optimal $\mu_r$-gerbe over $S$.  Then from  Corollary \ref{cor_vw_twisted_optimal_higher_rank}, and in the Mukai vector, 
$v=(r,0,-b=-\frac{k}{r})$, we have $-b=-c_2+r$ and here we understand that $c_2\in \frac{1}{r}\zz$,
here the extra $r$ comes form the fundamental class $r\cdot \omega$, 
we have 
$$c_2=b+r$$ 
for $k\in\zz_{\geq 0}$. Thus we calculate:
\begin{align*}
Z_{r, \sO}(\SS_g, q)&=
\sum_{c_2}\vw^{\tw}_{r,\sO,c_2}(\SS_g)q^{c_2}=\sum_{c_2}\vw^{\tw}_{v}(\SS_g)q^{c_2}\\
&=\frac{1}{r}\sum_{\frac{k}{r}: k\in\zz_{\geq -1}}\chi(\Hilb^{k+1}(S))q^{\frac{k}{r}}\cdot q^r\\
&=\frac{1}{r}\sum_{\frac{k}{r}: k\in\zz_{\geq 0}}\chi(\Hilb^{k}(S))q^{\frac{k}{r}}\cdot q^{\frac{r^2-1}{r}}\\
&=\frac{1}{r}q^r G(q^{\frac{1}{r}}).
\end{align*}
Therefore 
\begin{align*}
Z(S, \SU(r)/\zz_r; q)&=\sum_{g\in H^2(S,\mu_r)}Z_{r, \sL_g}(\SS_g, q)\\
&=\frac{1}{r^2}q^r G(q^r)+r^{\rho(S)}\cdot \frac{1}{r}q^r\left(\sum_{j=0}^{r-1}G\left(e^{\frac{2\pi ij}{r}}q^{\frac{1}{r}}\right)\right)+(r^{22}-r^{\rho(S)})\cdot \frac{1}{r}q^r G(q^{\frac{1}{r}})\\
&=\frac{1}{r^2}q^r G(q^r)+q^r\left(r^{21} G(q^{\frac{1}{r}})+r^{\rho(S)-1} \left(\sum_{j=1}^{r-1}G\left(e^{\frac{2\pi ij}{r}}q^{\frac{1}{r}}\right)\right) \right).
\end{align*}
\end{proof}

We also provide a proof of Vafa-Witten's conjecture for prime ranks. 
Comparing with Lemma \ref{lem_optimal_mu2}, for any optimal $\mu_r$-gerbe $\SS_{\opt}\to S$, we have 
\begin{lem}\label{lem_optimal_mur}
For any $r$-th root of unity $e^{2\pi i \frac{m}{r}} (0\leq m< r-1)$, define 
$$Z_{r,\sO}(\SS_{\opt}, (e^{\frac{2\pi i m}{r}})^r\cdot q):=\sum_{c_2\in\frac{1}{r}\zz}\vw^{\tw}_{v}(\SS_{\opt})(e^{2\pi i \frac{m}{r}})^{rc_2}q^{c_2}.$$
 Then we have 
$$Z_{r,\sO}(\SS_{\opt}, (e^{\frac{2\pi i m}{r}})^r\cdot q)=\frac{1}{r}q^r G(e^{\frac{2\pi i m}{r}}q^{\frac{1}{r}}).$$
\end{lem}
\begin{proof}
From the calculation for optimal gerbes in Theorem \ref{thm_SU2Z2_partition_function_higher_prime_rank}, the proof is similar to Lemma \ref{lem_optimal_mu2}.
\end{proof}

Let $Z_{r,0}(q)$ denote the Tanaka-Thomas partition function for $S$ with rank $r$ and trivial determinant 
$\sO$. 
\begin{thm}\label{thm_SUr/Zr_Vafa-Witten_K3}
Let $S$ be a smooth complex K3 surface. For a prime positive integer $r$, we define
\begin{align*}
Z^{\prime}_{r,\sO}(S, \SU(r)/\zz_r;q):=Z_{r,0}(q)+\sum_{0\neq g\in H^2(S,\mu_r)}\sum_{m=0}^{r-1}e^{\pi i \frac{r-1}{r} m g^2}
Z_{r,\sO}(\SS_{\opt}, (e^{\frac{2\pi i m}{r}})^r\cdot q)
\end{align*}
We call it the partition function of $SU(r)/\zz_r$-Vafa-Witten invariants, and we have:
$$Z^{\prime}(S, \SU(r)/\zz_r; q)=\frac{1}{r^2}q^r G(q^r)+q^r\left(r^{21} G(q^{\frac{1}{r}})+r^{10} \left(\sum_{m=1}^{r-1}G\left(e^{\frac{2\pi i m}{r}}q^{\frac{1}{r}}\right)\right) \right).$$
This generalizes the prediction of Vafa-Witten in \cite[\S 4]{VW} to the gauge group $\SU(r)/\zz_r$ and proves  Formula 
(4.11) in \cite{LL}.
\end{thm}
\begin{proof}
First Tanaka-Thomas partition function is:
$$Z_{r,0}(q)=\frac{1}{r^2}q^r G(q^r)+\frac{1}{r}q^r\left(\sum_{m=0}^{r-1}G\left(e^{\frac{2\pi i m}{r}}q^{\frac{1}{r}}\right)\right).$$
From Lemma \ref{lem_optimal_mur}, we always have  $Z_{r,\sO}(\SS_{\opt}, (e^{\frac{2\pi i m}{r}})^r\cdot q)=\frac{1}{r}q^r G(e^{\frac{2\pi i m}{r}}q^{\frac{1}{r}})$ for any $0\leq m< r-1$.  Now from \cite[Formula (A.15)]{LL}, and note that $S$ is a K3 surface, 
when $m=0$, in the sum of the definition of $Z^{\prime}_{r,\sO}(S, \SU(r)/\zz_r;q)$   there totally $r^{22}$ terms of  $\frac{1}{r}q^r G(q^{\frac{1}{q}})$, and when $m\neq 0$, 
$$\sum_{g\in H^2(S,\mu_r)}e^{\pi i \frac{r-1}{r} m g^2}=(\epsilon(m))^{22}r^{11}$$
where $\epsilon(m)=\left(\frac{m/2}{r}\right)$ if $m$ is even; and  $\epsilon(m)=\left(\frac{(m+r)/2}{r}\right)$ if $m$ is odd. The number$\left(\frac{a}{r}\right)$ is the Legendre symbols which is 
$1$ if $a(\mod r)$ is a perfect square, and $-1$ otherwise.  Thus $(\epsilon(m))^{22}=1$. Then 
the formula in the theorem follows from  \cite[Formula (A.15)]{LL}.
\end{proof}

\begin{rmk}
The result in Theorem \ref{thm_SUr/Zr_Vafa-Witten_K3} is also independent of the complex structure of $S$. 
For the general  higher rank case,  we need a multiple cover formula for the twisted Joyce-Song invariants introduced by Toda \cite{Toda_JDG} for $\mu_r$-gerbes, see \cite{Jiang_Tseng}. 
\end{rmk}



\subsection*{}

\end{document}